\documentclass{amsart}
\usepackage{amssymb}
\usepackage{latexsym}
\usepackage{amsmath}
\usepackage{amsfonts}
\usepackage{mathrsfs}
\usepackage[X2,T1]{fontenc}
\usepackage[applemac]{inputenc}
\usepackage{cite}
\usepackage{enumerate}
\usepackage{xcolor}

\makeatletter
\@namedef{subjclassname@2020}{\textup{2020} Mathematics Subject Classification}
\makeatother

\usepackage{calc}                   
\newtheorem{theorem}{Theorem}[section]
\newtheorem{lemma}[theorem]{Lemma}
\newtheorem{proposition}[theorem]{Proposition}
\newtheorem{corollary}[theorem]{Corollary}
\theoremstyle{definition}
\newtheorem{definition}[theorem]{Definition}

\newcommand{\rP}{\mathbb P}
\newcommand{\ro}{\mathbf R}
\newcommand{\co}{\mathbf C}
\newcommand{\qo}{\mathbf Q}
\newcommand{\no}{\mathbf N}
\newcommand{\nn}[1]{\mathbf N^{#1}}
\newcommand{\zo}{\mathbf Z}
\newcommand{\rr}[1]{\mathbf R^{#1}}
\newcommand{\sr}[1]{\mathbf S^{#1}}
\newcommand{\dd}{\mathrm {d}}
\newcommand{\leqs}{\leqslant}
\newcommand{\geqs}{\geqslant}
\newcommand{\sgn}{\operatorname{sgn}}
\newcommand{\rB}{\operatorname{B}}

\newcommand{\pdd}[2] {\partial_{#1} ^{#2}}
\newcommand{\ep}{\varepsilon}
\newcommand{\wpr}{{\text{\footnotesize $\#$}}}
\def\la{\langle}
\def\ra{\rangle}
\newcommand{\fy}{\varphi}
\newcommand{\supp}{\operatorname{supp}}
\newcommand{\wt}{\widetilde}

\newcommand{\eabs}[1]{\langle #1\rangle}
\newcommand{\charac}{\operatorname{char}}

\def\R{\mathbb R}

\def\1{\mathbb 1}
\def\cS{\mathscr{S}}
\def\cF{\mathscr{F}}
\def\cB{\mathcal{B}}
\newcommand{\cK}{\mathscr{K}}
\newcommand{\WFgs}{\mathrm{WF_{g}^{\it \sigma}}}
\newcommand{\WFg}{\mathrm{WF_{g}}}
\newcommand{\re}{{\rm Re \,}}

\def\la{\langle}
\def\ra{\rangle}

\theoremstyle{remark}
\newtheorem{remark}[theorem]{Remark}

\numberwithin{equation}{section}

\begin{document}

\title
[Propagation of singularities for Schr\"odinger equations]
{Propagation of singularities for anharmonic Schr\"odinger equations}

\author{Marco Cappiello}
\address{Dipartimento di Matematica, Universit\`a di Torino, Via Carlo Alberto 10, 10123 Torino, Italy}
\email{marco.cappiello[A]unito.it}

\author{Luigi Rodino}
\address{Dipartimento di Matematica, Universit\`a di Torino, Via Carlo Alberto 10, 10123 Torino, Italy}
\email{luigi.rodino[A]unito.it}
\author{Patrik Wahlberg}
\address{Dipartimento di Scienze Matematiche, Politecnico di Torino, Corso Duca degli Abruzzi 24, 10129 Torino, Italy}
\email{patrik.wahlberg[A]polito.it}






\begin{abstract} 
We consider evolution equations for two classes of generalized anharmonic oscillators and the associated initial value problem in the space of tempered distributions. 
We prove that the Cauchy problem is well posed in anisotropic Shubin--Sobolev modulation spaces of Hilbert type, 
and we investigate propagation of suitable notions of singularities.
\end{abstract}

\maketitle

\section{Introduction and statement of the results}\label{sec:intro}

The main goal of this paper is to study propagation of microlocal singularities for the Cauchy problem 
\begin{equation}\label{eq:anharmonicCP}
\begin{cases} 
\partial_t u + i \left( |x|^{2k} + \left( - \Delta \right)^m \right) u = 0 \\ 
u(0,x) = u_0(x) 
\end{cases}, \qquad (t,x) \in \ro \times \rr d, \quad k,m \in \no \setminus 0.
\end{equation}	
In Quantum Mechanics this is usually called Schr\"odinger anharmonic oscillator when $m = 1$. 
Under the addition of a quadratic term in the potential it can be regarded as a perturbation of the standard harmonic oscillator 
for small values of the potential.

We shall also consider more general problems than 
\eqref{eq:anharmonicCP}
of the form 
\begin{equation} \label{eq:SchrCP}
\begin{cases} 
\partial_t u + i A u=0 \\ 
u(0,x) = u_0(x) 
\end{cases}, \qquad (t,x) \in \ro \times \rr d, 
\end{equation}
where $A$ is a partial differential operator with polynomial coefficients of form
\begin{equation}\label{eq:diffoperator}
A = \sum\limits_{k|\alpha| +m|\beta| \leqs km} c_{\alpha\beta} x^\beta D_x^\alpha, \quad D^\alpha=(-i)^{|\alpha|}\partial_x^\alpha, 
\quad m,k \in 2 \no, \quad c_{\alpha\beta} \in \co. 
\end{equation}
As a generalization of \eqref{eq:anharmonicCP} in a different direction we consider finally \eqref{eq:SchrCP}
when $A$ has the Weyl symbol $a(x,\xi) = ( |x|^{2k} + |\xi|^{2m} )^p$
with $k,m \in \no \setminus 0$ and $p \in \ro \setminus 0$.

The literature on quantum anharmonic oscillators is enormous. The stationary equation $A u=0$ 
has been studied in thousands of titles, see for example \cite{Bambusi1, Bambusi2, BPV, Voros, Turbiner, CDR1, CDR2} and the references quoted therein. 
In this paper we refer in particular to \cite{CGPR, CGR2, Nicola1}. 
Helffer and Robert \cite{Helffer1} have developed Fourier integral operator parametrices to \eqref{eq:SchrCP}
when $A$ is an unbounded essentially self-adjoint elliptic operator
with a symbol that can be expanded in anisotropically homogeneous terms. 
Several contributions appeared recently, 
in particular for quadratic Hamiltonians, i.e. $k = m = 2$ in \eqref{eq:diffoperator}, 
see \cite{CaWa, CGNR, CNR, CRbook, Cordero1, Nakamura1, Nicola2, PravdaStarov1}. 

The first question for \eqref{eq:anharmonicCP} and \eqref{eq:SchrCP} is well-posedness in suitable function spaces. 
Natural choices are the Schwartz space $\cS(\rr d)$, the tempered distributions $\cS'(\rr d)$, 
and intermediate scales of global Sobolev type spaces. 
Such a scale of spaces can be defined using the short-time Fourier transform of a function or tempered distribution $u$:
\begin{equation*}
V_\varphi u(x,\xi) = (2\pi)^{-\frac{d}{2}} \int_{\rr d} e^{-i\langle y,\xi \rangle} u(y) \overline{\fy(y-x)}\, \dd y 
\end{equation*}
for a window function $\fy \in \cS(\rr d) \setminus 0$. 
We define $M_{\sigma,s}(\rr d)$ for anisotropy parameter $\sigma > 0$ and regularity parameter $s \in \ro$ 
as the space of all $u \in \cS'(\rr d)$ such that
\begin{equation*}
\| u \|_{M_{\sigma,s}} = \left( \iint_{\rr {2d}} |V_\fy u (x,\xi)|^2 \theta_\sigma(x,\xi)^{2s} \, \dd x \dd \xi \right)^{\frac12}  < \infty,
\end{equation*} 
where $\theta_\sigma(x,\xi) = 1+|x|+|\xi|^{\frac1\sigma}$ is an anisotropic weight on phase space $T^* \rr d$. 
These spaces are modulation Hilbert spaces and anisotropic versions of Shubin's spaces. 

In Section \ref{sec:wellposed} we prove the following well-posedness result for 
\eqref{eq:SchrCP} under the assumption that $A$ in \eqref{eq:diffoperator} is elliptic, in an anisotropic global sense, and positive. 
The involved function spaces are $M_{\sigma,s}(\rr d)$ with $\sigma= \frac{k}{m}$. 
Given a function or distribution space $X$ and an interval $I \subseteq \ro$, 
we denote here by $\cB^n (I, X)$ the space of all functions $f: I  \to X$ 
such that $f^{(j)} \in L^\infty(I, X) \cap C(I,X)$ for all $j = 0, 1, \dots, n$.

\begin{theorem}\label{thm:wptheorem1}
Let $A$ be a positive operator of the form \eqref{eq:diffoperator} such that 
\begin{equation}\label{eq:ellipticity}
a_k(x,\xi) := \sum\limits_{k|\alpha|+m|\beta|=km} c_{\alpha\beta}x^\beta\xi^\alpha
\neq 0 \quad \mbox{for} \quad (x,\xi)\neq (0,0).
\end{equation}
Then for every $s \in \ro$ and $u_0 \in M_{\sigma,s}(\rr d)$ there exists a unique solution $u= u(t,\cdot) \in \cB (\ro, M_{\sigma,s}(\rr d)) \cap \cB^1 (\ro, M_{\sigma,s - k}(\rr d))$ of \eqref{eq:SchrCP}. 
If $s = 0$ then $\| u(t,\cdot) \|_{L^2} = \| u_0 \|_{L^2}$ for all $t \in \ro$. 
\end{theorem}

When we emphasize the dependence of $u$ on $u_0$ we write $u = u(t,\cdot) = \cK_t u_0$
where $\cK_t$ is the solution operator or propagator. 
As a consequence of Theorem \ref{thm:wptheorem1} we get well-posedness in $\cS(\rr d)$ and in $\cS'(\rr d)$, see
Corollary \ref{cor:wellposedness}. 

The polynomial $a_k$ is the principal symbol of $A$ considered as an anisotropic Shubin symbol, 
and condition \eqref{eq:ellipticity} expresses an assumption of anisotropic ellipticity. 

Another functional setting we study for well-posedness is the Gelfand--Shilov spaces of Roumieu and Beurling type. 
Given $\mu >0,\nu>0$ with $\mu+\nu \geqs 1,$ ($\mu+\nu > 1$), a smooth function $f$ belongs to $\mathcal{S}_\nu^\mu(\rr d)$ (respectively $\Sigma_\nu^\mu(\rr d)$) if
\begin{equation*}
\sup_{x \in \rr d} |x^\alpha \partial^\beta f(x)| \leqs C h^{|\alpha+\beta|}\alpha!^\nu \beta!^\mu, \quad \alpha, \beta \in \nn d, 
\end{equation*}
for some $C > 0$ and $h > 0$
(respectively for every $h > 0$ and for some $C > 0$ depending on $h$).

Gelfand--Shilov spaces were introduced as subspaces of the Schwartz space for the study of partial differential equations  \cite{Gelfand1}. These functions are characterized by Gevrey or analytic regularity and exponential decay at infinity. 
One motivation for considering Gelfand--Shilov spaces comes from the stationary case 
of elliptic anharmonic oscillators
where they give more precise information 
than the Schwartz space, 
both on the regularity and on the behavior at infinity of the solutions. 

For instance, it is known \cite{CGR2} that under the assumption of anisotropic ellipticity \eqref{eq:ellipticity},
the solutions of the equation $Au=0$ with $A$ as in \eqref{eq:diffoperator} and assumed positive, belong to 
$\mathcal{S}_{\frac{m}{k+m}}^{\frac{k}{k+m}}(\rr d)$. The same conclusion holds for the eigenfunctions of $A$.

Concerning the evolution equation \eqref{eq:SchrCP} we prove in this paper well-posedness in
$\mathcal{S}_\nu^\mu(\rr d)$ and in $\Sigma_\nu^\mu(\rr d)$
when 
\begin{equation}\label{eq:muenu}
\mu= \frac{kr}{k+m}, \qquad \nu= \frac{mr}{k+m}, 
\end{equation}
with $r \geqs 1$ for $\mathcal{S}_\nu^\mu(\rr d)$, and $r >1$ for $\Sigma_\nu^\mu(\rr d)$.

\begin{theorem}\label{thm:wptheorem2}
Let $A$ be a positive differential operator of the form \eqref{eq:diffoperator} satisfying \eqref{eq:ellipticity}, 
and define $\mu$ and $\nu$ by \eqref{eq:muenu}. 
If $r \geqs 1$ then for every $u_0 \in \mathcal{S}_\nu^\mu(\rr d)$ the problem \eqref{eq:SchrCP} admits a unique solution $u \in \mathcal{B}^1(\ro, \mathcal{S}_\nu^\mu(\rr d))$.  
The same result holds replacing $\mathcal{S}_\nu^\mu(\rr d)$ by $\Sigma_\nu^\mu(\rr d)$ for $r >1$.
\end{theorem}

Beyond the results on well-posedness of evolution equations of the form \eqref{eq:SchrCP}, 
our main interest in this paper is how the solution propagates singularities. 
We follow H\"ormander's ideas of microlocal analysis, adopting suitable notions of wave front sets, and fix attention on large values of space and frequency variables. 
Concerning the history of microlocal analysis, 
the classical results for hyperbolic equations, cf. \cite{CHvol2}, were not able to determine the correct direction of propagation, among the different characteristic curves passing through a singularity point. 
H\"ormander \cite{Hormander1} finally solved the mystery by defining the wave front set: the microsingularities propagate exactly along the bicharacteristics passing through a point in the phase space. 

More recently H\"ormander's approach has been applied to Schr\"odinger equations, the microlocal propagation being determined by the Hamilton flow of the principal symbol of the Hamiltonian. 
Different variants of wave front sets were used, in particular for bounded potentials, cf. \cite{CKS, Hassell1, Ito1, Ito2, Wunsch1}. In the case of quadratic Hamiltonians the correct definition was already proposed by H\"ormander \cite{Hormander3}, afterwards reconsidered by several authors under different names and notations, see \cite{Nakamura1, RW1}. 

In our results on propagation of singularities for \eqref{eq:SchrCP}, the operator $A = a^w(x,D)$ will be the Weyl quantization of a symbol of the form 
\begin{equation}\label{eq:hamiltoniansymbol}
a(x,\xi) = \left( |x|^{2k} + |\xi|^{2m} \right)^p, \quad x, \xi \in \rr d, 
\end{equation}
where $k,m \in \no \setminus 0$ and $p \in \ro \setminus 0$. 
This generalizes the Hamiltonian in \eqref{eq:anharmonicCP} which equals \eqref{eq:hamiltoniansymbol} with $p = 1$. 

\begin{remark}\label{rem:cutoffsmooth}
If $p \in \no \setminus 0$ then $a \in C^\infty(\rr {2d})$. 
For $p \in \ro \setminus \no$ it is not smooth in the origin and then we multiply with a cutoff function around the origin
in order to always have $a \in C^\infty(\rr {2d})$. 
More precisely for $\delta > 0$ we define
$\psi_\delta(x,\xi) = \chi(|x|^2 + |\xi|^2) \in C^\infty(\rr {2d})$ where $\chi \in C^{\infty}(\ro)$,
$0 \leqs \chi \leqs 1$, $\chi(t) = 0$ for $t \leqs \frac{\delta^2}{4}$ and $\chi(t) = 1$ for $t \geqs \delta^2$
for a given $\delta > 0$. 
Thus 
$a(x,\xi) = \psi_\delta (x,\xi) \left( |x|^{2k} + |\xi|^{2m} \right)^p \in C^\infty(\rr {2d})$ if $p \in \ro \setminus \no$. 
\end{remark}


Inspired by H\"ormander's idea and the literature quoted above, we look for results for propagation of singularities along the Hamilton flow of the symbol $a$, 
that is solutions to 
\begin{equation} \label{HJintro}
\begin{cases} 
x'(t)  = \nabla_\xi a( x(t), \xi(t) ) \\ 
\xi'(t)  = -\nabla_x a( x(t), \xi(t) )  \\ 
x(0)=y  \\ 
\xi(0)=\eta
\end{cases}
\end{equation}
where $(y,\eta) \in T^* \rr d$. The Hamilton flow is denoted $( x(t), \xi(t) ) = \chi_t (y,\eta)$.

For the concept of singularities it is natural to adopt the $\sigma$-anisotropic Gabor wave front set $\WFgs(u)$ \cite[Definition~4.1]{RW2}
of a tempered distribution $u \in \cS'(\rr d)$,
defined with an anisotropy parameter $\sigma > 0$ as follows.
A point $z_0 = (x_0,\xi_0) \in T^* \rr d \setminus 0$ satisfies $z_0 \notin \WFgs ( u )$
if there exists an open set $U \subseteq T^* \rr d$ such that $z_0 \in U$ and 
\begin{equation*}
\sup_{(x,\xi) \in U, \ \lambda > 0} \lambda^N |V_\fy u (\lambda x, \lambda^\sigma \xi)| < + \infty \quad \forall N \geqs 0. 
\end{equation*}
The condition $(x,\xi) \in \WFgs ( u )$ thus means a lack of superpolynomial decay of the short-time Fourier transform 
along curves of the form
\begin{equation}\label{eq:sigmaconiccurve}
\ro_+ \ni \lambda \mapsto (\lambda x, \lambda^\sigma \xi) \in T^* \rr d \setminus 0
\end{equation}
in a neighborhood of $(x,\xi) \in T^* \rr d \setminus 0$. 
The $\sigma$-anisotropic Gabor wave front set $\WFgs(u)$ is $\sigma$-conic in the sense of invariant on curves of type 
\eqref{eq:sigmaconiccurve}.

An important insight is that propagation of singularities depend crucially on the parameter $p \in \ro \setminus 0$
in \eqref{eq:hamiltoniansymbol}, 
with a critical value 
\begin{equation}\label{eq:pcrit}
p_c = \frac12 \left( \frac1k + \frac1m \right). 
\end{equation}
Note that $p_c \leqs 1$ with equality only if $k = m = 1$, which is the well known case of the harmonic oscillator. 
To study propagation of singularities for \eqref{eq:anharmonicCP} with either $k > 1$ or $m > 1$ the parameter must be supercritical, that is $p > p_c$. 
We treat the cases $p \leqs p_c$ and $p > p_c$ using different notions of singularities.
Our theorem on the propagation of $\WFgs$ concerns $p \leqs p_c$. 

\begin{theorem}\label{thm:propanisogabor}
Suppose $k, m \in \no \setminus 0$, $\sigma = \frac{k}{m}$, $0 \neq p \leqs p_c$, 
let the symbol $a$ be defined by \eqref{eq:hamiltoniansymbol} and Remark \ref{rem:cutoffsmooth}, 
and consider the Cauchy problem \eqref{eq:SchrCP} with $A = a^w(x,D)$. 
If $u_0 \in \cS'(\rr d)$ then there exists a unique solution $u=u(t,\cdot) \in \cS'(\rr d)$ of \eqref{eq:SchrCP} for all $t \in \ro$.  
If $\chi_t: \rr {2d} \setminus 0 \to \rr {2d} \setminus 0$ is the Hamilton flow corresponding to $a$,  
then we have if $p = p_c$
\begin{equation}\label{eq:propcritical}
\WFgs ( u(t,\cdot) ) = \chi_t \left( \WFgs u_0 \right), \quad t \in \ro, 
\end{equation}
and if $p < p_c$
\begin{equation}\label{eq:propsubcritical}
\WFgs ( u (t,\cdot) ) = \WFgs ( u_0 ), \quad t \in \ro. 
\end{equation}
\end{theorem}

In fact the propagation result \eqref{eq:propcritical} when $p = p_c$  
is a particular case of \cite[Theorem~8.3]{CRW}. 
The new result \eqref{eq:propsubcritical} shows that when the exponent  
is subcritical ($p < p_c$) then the anisotropic Gabor wave front set does not propagate
but remains invariant. 

In the supercritical case $p > p_c$ 
there is evidence that $\WFgs$ is not the correct notion of singularities for a propagation result, 
see \cite[Remark~5.6 and Section~6]{CRW}. 
For $d = 1$ we work out the explicit, periodic Hamilton flow for the symbol \eqref{eq:hamiltoniansymbol} 
in Section \ref{sec:Hamiltonflow}. 
The results give further evidence against using $\WFgs$ in the supercritical case.

Thus the wave front set $\WFgs$ is not suitable to handle propagation for the generalized anharmonic oscillator equation 
\eqref{eq:anharmonicCP} if $(k,m) \neq (1,1)$, since $p = 1 > p_c$. 
For \eqref{eq:anharmonicCP} we introduce instead the notion of a filter of singularities $\mathcal F(u)$ of 
$u \in \cS'(\rr d)$. The idea has origins e.g. in \cite{RodinoAnnPisa}.   
The filter of singularities consists of anisotropic annular subsets of $T^* \rr d$ of the form 
\begin{equation}\label{eq:anisoradial}
\wt \Sigma = \{(x,\xi) \in T^* \rr d: |x|^{2k} + |\xi|^{2m} \in \Sigma\} \subseteq T^* \rr d
\end{equation}
where $\Sigma \subseteq \ro_+$. 
The subsets in the filter are the complements of anisotropic annular sets $\wt \Sigma \subseteq T^* \rr d$ 
where $u$ is smooth in the following sense: 
There exists $r \in \ro$ and a symbol $a \in G^{r,\sigma}$ (an anisotropic Shubin symbol) with $\sigma = \frac{k}{m}$, 
such that $a^w(x,D) u \in \cS(\rr d)$
and the symbol $a$ satisfies 
\begin{equation*}
|a (x,\xi)| \geqs C \left( 1 + |x| + |\xi|^{\frac1\sigma} \right)^r, \quad (x,\xi) \in \wt \Sigma, \quad |(x,\xi)| \geqs R, 
\end{equation*}
with $C, R > 0$. 
Note that the Hamilton flow is invariant on sets of the form \eqref{eq:anisoradial}. 
In the filter of singularities anisotropic annular sets of the form \eqref{eq:anisoradial}
replaces $\sigma$-conic subsets that are used in $\WFgs$. 

Our result for the invariance of the filter of singularities for the solutions to \eqref{eq:anharmonicCP} 
reads as follows.

\begin{theorem}\label{thm:propagationannular}
Suppose $k,m \in \no \setminus 0$ satisfy
\begin{equation}\label{eq:kmcondition} 
\frac1k + \frac1m > \frac23
\end{equation}
and let $u = u(t,\cdot)$ be the solution of \eqref{eq:anharmonicCP}
with $u_0 \in \cS'(\rr d)$, granted by Theorem \ref{thm:wptheorem1}. 
Then $\mathcal{F}(u(t,\cdot)) =\mathcal{F}(u_0)$ for every $t \in \ro$.
\end{theorem}

The condition \eqref{eq:kmcondition} is satisfied for any $k \in \no \setminus 0$ if $m = 1$ 
which is the anharmonic oscillator case. 

\begin{remark}\label{rem:WFGS}
A natural open problem is to study propagation of singularities in the setting of tempered ultradistributions of Gelfand--Shilov type, 
corresponding to Gelfand--Shilov versions of Theorems \ref{thm:propanisogabor} and  \ref{thm:propagationannular}. 
There are notions of Gelfand--Shilov wave front sets which are anisotropic \cite{RW3}. 
We do not treat this problem in the current paper. 
\end{remark}

The paper is organized as follows. In Section \ref{sec:prelim} we briefly recall the definition of $\sigma$-anisotropic symbols introduced in \cite{RW2} and the basic results of the related pseudodifferential calculus and microlocal analysis. 
In Section \ref{sec:wellposed} we prove the well-posedness Theorems \ref{thm:wptheorem1} and \ref{thm:wptheorem2}. In Section \ref{sec:Hamiltonflow} we compute explicitly the solution to Hamilton's system \eqref{HJintro} for the Hamiltonian \eqref{eq:hamiltoniansymbol} in dimension $d=1$ and deduce estimates for the derivatives of the composition of an anisotropic symbol with the inverse Hamilton flow. The critical power $p_c$ \eqref{eq:pcrit} plays a crucial role. 
In Section \ref{sec:propanisogabor} we prove 
Theorem \ref{thm:propanisogabor} where
$p \leqs p_c$. 
The last two sections are devoted to the case $p = 1 > p_c$. In Section \ref{sec:filtersing} we introduce the filter of singularities and study its properties, 
and in Section \ref{sec:propannular} we prove Theorem \ref{thm:propagationannular}.

\section{Preliminaries}\label{sec:prelim}

The unit sphere in $\rr d$ is denoted $\sr {d-1} \subseteq \rr d$. 
An open ball of center $x_0 \in \rr d$ and radius $r > 0$ is denoted $\rB_r (x_0)$, 
and $\rB_r (0) = \rB_r$. 
The set of positive real (rational) numbers is denoted $\ro_+$ ($\qo_+$), and the power set of the set $\Omega$ is denoted $\rP(\Omega)$. 
We write $f (x) \lesssim g (x)$ provided there exists $C>0$ such that $f (x) \leqs C \, g(x)$ for all $x$ in the domain of $f$ and of $g$. 
If $f (x) \lesssim g (x) \lesssim f(x)$ then we write $f \asymp g$. 
The partial derivative $D_j = - i \partial_j$, $1 \leqs j \leqs d$, acts on functions and distributions on $\rr d$, 
with extension to multi-indices. 
If $I \subseteq \ro$ is an open interval, $k \in \no$ and $X$ is a topological vector space then $\cB^k (I, X)$ 
denotes the space of all functions $f: I  \to X$ 
such that $f^{(j)} \in L^\infty(I, X) \cap C(I,X)$ for all $j = 0, 1, \dots, k$, 
and 
\begin{equation*}
\cB^\infty(I) = \bigcap_{k \in \no} \cB^k (I, \ro). 
\end{equation*}
We use $\eabs{x} = (1 + |x|^2)^{\frac12}$ for $x \in \rr d$.
The Poisson bracket of $a(x,\xi), b(x,\xi) \in C^1(\rr {2d})$ is 
\begin{equation}\label{eq:Poissonbracket}
\{ a, b \} = \la \nabla_\xi a, \nabla_x b \ra - \la \nabla_x a, \nabla_\xi b \ra
\end{equation}
where $\la \cdot, \cdot \ra$ denotes the scalar product on $\rr d$. 

The normalization of the Fourier transform is
\begin{equation*}
 \cF f (\xi )= \widehat f(\xi ) = (2\pi )^{-\frac d2} \int _{\rr
{d}} f(x)e^{-i \la x,\xi \ra }\, \dd x, \qquad \xi \in \rr d, 
\end{equation*}
for $f\in \cS(\rr d)$ (the Schwartz space). 
The conjugate linear action of a tempered distribution $u \in \cS'(\rr d)$ on a test function $\phi \in \cS(\rr d)$ is written $(u,\phi)$, consistent with the $L^2$ inner product $(\cdot ,\cdot ) = (\cdot ,\cdot )_{L^2}$ which is conjugate linear in the second argument. 

Denote translation by $T_x f(y) = f( y-x )$ and modulation by $M_\xi f(y) = e^{i \la y, \xi \ra} f(y)$ 
for $x,y,\xi \in \rr d$ where $f$ is a function or distribution defined on $\rr d$. 
Let $\fy \in \cS(\rr d) \setminus \{0\}$. 
The short-time Fourier transform (STFT) of a tempered distribution $u \in \cS'(\rr d)$ is defined by 
\begin{equation*}
V_\fy u (x,\xi) = (2\pi )^{-\frac d2} (u, M_\xi T_x \fy) = \cF (u T_x \overline \fy)(\xi), \quad x,\xi \in \rr d. 
\end{equation*}
The function $V_\fy u$ is smooth and polynomially bounded \cite[Theorem~11.2.3]{Grochenig1} as
\begin{equation}\label{eq:STFTtempered}
|V_\fy u (x,\xi)| \lesssim \eabs{(x,\xi)}^{k}, \quad (x,\xi) \in T^* \rr d, 
\end{equation}
for some $k \geqs 0$. 
We have $u \in \cS(\rr d)$ if and only if
\begin{equation}\label{eq:STFTschwartz}
|V_\fy u (x,\xi)| \lesssim \eabs{(x,\xi)}^{-N}, \quad (x,\xi) \in T^* \rr d, \quad \forall N \geqs 0.  
\end{equation}

The inverse transform is given by
\begin{equation}\label{eq:STFTinverse}
u = (2\pi )^{-\frac d2} \iint_{\rr {2d}} V_\fy u (x,\xi) M_\xi T_x \fy \, \dd x \, \dd \xi
\end{equation}
provided $\| \fy \|_{L^2} = 1$, with action under the integral understood, that is 
\begin{equation}\label{eq:moyal}
(u, f) = (V_\fy u, V_\fy f)_{L^2(\rr {2d})}
\end{equation}
for $u \in \cS'(\rr d)$ and $f \in \cS(\rr d)$, cf. \cite[Theorem~11.2.5]{Grochenig1}. 

The symbol $\sigma > 0$ will denote the anisotropy parameter.  
A $\sigma$-conic subset of $T^* \rr d \setminus 0$ is closed under the anisotropic scaling
$T^* \rr d \setminus 0 \ni (x,\xi) \mapsto ( \lambda x, \lambda^\sigma \xi) \in T^* \rr d \setminus 0$ for all $\lambda > 0$. 
We use the anisotropic phase space weight
\begin{equation*}
\theta_\sigma(x,\xi)= 1+|x|+|\xi|^{\frac1{\sigma}}, \quad x, \xi \in \rr d. 
\end{equation*}
Then 
\begin{equation}\label{eq:lambdaboundisotropic}
\eabs{ (x,\xi) }^{\min \left( 1, \frac1\sigma \right)}
\lesssim \theta_\sigma(x,\xi)
\lesssim \eabs{(x,\xi)}^{\max \left( 1, \frac1\sigma \right)}, \quad (x,\xi) \in \rr {2d} \setminus 0. 
\end{equation}

We present briefly the pseudodifferential calculus we will use in the proof of Theorems
\ref{thm:propanisogabor} and \ref{thm:propagationannular}, and refer to \cite{BBR, Hormander2, Nicola1, Shubin1, RW2} for more precise and general presentations.
The space of $\sigma$-anisotropic global symbols $G^{r,\sigma}$ of order $r \in \ro$ consists of all functions $a \in C^\infty(\rr {2d})$ 
which satisfy the estimates
\begin{equation}\label{eq:anisosymbolest} 
|\pdd x \alpha \pdd \xi \beta a(x,\xi)| 
\lesssim \theta_\sigma (x,\xi)^{r-|\alpha|-\sigma|\beta|}, \qquad x,\xi \in \rr d, \quad \alpha, \beta \in \nn d.
\end{equation}
We have
\begin{equation*}
\bigcap_{r \in \ro} G^{r,\sigma} = \cS(\rr {2d})  
\end{equation*}
and $G^{r,1} = G_1^r$, 
where $G_\rho^r$ denotes the standard Shubin class with parameter $0 \leqs \rho \leqs 1$ \cite{Shubin1}.
A symbol $a \in G_\rho^r$ satisfies 
\begin{equation*}
|\pdd x \alpha \pdd \xi \beta a(x,\xi)| 
\lesssim \eabs{(x,\xi)}^{r - \rho|\alpha + \beta|}, \qquad x,\xi \in \rr d, \quad \alpha, \beta \in \nn d. 
\end{equation*} 

\begin{remark}\label{rem:Shubin}
As a consequence of \eqref{eq:lambdaboundisotropic} we have the inclusion 
\begin{equation}\label{eq:Gmsigmainclusion}
G^{r,\sigma} \subseteq G_\rho^{r_0}, 
\end{equation}
where $r_0 = \max \left( r, \frac{r}{\sigma} \right)$ and $\rho = \min \left( \sigma, \frac1\sigma \right)$. 
Hence the Shubin calculus \cite{Shubin1} applies to the anisotropic Shubin symbols, but the anisotropy is lost. 
\end{remark}

\begin{remark}\label{rem:NicolaRodino}
The pseudodifferential calculus in \cite{Nicola1} is not directly applicable to the symbol classes $G^{r,\sigma}$ unless $\sigma = 1$. 
In fact if $\sigma \neq 1$ then either the space weight function $\Phi(x,\xi) = 1 + |x| + |\xi|^{\frac1\sigma}$ or the frequency weight function
$\Psi(x,\xi) = 1 + |x|^\sigma + |\xi|$ is not sublinear, which is required in \cite[Eq.~(1.1.1)]{Nicola1}. 
Nevertheless from \eqref{eq:lambdaboundisotropic} it follows that $G^{r,\sigma} \subseteq S(M; \Phi, \Psi)$ 
as defined in \cite[Definition~1.1.1]{Nicola1} with 
$M(x,\xi) = \theta_\sigma(x,\xi)^r$, $\Phi(x,\xi) = \eabs{(x,\xi)}^{\min \left(1, \frac1\sigma \right)}$
and $\Psi(x,\xi) = \eabs{(x,\xi)}^{\min(1,\sigma)}$. 
The functions $\Phi$ and $\Psi$ are sublinear. 
The weight $M$ is anisotropic but $\Phi$ and $\Psi$ are not. 
Thus the pseudodifferential calculus in \cite[Chapter~1.2]{Nicola1} applies to $G^{r,\sigma}$, 
but the anisotropic behavior of the derivatives is again lost. 
\end{remark}

\begin{remark}\label{rem:symbolcomparison}
When $\sigma > 0$ is rational, that is $\sigma = \frac{k}{m}$ with $k,m \in \no \setminus 0$
then the symbol classes $G^{r,\sigma}$ coincide with the symbol classes $S_{m,k}^r$ in 
\cite{Helffer1}, as $G^{r,\sigma} = S_{m,k}^{r m}$ for $r \in \ro$. 
In the general case $\sigma \in \ro_+$ it has been shown in
\cite[Proposition~4.2]{CDR2} that $G^{r,\sigma}$ is a particular case 
of the Weyl--H\"ormander symbol classes \cite[Chapter~18.4]{Hormander1}. 
More precisely if $r \in \ro$ then $G^{r,\sigma} = S( \theta_\sigma^r, g)$
with the metric
\begin{equation*}
g = \frac{\dd x^2}{\theta_\sigma(x,\xi)^2} + \frac{\dd \xi^2}{ \theta_\sigma(x,\xi)^{2 \sigma} }. 
\end{equation*}
\end{remark}

For $a \in G^{r,\sigma}$ and $\tau \in \ro$ a pseudodifferential operator in the $\tau$-quantization is defined by
\begin{equation}
a_\tau(x,D)f(x)= (2\pi)^{-d} \int_{\rr {2d}}e^{i\langle x-y,\xi \rangle} a ((1-\tau)x+\tau y,\xi) f(y)\, \dd y \, \dd\xi,
\end{equation}
interpreted as an oscillatory integral. 
If $\tau=0$ we get the Kohn--Nirenberg (normal) quantization $a(x,D) = a_0(x,D)$, and if $\tau=\frac12$ we get the Weyl quantization. 
We will use exclusively the Weyl quantization in this paper, writing $a^w = a^w(x,D) = a_{\frac12}(x,D)$. 
By \cite[Proposition~3.3]{RW2} the classes $G^{r,\sigma}$ are invariant under change of the quantization parameter $\tau \in \ro$.

If $a \in G^{r,\sigma}$  the operator $a^w(x,D)$ acts continuously on $\cS(\rr d)$ and extends by duality to a continuous operator on $\cS'(\rr d)$.
Given a sequence $a_j \in G^{r_j,\sigma}$, $j=1,2,\ldots$, with $r_j \to -\infty$ as $j \to +\infty$, we write 
$a \sim \sum_{j=1}^\infty a_j$ to mean that for $N \geqs 2$ 
\begin{equation*}
a-\sum_{j=1}^{N-1} a_j \in G^{\mu_N,\sigma},
\end{equation*}
with $\mu_N = \max\limits_{j \geqs N} r_j$.  
Given a sequence $(a_j)$ as above, 
we can construct a symbol $a \in G^{\mu_1,\sigma}$, unique modulo $\cS(\rr {2d})$, such that $a \sim \sum_{j=1}^\infty a_j$ \cite[Lemma~3.2]{RW2}.

The Weyl product $a \wpr b$ of two symbols $a \in G^{r,\sigma}$ and $b \in G^{t,\sigma}$ 
is defined as the symbol of the operator composition: $(a \wpr b)^w(x,D) = a^w(x,D) b^w(x,D)$ provided the composition is well defined. 
By \cite[Proposition~3.3]{RW2} we have $a \wpr b \in G^{r+t,\sigma}$, 
and asymptotic expansion 
\begin{equation*}
a \wpr b(x,\xi) \sim \sum_{\alpha, \beta \geqs 0} \frac{(-1)^{|\beta|}}{\alpha! \beta!} \ 2^{-|\alpha+\beta|}
D_x^\beta \pdd \xi \alpha a(x,\xi) \, D_x^\alpha \pdd \xi \beta b(x,\xi)
\end{equation*}
with terms in $G^{r + t -(1 + \sigma)|\alpha + \beta|,\sigma}$.

We will use the following scale of regularity spaces \cite[Definition~4.1]{CRW}.

\begin{definition}\label{def:Sobolevaniso}
Let $\fy \in \cS(\rr d) \setminus 0$. 
The anisotropic Shubin--Sobolev modulation space $M_{\sigma,s} (\rr d)$ with anisotropy parameter $\sigma > 0$ and order $s \in \ro$ is the Hilbert subspace of $\cS'(\rr d)$
defined by the norm
\begin{equation}\label{eq:SSmodnorm}
\| u \|_{M_{\sigma,s}} = \left( \iint_{\rr {2 d}} |V_\fy u (x,\xi)|^2 \, \theta_\sigma (x,\xi)^{2 s} \, \dd x \, \dd \xi \right)^{\frac12}. 
\end{equation}
\end{definition}

For any $\sigma > 0$ we have $M_{\sigma,0} (\rr d) = L^2(\rr d)$ \cite{Grochenig1}, 
and $M_{\sigma,s_1} (\rr d) \subseteq M_{\sigma,s_2}(\rr d)$ is a continuous inclusion when $s_1 \geqs s_2$. 
It holds \cite{Grochenig1}
\begin{equation}\label{eq:schwartzmodsp}
\cS(\rr d) = \bigcap_{s \in \ro} M_{\sigma,s} (\rr d)
\end{equation}
and 
\begin{equation}\label{eq:temperedmodsp}
\cS'(\rr d) = \bigcup_{s \in \ro} M_{\sigma,s} (\rr d).  
\end{equation}

If $\sigma = \frac{k}{m}$ with $k,m \in \no \setminus 0$
then the norm $\| \cdot \|_{M_{\sigma,s}}$ is equivalent to the norm 
\begin{equation}\label{eq:locopmodspace}
\| u \|_{M_{\sigma,s}(\rr d)}
\asymp 
\| A_s u \|_{L^2(\rr d)} 
\end{equation}
where $A_s$ is a 
localization operator defined by
\begin{equation}\label{eq:locop}
(A_s f, g) = ( w_{k,m}^{\frac{s}{k}} V_\fy f, V_\fy g) = (V_\fy^* w_{k,m}^{\frac{s}{k}} V_\fy f, g), \quad f,g \in \cS(\rr d), 
\end{equation}
where $\fy \in \cS(\rr d)$ satisfy $\| \fy \|_{L^2} = 1$ \cite{CRW,Grochenig2}. 
It has symbol $w_{k,m}^{\frac{s}{k}} \in C^\infty(\rr {2d})$ where 
\begin{equation*}
w_{k,m} (x, \xi) 
= \left( 1 + | x |^{2k} + | \xi |^{2 m} \right)^{\frac12}
\asymp \theta_{\sigma}(x,\xi)^{k}. 
\end{equation*}

Pseudodifferential operators with anisotropic Shubin symbols act continuously between 
$M_{\sigma,s}(\rr d)$ spaces with natural loss of regularity according to \cite[Proposition~4.2]{CRW}: 

\begin{theorem}\label{thm:PseudoShubinSobolev}
If $\sigma > 0$, $r \in \ro$ and $a \in G^{r,\sigma}$, then for all $s \in \ro$ the operator
\begin{equation*}
a^w(x,D): M_{\sigma,s+r}(\rr d) \to M_{\sigma,s}(\rr d)
\end{equation*}
is continuous. 
\end{theorem}

\begin{remark}\label{rem:shubinsobolevspace}
The spaces $M_{\sigma,s}(\rr d)$ are particular cases of the Sobolev spaces in \cite[Sections~1.5 and 1.7.4]{Nicola1}. 
If $\sigma = \frac{k}{m} \in \qo_+$ then 
they are also identical to Helffer and Robert's spaces $B_{m,k}^s(\rr d)$ \cite[Remarque~4.2]{Helffer1}
according to $B_{m,k}^s = M_{\sigma,(1 + \sigma)s}$ for $s \in \ro$, 
and the spaces $Q_{m,k}^s(\rr d)$ in \cite{CGPR}
according to $Q_{m,k}^s = M_{\sigma,\min(1,\sigma) s}$ for $s \in \ro$. 
Finally the spaces $H_{m,k}^s(\rr d)$ in\cite[Definition~4.7]{CDR1}
can be expressed as 
$H_{m,k}^s = M_{\sigma,ks}$ when $\sigma = \frac{k}{m} \in \qo_+$ and $s \in \ro$. 
\end{remark}

We will use anisotropic Gabor wave front sets \cite{RW2,Wahlberg1} of tempered distributions. 
The concept is related to H.~Zhu's \cite[Definition~1.3]{Zhu1} of quasi-homogeneous 
wave front set defined by two non-negative parameters. 
Zhu uses a semiclassical formulation whereas we use the STFT. 

\begin{definition}\label{def:WFgs}
Suppose $u \in \cS'(\rr d)$, $\fy \in \cS(\rr d) \setminus 0$, and $\sigma > 0$. 
A point $z_0 = (x_0,\xi_0) \in T^* \rr d \setminus 0$ satisfies $z_0 \notin \WFgs ( u )$
if there exists an open set $U \subseteq T^* \rr d$ such that $z_0 \in U$ and 
\begin{equation}\label{eq:WFgs}
\sup_{(x,\xi) \in U, \ \lambda > 0} \lambda^N |V_\fy u (\lambda x, \lambda^\sigma \xi)| < + \infty \quad \forall N \geqs 0. 
\end{equation}
\end{definition}

If $\sigma = 1$ we have $\WFgs ( u ) = \WFg (u)$ 
which denotes the usual Gabor wave front set \cite{Hormander3,RW1}. 
We call $\WFgs( u )$ the $\sigma$-anisotropic Gabor wave front set. 
The set $\WFgs ( u )$ is $\sigma$-conic. 
Referring to \eqref{eq:STFTtempered} and \eqref{eq:STFTschwartz} we see
that $\WFgs ( u )$ records $\sigma$-conic curves 
$0 < \lambda \mapsto (\lambda x, \lambda^\sigma \xi)$ where $V_\fy u$ does not behave like the STFT of a Schwartz function. 
If $u \in \cS'(\rr d)$ then $\WFgs ( u ) = \emptyset$ if and only if $u \in \cS (\rr d)$. 

As wave front sets of many other types, the anisotropic Gabor wave front set can be written as 
\begin{equation}\label{eq:WFchar}
\WFgs (u) = \bigcap_{a \in G^{r,\sigma}: \ a^w(x,D) u \in \cS} \charac_\sigma (a)
\end{equation}
where $r \in \ro$ and where the characteristic set $\charac_\sigma (a) \subseteq T^* \rr d \setminus 0$
is defined as follows. 
It holds $z_0 \notin \charac_\sigma (a)$ if there exists 
an open $\sigma$-conic set $\Gamma \subseteq T^* \rr d \setminus 0$ containing $z_0 \neq 0$
such that 
\begin{equation*}
|a( x, \xi )| \geqs C \theta_\sigma(x,\xi)^r, \quad (x,\xi) \in \Gamma, \quad |(x, \xi) |  \geqs R, 
\end{equation*}
for suitable $C, R > 0$. See \cite[Proposition~3.5]{Wahlberg1}. 

Let $\nu ,\mu > 0$. The
Gelfand--Shilov space $\mathcal S _\nu^\mu (\rr d)$ ($\Sigma _\nu^\mu (\rr d)$) of
Roumieu (Beurling) type \cite{Gelfand1} consists of all $f\in C^\infty (\rr d)$
such that
\begin{equation}\label{eq:gsseminorm}
\| f \|_{\mathcal S_{\nu,h}^\mu} \equiv \sup \frac {|x^\alpha \partial ^\beta
f(x)|}{h^{|\alpha  + \beta |}\alpha !^\nu \, \beta !^\mu}
\end{equation}
is finite for some (every) $h>0$. 
The supremum refers to all $\alpha ,\beta \in \mathbf N^d$ and $x\in \rr d$.

The seminorms \eqref{eq:gsseminorm} induce an inductive limit topology for the
space $\mathcal S _\nu^\mu (\rr d)$ and a projective limit topology for
$\Sigma _\nu^\mu (\rr d)$. The latter space is a Fr{\'e}chet space under this topology.
The space $\mathcal S _\nu^\mu (\rr d) \neq \{ 0\}$ ($\Sigma _\nu^\mu (\rr d) \neq \{0\}$),
if and only if $\nu +\mu \geqs 1$ ($\nu+\mu > 1$) \cite{Petersson1}.

\section{Well-posedness} \label{sec:wellposed}

In this section we prove the well-posedness results stated in the Introduction.
We use characterizations of the spaces $M_{\sigma,s}(\rr d)$ with $s \in \ro$, $\cS(\rr d)$, and $\cS'(\rr d)$,
as well as the Gelfand--Shilov spaces $\mathcal{S}^\mu_\nu(\rr d)$ and $\Sigma_\nu^\mu(\rr d)$, 
in terms of the behavior of the coefficients of their elements
with respect to the orthonormal basis of eigenfunctions of a positive operator $A$ of the form \eqref{eq:diffoperator} satisfying \eqref{eq:ellipticity}. 
We use results from \cite{gpr1, CGPR}.
 
Let $k,m \in 2\no \setminus 0$, $\sigma = \frac{k}{m}$ and define the operator $A$ by \eqref{eq:diffoperator}.
The Kohn--Nirenberg symbol for $A = a(x,D)$ is 
\begin{equation*}
a(x,\xi) = \sum\limits_{k|\alpha| +m|\beta| \leqs km} c_{\alpha\beta} x^\beta \xi^\alpha. 
\end{equation*}
By \cite[Example~3.9]{RW2} we know that $a \in G^{k,\sigma}$, and by \cite[Proposition~3.3]{RW2}
we have $A = b^w(x,D)$ with $b \in G^{k,\sigma}$. 
The anisotropic ellipticity condition \eqref{eq:ellipticity} amounts to \cite{CGR2}
\begin{equation}\label{eq:ellipticity2}
|a(x,\xi)| \geqs C \theta_\sigma(x,\xi)^k, \quad |(x,\xi)| \geqs R, 
\end{equation}
for some $C, R > 0$. 
The operator
\begin{equation*}
A: M_{\sigma,k}(\rr d)
\to L^2(\rr d)
\end{equation*}
is continuous by Theorem \ref{thm:PseudoShubinSobolev}. 

By \cite[Theorem~4.2.4]{Nicola1} the closure of $A$ as an unbounded operator in $L^2$
coincides with its maximal realization, due to the ellipticity \eqref{eq:ellipticity2}. 
In the sequel we denote the closure still by $A$. 

The operator $A$ is Fredholm 
\cite[Theorem~10.1]{BBR}, \cite[Theorem~1.6.9]{Nicola1}. 
The finite-dimensional null space $\textrm{Ker}\, A$ consists of functions in the Gelfand--Shilov space $\mathcal{S}_{\frac{m}{k+m}}^{\frac{k}{k+m}}(\rr d)$ \cite{CGR2}.
The assumption that $A$ is positive implies that its Weyl symbol is real-valued \cite[p.~30]{CRW}. 
The added assumption of ellipticity in the sense of \eqref{eq:ellipticity2}
guarantees the existence
of an orthonormal  basis of eigenfunctions $\fy_j \in L^2(\rr d)$, $j \in \no \setminus 0$, 
with eigenvalues $\lambda_j > 0$ such that $\lim\limits_{j\to +\infty }\lambda_j = +\infty$ (see \cite{CGR2,Shubin1} and \cite[Theorem~4.2.9]{Nicola1}). 

From \cite[Theorem 3.2, page 129]{BBR} one deduces the asymptotic behavior of the eigenvalues
\begin{equation}\label{eq:asympeigenvalues}
\lambda_j \asymp j^{\frac{mk}{d(m+k)}} 
= j^{\frac{k}{d (1+\sigma)}}
\quad {\rm as} \quad j \to +\infty.
\end{equation}

Since the Weyl symbol of $A-\lambda_j$ also satisfies \eqref{eq:ellipticity2} it follows 
that $\fy_j \in \mathcal{S}_{\frac{m}{k+m}}^{\frac{k}{k+m}}(\rr d)$ for $j \geqs 1$ \cite{CGR2}.
Given $u \in L^2(\rr d)$ we can expand
\begin{equation}\label{eq:expansion}
u=\sum\limits_{j=1}^{+\infty} u_j \fy_j
\end{equation}
with coefficients $( u_j )_{j \geqs 1} \in l^2(\no \setminus 0)$ defined by
\begin{equation}\label{eq:coefficient}
u_j = ( u,\fy_j ), \quad j = 1,2,\ldots . 
\end{equation}

The following result characterizes anisotropic Shubin--Sobolev modulation spaces $M_{\sigma,s}(\rr d)$ 
when $\sigma \in \qo_+$
and $s \in \ro$
in terms of the behavior of the coefficients with respect to the eigenvectors and the eigenvalues of a positive elliptic differential operator.

\begin{proposition}\label{prop:modspacenorm}
Let $k,m \in 2\no \setminus 0$, $\sigma = \frac{k}{m}$ and define the operator $A$ by \eqref{eq:diffoperator}.
Suppose $A$ is positive and elliptic in the sense of \eqref{eq:ellipticity}, 
let $(\fy_j)_{j \geqs 1} \subseteq \mathcal{S}_{\frac{m}{k+m}}^{\frac{k}{k+m}}(\rr d)$ be its orthonormal basis of eigenvectors, 
and let $( \lambda_j )_{j \geqs 1} \subseteq \ro_+$ be the corresponding eigenvalues.
If $u \in \cS'(\rr d)$ and $u_j = (u,\fy_j)$ for $j \geqs 1$ then for all $s \in \ro$
\begin{equation}\label{eq:charmodsp}
\| u \|_{M_{\sigma,s}(\rr d)}^2
\asymp \sum_{j = 1}^{+\infty} \lambda_j^{\frac{2 s}{k}} | u_j |^2.
\end{equation}
\end{proposition}

\begin{proof}
According to the discussion above
the Weyl symbol of $A = b^w(x,D)$ satisfies $b \in G^{k,\sigma}$, and the symbol $b$ is elliptic in the sense of 
\begin{equation*}
| b(x,\xi) | \geqs C \theta_\sigma(x,\xi)^k, \quad |(x,\xi)| \geqs R, 
\end{equation*}
for some $C, R > 0$. 

By Remark \ref{rem:NicolaRodino} we may consider the symbol $b$ as an anisotropic symbol 
in an isotropic calculus.  
This means that $b \in S(M; \Phi, \Psi)$ 
as defined in \cite[Definition~1.1.1]{Nicola1} with 
$M(x,\xi) = \theta_\sigma(x,\xi)^k$, $\Phi(x,\xi) = \eabs{(x,\xi)}^{\min \left(1, \frac1\sigma \right)}$
and $\Psi(x,\xi) = \eabs{(x,\xi)}^{\min(1,\sigma)}$. 
Note that $M$ is anisotropic while $\Phi$ and $\Psi$ are isotropic. 
The symbol $b$ is then elliptic, denoted $b \in {\rm Hypo}(M, M; \Phi, \Psi)$ in the sense of \cite[Definition~1.3.2]{Nicola1}. 
Since $b^w(x,D)$ is positive, the symbol $b$ is real-valued. 

It now follows from \cite[Theorems~4.3.5 and 4.3.6]{Nicola1} that for any $p > 0$ the power operator $A^p$
is a well defined unbounded operator in $L^2(\rr d)$, with eigenvalues $( \lambda_j^p) _{j \geqs 1}$, and 
Weyl symbol $b_p \in {\rm Hypo}(M^p, M^p; \Phi, \Psi)$. 
By the spectral theorem the operator $A^p$ has the same eigenfunctions $( \fy_j)_{j \geqs 1}$ as $A$ so we have
if $u \in \cS(\rr d)$
\begin{equation*}
A^p u = \sum_{j=1}^{+\infty}  \lambda_j^p (u, \fy_j) \fy_j. 
\end{equation*}

Suppose $s > 0$. If $p = \frac{s}{k}$ then $b_p \in {\rm Hypo}(M^p, M^p; \Phi, \Psi)$ is an elliptic symbol. 
We note that $M^p = \theta_\sigma^s \asymp w_{k,m}^{\frac{s}{k}}$. 
By \cite[Propositions~1.5.5~(a) and 1.7.12, and Remark~1.7.13]{Nicola1} the operator 
$A^{\frac{s}{k}} = b_{\frac{s}{k}}^w(x,D) : M_{\sigma,s} (\rr d) \to L^2(\rr d)$ is continuous and 
invertible with bounded inverse 
\begin{equation*}
\left( A^{\frac{s}{k}} \right)^{-1} u = \sum_{j = 1}^{+\infty} \lambda_j^{- \frac{s}{k}} (u, \fy_j) \fy_j : L^2(\rr d) \to M_{\sigma,s} (\rr d). 
\end{equation*}

The inverse is in fact a pseudodifferential operator: $\left( A^{\frac{s}{k}} \right)^{-1} = c_\frac{s}{k}^w(c,D)$
for some $c_\frac{s}{k} \in {\rm Hypo}(M^{-p}, M^{-p}; \Phi, \Psi)$. 
Indeed by \cite[Theorem~1.3.6]{Nicola1} there exists a parametrix $c \in {\rm Hypo}(M^{-p}, M^{-p}; \Phi, \Psi)$ such 
that $b_{\frac{s}{k}}^w c^w = I + r^w$ where $r \in \cS(\rr {2d})$, 
which gives 
\begin{equation*}
\left( A^{\frac{s}{k}} \right)^{-1} = c^w - \left( A^{\frac{s}{k}} \right)^{-1} r^w
= c_\frac{s}{k}^w
\end{equation*}
since $\left( A^{\frac{s}{k}} \right)^{-1} r^w: \cS' \to \cS$ is regularizing with a symbol in $\cS(\rr {2d})$
which can be absorbed in $c$ into $c_\frac{s}{k} \in {\rm Hypo}(M^{-p}, M^{-p}; \Phi, \Psi)$. 

As a consequence of \eqref{eq:locopmodspace}, \cite[Propositions~1.5.3 and 1.7.12]{Nicola1} we may define the 
$M_{\sigma,s}(\rr d)$ norm equivalently by replacing the localization operator $A_s$ in \eqref{eq:locopmodspace}
by the operator $A^{\frac{s}{k}}$. 
Hence
\begin{equation*}
\| u \|_{M_{\sigma,s}(\rr d)}^2
\asymp \| A^{\frac{s}{k}} u \|_{L^2(\rr d)}^2
= \sum_{j = 1}^{+\infty} \lambda_j^{\frac{2 s}{k}} | u_j |^2
\end{equation*}
which proves the claim \eqref{eq:charmodsp} when $s > 0$. 

Finally we need to consider $s \leqs 0$. Since $M_{\sigma,0} = L^2$ the result is trivial when $s = 0$
so we may assume $s < 0$. Again using
\eqref{eq:locopmodspace}, \cite[Propositions~1.5.3 and 1.7.12]{Nicola1} we may define the 
$M_{\sigma,s}(\rr d)$ norm equivalently by replacing the localization operator $A_s$ in \eqref{eq:locopmodspace}
by the operator $(A^{\frac{|s|}{k}})^{-1} = c_\frac{|s|}{k}^w$. 
We get
\begin{equation*}
\| u \|_{M_{\sigma,s}(\rr d)}^2
\asymp \| (A^{\frac{|s|}{k}})^{-1} u \|_{L^2(\rr d)}^2
= \sum_{j = 1}^{+\infty} \lambda_j^{-\frac{2 |s|}{k}} | u_j |^2
= \sum_{j = 1}^{+\infty} \lambda_j^{\frac{2 s}{k}} | u_j |^2. 
\end{equation*}
\end{proof}

\begin{remark}\label{rem:isotropic}
In the proof of Proposition \ref{prop:modspacenorm} we use results from 
the pseudodifferential calculus in \cite{Nicola1}, which does not admit 
anisotropic space and frequency weights $\Phi$ and $\Psi$ respectively, cf. Remark \ref{rem:NicolaRodino}. 
But this calculus does admit anisotropic weight functions $M$. 
This suffices for our purposes in Proposition \ref{prop:modspacenorm} 
concerning the anisotropic Sobolev--Shubin modulation spaces
$M_{\sigma,s}(\rr d)$. 
\end{remark}

\begin{remark}\label{rem:seminormschwartz}
The proof of Proposition \ref{prop:modspacenorm} and \cite{CRW,Grochenig1} show that 
\begin{equation*}
\cS(\rr d) \ni u \mapsto \left( \sum_{j = 1}^{+\infty} \lambda_j^{\frac{2 s}{k}} | (u, \fy_j) |^2 \right)^{\frac12}
\end{equation*}
for $s \geqs 0$ is a family of seminorms for the Fr\'echet space topology of $\cS(\rr d)$. 
Using the orthonormality of $(\fy_j)_{j \geqs 1}$,  from Proposition \ref{prop:modspacenorm} and Remark \ref{rem:seminormschwartz} we obtain that 
if $u \in \cS'(\rr d)$ then for some $C, s > 0$
\begin{align*}
|u_j| 
& = | (u, \fy_j)|
\leqs C \left( \sum_{n = 1}^{+\infty} \lambda_n^{\frac{2 s}{k}} | (\fy_j, \fy_n) |^2 \right)^{\frac12} 
= C \lambda_j^{\frac{s}{k}} 
\asymp j^{\frac{s m}{d(k+m)}} \\
& = j^{\frac{s}{d(1+\sigma)}}, \quad j \in \no \setminus 0, 
\end{align*}
in the last step using \eqref{eq:asympeigenvalues}. 
\end{remark}

\begin{corollary}\label{cor:wellposedschwartz}
Under the assumptions of Proposition \ref{prop:modspacenorm} we have with $u_j = (u,\fy_j)$ for $j \geqs 1$
\begin{align}
& u \in \cS (\rr d) \quad 
\Longleftrightarrow \quad 
|u_j|= O(\lambda_j^{-s}), \quad j \to +\infty, \quad \forall s > 0 \label{eq:charschwartz1} \\ 
& \qquad \qquad \qquad \Longleftrightarrow  \quad |u_j|= O(j^{-s}), \quad j\to +\infty, \quad \forall s > 0. \label{eq:charschwartz2}
\end{align}
\end{corollary}

\begin{proof}
The conclusion \eqref{eq:charschwartz1} is a consequence of \eqref{eq:charmodsp} and \eqref{eq:schwartzmodsp}, 
and \eqref{eq:charschwartz2} follows from \eqref{eq:asympeigenvalues}. 
\end{proof}

From Proposition \ref{prop:modspacenorm}, Remark \ref{rem:seminormschwartz} 
and Corollary \ref{cor:wellposedschwartz} it follows that 
the expansion \eqref{eq:expansion} converges in $\cS'(\rr d)$ when $u \in \cS'(\rr d)$
with coefficients \eqref{eq:coefficient}. 
Conversely if $(u_j)_{j \geqs 1} \subseteq \co$ is a sequence such that $|u_j| \lesssim j^{r}$
for some $r \geqs 0$ then the expansion \eqref{eq:expansion} converges in $\cS'(\rr d)$. 

\vskip0.2cm

\textit{Proof of Theorem \ref{thm:wptheorem1}.}
Let $u_0 \in M_{\sigma,s}(\rr d)$. 
We can write 
$u_0 = \sum_{j=1}^{+\infty} c_j \fy_j$ where the coefficients $c_j = (u_0, \fy_j)$ satisfy
\eqref{eq:charmodsp} in Proposition \ref{prop:modspacenorm}. 
Define for $t \in \ro$ the series 
\begin{equation}\label{eq:seriessolution}
u(t,x) = \sum_{j=1}^\infty c_j e^{ - i\lambda_j t } \fy_j(x). 
\end{equation}

Then $u(0,\cdot) = u_0$ and for $j \in \no \setminus 0$ we have 
\begin{equation}\label{eq:solutioncoeff}
( u(t,\cdot), \fy_j) = c_j e^{- i \lambda_j t } 
\end{equation}
so Proposition \ref{prop:modspacenorm} gives 
\begin{equation}\label{eq:solutionnorm}
\| u(t,\cdot) \|_{M_{\sigma,s}}^2
\asymp \sum_{j=1}^\infty \lambda_j^{\frac{2s}{k}} |c_j|^2 
\asymp \| u_0 \|_{M_{\sigma,s}}^2.
\end{equation}
The series \eqref{eq:seriessolution} hence defines a function
$u(t,\cdot) \in L^\infty(\ro, M_{\sigma,s}(\rr d) )$. 
Due to $M_{\sigma,0} = L^2$ it is clear that $\| u(t,\cdot) \|_{L^2} = \| u_0 \|_{L^2}$ for all $t \in \ro$. 

In the same way we have for $t,\tau \in \ro$
\begin{equation}\label{eq:modspdiff}
\| u(t + \tau,\cdot) - u(t,\cdot)\|_{M_{\sigma,s}}^2
\asymp \sum_{j=1}^\infty \lambda_j^{\frac{2s}{k}} |c_j|^2 \left| e^{- i \lambda_j \tau} - 1 \right|^2. 
\end{equation}
The assumption $u_0 \in M_{\sigma,s}(\rr d)$ and the dominated convergence theorem implies that 
$u(t,\cdot) \in \cB(\ro, M_{\sigma,s}(\rr d) )$. 

From \eqref{eq:solutioncoeff} we see that the operator $A$ acts on the series \eqref{eq:seriessolution} as
\begin{equation*}
A u(t,\cdot) = \sum_{j=1}^\infty \lambda_j ( u(t,\cdot), \fy_j) \fy_j 
= \sum_{j=1}^\infty \lambda_j c_j e^{- i \lambda_j t }  \fy_j. 
\end{equation*}
Since 
\begin{equation*}
\partial_t u(t,x) = - i \sum_{j=1}^\infty \lambda_j c_j e^{-i\lambda_j t } \fy_j (x) 
\end{equation*}
the series \eqref{eq:seriessolution} solves the Cauchy problem \eqref{eq:SchrCP}. 
Finally \eqref{eq:solutionnorm} gives
\begin{equation*}
\| \partial_t u(t,\cdot) \|_{M_{\sigma,s-k}}^2
\asymp \sum_{j=1}^\infty \lambda_j^{\frac{2(s-k)}{k}} | \lambda_j |^2 |c_j|^2 
\asymp \| u_0 \|_{M_{\sigma,s}}^2. 
\end{equation*}
This means that $\partial_t u(t, \cdot) \in L^\infty(\ro, M_{\sigma,s-k}(\rr d) )$, 
and it also gives the information that $\partial_t u + i A u=0$ holds in the space 
$L^\infty(\ro, M_{\sigma,s-k}(\rr d))$. 
As in \eqref{eq:modspdiff} it follows that $\partial_t u(t, \cdot) \in L^\infty(\ro, M_{\sigma,s-k}(\rr d) ) \cap C(\ro, M_{\sigma,s-k}(\rr d) )$
which proves the claim $u(t, \cdot) \in \cB^1 (\ro, M_{\sigma,s-k}(\rr d) )$. 
Finally the uniqueness of the solution is a consequence of \eqref{eq:solutionnorm}.
\qed

As consequences of Theorem \ref{thm:wptheorem1}, Remark \ref{rem:seminormschwartz} 
and \eqref{eq:temperedmodsp}
we get well-posedness in $\cS(\rr d)$ and in $\cS'(\rr d)$ in the following respective senses. 

\begin{corollary}\label{cor:wellposedness}
Let $A$ be a positive operator of the form \eqref{eq:diffoperator} that satisfies \eqref{eq:ellipticity}. 

\begin{itemize}

\item [(i)]
If $u_0 \in \cS(\rr d)$ then there exists a unique solution $u(t, \cdot) \in \cB^1 (\ro, \cS(\rr d))$ of \eqref{eq:SchrCP}. 

\item [(ii)]
If $u_0 \in \cS'(\rr d)$ then there exists $s \in \ro$ such that $u_0 \in M_{\sigma,s}(\rr d)$ and a unique solution 
$u(t, \cdot) \in \cB (\ro, M_{\sigma,s}(\rr d)) \cap \cB^1 (\ro, M_{\sigma,s - k}(\rr d))$ of \eqref{eq:SchrCP}.
\end{itemize}
\end{corollary}

Next we formulate a version of Theorem \ref{thm:wptheorem1} for the inhomogeneous equation
\begin{equation}\label{eq:SchrCPinhomo}
\begin{cases} 
\partial_t u (t,x) + i A u(t,x) = f(t,x) \\ 
u(0,x) = u_0(x) 
\end{cases}, \qquad (t,x) \in (-T,T) \times \rr d 
\end{equation}
where $T > 0$.

\begin{theorem}\label{thm:wptheorem3}
Let $T > 0$ and let $A$ be a positive operator of the form \eqref{eq:diffoperator} that satisfies \eqref{eq:ellipticity}. 
If $s \in \ro$, $f(t,\cdot) \in C((-T,T), M_{\sigma,s}(\rr d))$ and $u_0 \in M_{\sigma,s}(\rr d)$, 
then there exists a unique solution $u(t,\cdot) \in C((-T,T), M_{\sigma,s}(\rr d))\cap C^1((-T,T), M_{\sigma,s-k}(\rr d))$  
of \eqref{eq:SchrCPinhomo}. 
\end{theorem}

\begin{proof}
Let $c_j = ( u_0, \fy_j)$ and $f_j (t) = ( f(t,\cdot), \fy_j ) \in C((-T,T), \co)$ for $j \geqs 1$. 
Define for $t \in (-T,T)$ the series 
\begin{equation}\label{eq:seriessolution2}
u(t,x) = \sum_{j=1}^\infty e^{ - i\lambda_j t } \left( c_j +  \int_0^t f_j(\tau) e^{i\lambda_j \tau} \, \dd \tau \right) \fy_j(x)
\end{equation}
which satisfies $u(0,\cdot) = u_0$. 
Consider the coefficient functions for $j \geqs 1$
\begin{equation}\label{eq:funccoeff}
u_j(t) = ( u(t,\cdot), \fy_j )
= e^{ - i\lambda_j t } \left( c_j +  \int_0^t f_j(\tau) e^{i\lambda_j \tau} \, \dd \tau \right). 
\end{equation}

From \eqref{eq:funccoeff} we get 
\begin{equation*}
|u_j(t)| \leqs |c_j | +  \int_{-|t|}^{|t|} | f_j(\tau) | \, \dd \tau
\end{equation*}
and the Cauchy--Schwarz inequality implies
\begin{equation*}
|u_j(t)|^2 \leqs 2 |c_j |^2 +  4 |t|  \int_{-|t|}^{|t|} | f_j(\tau) |^2 \, \dd \tau. 
\end{equation*}
The assumptions $f(t, \cdot) \in C((-T,T), M_{\sigma,s}(\rr d))$, $u_0 \in M_{\sigma,s}(\rr d)$, and 
Proposition \ref{prop:modspacenorm}
then yield $ u(t,\cdot) \in L^\infty( (-T,T), M_{\sigma,s}(\rr d))$. 

From \eqref{eq:funccoeff} we obtain for $t, \nu \in (-T,T)$ and $j \geqs 1$
\begin{align*}
& u_j(t + \nu) - u_j (t) \\
& = e^{ - i\lambda_j t } \left( \left( e^{ - i\lambda_j \nu }-1 \right) \left( c_j +  \int_0^t f_j(\tau) e^{i\lambda_j \tau} \dd \tau \right) 
+ e^{ - i\lambda_j \nu } \int_t^{t+\nu} f_j(\tau) e^{i\lambda_j \tau} \dd \tau \right)
\end{align*}
which gives
\begin{equation*}
|u_j(t + \nu) - u_j (t)| 
\leqs \left| e^{ - i\lambda_j \nu }-1 \right| \left( |c_j| +  \int_{-|t|}^{|t|} |f_j(\tau)| \dd \tau \right) 
+ \int_t^{t+\nu} |f_j(\tau) |  \dd \tau.
\end{equation*}
From this it follows that $u(t, \cdot) \in C( (-T,T), M_{\sigma,s}(\rr d))$. 
Finally we have for all $j \geqs 1$
\begin{align*}
(\partial_t u(t,x) , \fy_j)
& = - i \lambda_j e^{ - i\lambda_j t } \left( c_j +  \int_0^t f_j(\tau) e^{i\lambda_j \tau} \, \dd \tau \right)
+ f_j (t) \\
& = - i ( A u(t,\cdot) + f (t,\cdot), \fy_j). 
\end{align*}
This shows that $u$ is a solution to \eqref{eq:SchrCPinhomo}, 
and due to the continuity $A: M_{\sigma,s}(\rr d) \to M_{\sigma,s-k}(\rr d)$ we have 
$\partial_t u  + i A u = f$ in $C( (-T,T), M_{\sigma,s-k}(\rr d))$. 
In particular we have $\partial_t u \in C((-T,T), M_{\sigma,s-k}(\rr d))$. 

The uniqueness of the solution follows from the observation that 
$u_0 = 0$ and $f = 0$ imply $u = 0$. 
\end{proof}

\begin{corollary}\label{cor:wellposednessinhom}
Let $A$ be a positive operator of the form \eqref{eq:diffoperator} that satisfies \eqref{eq:ellipticity}. 

\begin{itemize}

\item [(i)]
If $u_0 \in \cS(\rr d)$ and $f \in C((-T,T), \mathscr{S}(\rr d))$ then there exists a unique solution $u(t, \cdot) \in C^1 ((-T,T), \cS(\rr d))$ of \eqref{eq:SchrCPinhomo}. 

\item [(ii)]
If $u_0 \in \cS'(\rr d)$ then there exists $s \in \ro$ such that $u_0 \in M_{\sigma,s}(\rr d)$. 
If $f \in C((-T,T), M_{\sigma,s}(\rr d) )$ then there exists a unique solution 
$u(t, \cdot) \in C ((-T,T), M_{\sigma,s}(\rr d)) \cap C^1 ((-T,T), M_{\sigma,s - k}(\rr d))$
of \eqref{eq:SchrCPinhomo}. 
\end{itemize}
\end{corollary}

To prove Theorem \ref{thm:wptheorem2}, that is well-posedness in Gelfand--Shilov spaces for the Cauchy problem \eqref{eq:anharmonicCP}, we need a characterization of the Gelfand--Shilov spaces analogous to the expression 
\eqref{eq:schwartzmodsp} for the Schwartz space
combined with Proposition \ref{prop:modspacenorm} with 
$s \geqs 0$. 
The following result is a consequence of \cite[Theorem~1.3]{CGPR}. 

\begin{proposition}\label{prop:GSseminorm}
Let $k,m \in 2\no \setminus 0$ and define the operator $A$ by \eqref{eq:diffoperator}.
Suppose $A$ is positive and elliptic in the sense of \eqref{eq:ellipticity}, 
let $(\fy_j)_{j \geqs 1} \subseteq \mathcal{S}_{\frac{m}{k+m}}^{\frac{k}{k+m}}(\rr d)$ be its orthonormal basis of eigenfunctions, 
and let $( \lambda_j )_{j \geqs 1} \subseteq \ro_+$ be its eigenvalues. 
Define $\nu$ and $\mu$ by \eqref{eq:muenu} with $r \geqs 1$, let $u \in \cS(\rr d)$ and $u_j = (u, \fy_j)$ for $j \geqs 1$ and set $\rho = \frac{k+m}{k m r}$. 
The following conditions are equivalent: 

\begin{itemize}

\item[(i)] $u \in \mathcal S_\nu^\mu(\rr d)$ ($u \in \Sigma_\nu^\mu(\rr d)$ and $r > 1$)

\item[(ii)] There exists $\alpha > 0$ such that (for every $\alpha > 0$ we have)
\begin{equation*}
\sup_{j \geqs 1} |u_j|^2 e^{ \alpha \lambda_j^\rho} < + \infty.  
\end{equation*}
\end{itemize}
\end{proposition}

\vskip0.2cm

\textit{Proof of Theorem \ref{thm:wptheorem2}.}
We prove the result for initial datum $u_0 \in \mathcal{S}_\nu^\mu (\rr d)$ where $\mu, \nu$ are defined by \eqref{eq:muenu} and $r \geqs 1$. 
The argument for $\Sigma_\nu^\mu (\rr d)$ and $r > 1$ is similar.
The idea of the proof is the same as for Theorem \ref{thm:wptheorem1}. 

Let $u_0 \in \mathcal{S}_\nu^\mu(\rr d)$. 
If $c_j = (u_0, \fy_j)$ then $u_0 = \sum_{j=1}^{+\infty} c_j \fy_j$ where the coefficients satisfy
\begin{equation}\label{eq:GelfandShilovcoeff}
\sup_{j \geqs 1} |c_j|^2 e^{ \alpha \lambda_j^\rho} 
< + \infty
\end{equation}
for some $\alpha > 0$ by Proposition \ref{prop:GSseminorm}. 
Define for $t \in \ro$ the series 
\begin{equation*}
u(t,x) = \sum_{j=1}^\infty c_j e^{ - i\lambda_j t } \fy_j(x). 
\end{equation*}
Then $u(0,\cdot) = u_0$ and $u(t,\cdot)$ solves the Cauchy problem \eqref{eq:SchrCP} as shown in the proof of Theorem \ref{thm:wptheorem1}. 

By Proposition \ref{prop:GSseminorm} it follows as in the proof of Theorem \ref{thm:wptheorem1} that 
$u(t, \cdot) \in \mathcal{B}(\ro, \mathcal{S}_{\nu}^\mu(\rr d))$. 
From $(\partial_t u(t,\cdot), \fy_j) = - i  c_j \lambda_j e^{-i\lambda_j t }$
we obtain for any $\ep > 0$
\begin{equation*}
\sup_{ j \geqs 1 } | (\partial_t u(t,\cdot), \fy_j) |^2 e^{  (\alpha-\ep) \lambda_j^\rho} 
= \sup_{ j \geqs 1 }  \lambda_j^2 | c_j  |^2 e^{ (\alpha-\ep) \lambda_j^\rho} 
< + \infty
\end{equation*}
by \eqref{eq:GelfandShilovcoeff}. 
This implies $\partial_t u(t, \cdot)\in \mathcal{B}(\ro, \mathcal{S}_{\nu}^\mu(\rr d))$. 
\qed

\section{The Hamilton flow for a class of generalized anharmonic oscillators}\label{sec:Hamiltonflow}

In this section we compute the Hamilton flow for 
a class of generalized anharmonic oscillator Hamiltonians in dimension $d=1$, 
and we discuss some consequences. 
Consider the Hamiltonian 
\begin{equation}\label{eq:hamiltonian}
a(x,\xi) = ( x^{2k} + \xi^{2m} )^p, \quad x,\xi \in \ro, 
\end{equation}
where $k,m \in \no \setminus 0$ and $p \in \ro \setminus 0$. 
The Hamilton system is 
\begin{equation}\label{eq:hamilton1}
\left\{
\begin{array}{l}
x' (t) = 2 m p \left( x(t)^{2k} + \xi(t)^{2m} \right)^{p-1} \xi(t)^{2m-1} \\
\xi' (t) = -2 k p \left( x(t)^{2k} + \xi(t)^{2m} \right)^{p-1} x(t)^{2k-1} \\
x(0) = y \\
\xi(0) = \eta 
\end{array}
\right. .
\end{equation}
If $y = \eta = 0$ then the unique solution is $x(t) \equiv 0 \equiv \xi(t)$,
so we may assume that 
\begin{equation}\label{eq:initialdatum}
c := y^{2k} + \eta^{2m} > 0.
\end{equation}
The Hamiltonian is preserved along solution curves \cite{Arnold1} so 
\eqref{eq:hamilton1} simplifies into
\begin{equation}\label{eq:hamilton2}
\left\{
\begin{array}{l}
x' (t) = 2 m p c^{p-1} \xi(t)^{2m-1} \\
\xi' (t) = -2 k p c^{p-1} x(t)^{2k-1} \\
x(0) = y \\
\xi(0) = \eta
\end{array}
\right. .
\end{equation}

The well known Picard--Lindel\"of theorem implies that there exists a unique solution to \eqref{eq:hamilton2} for any $(y,\eta) \in \rr 2 \setminus 0$, 
in an open interval $I = I_{y,\eta} \subseteq \ro$ containing $t = 0$, which is smooth as a function of $t$ \cite{Hartman1}.  
In the sequel we show that the solution is periodic and thus well defined for all $t \in \ro$, 
for all initial data $(y,\eta) \in \rr 2 \setminus 0$. 

We use the function indexed by $k,m \in \no \setminus 0$
\begin{equation}\label{eq:gdef}
g_{k,m}(x) = \int_0^x \left( 1 - t^{2k} \right)^{\frac1{2m} - 1} \dd t. 
\end{equation}

Writing $1 - t^{2k} = (1 - t ^2) ( 1 + t^2 + \cdots + t^{2(k-1)} )$
we see that the singularities of the integrand in \eqref{eq:gdef} 
$t = 1^-$ and at $t = -1^+$ are integrable, since $\frac12 \leqs 1-\frac1{2m} < 1$. 
Thus $g_{k,m}: [-1,1] \to \ro$ is a continuous, odd and strictly increasing function with 
\begin{equation}\label{eq:glimit}
\lim_{x \to 1^-} g_{k,m}(x) = - \lim_{x \to -1^+} g_{k,m}(x) := \tau_{k,m} > 0.  
\end{equation}
It follows that $g_{k,m}^{-1}: [-\tau_{k,m},\tau_{k,m}] \to [-1,1]$ is continuous, odd and strictly increasing. 

In the following lemma we identify the function $g_{k,m}$, find an expression for $\tau_{k,m}$, and we figure out how
$\tau_{k,m}$ and $g_{k,m}$ transform with respect to an interchange of indices. 

\begin{lemma}\label{lem:ellipticfunction}
Let $k,m \in \no \setminus 0$, define $g_{k,m}$ by \eqref{eq:gdef}, 
and let $\tau_{k,m} > 0$ 
be defined by \eqref{eq:glimit}. 
Then 
\begin{equation}\label{eq:incompletebeta}
g_{k,m}(x) = \frac1{2k} B\left( x^{2k}, \frac{1}{2k}, \frac{1}{2m}\right)
\end{equation}
where $B(x,p,q)$ is the incomplete beta function, 
\begin{equation}\label{eq:tau}
\tau_{k,m} = \frac{\Gamma \left( \frac1{2k}\right) \Gamma \left( \frac1{2m}\right)}{2 k \Gamma \left( \frac1{2k} + \frac1{2m} \right)}, 
\end{equation}
$\tau_{m,k} = \frac{k}{m} \tau_{k,m}$, and 
\begin{equation}\label{eq:hexprg}
g_{m,k}(x) = \sgn(x) \frac{k}{m} \left( \tau_{k,m} -  g_{k,m} \left( (1-x^{2m} )^{\frac1{2k}} \right) \right), \quad -1 \leqs x \leqs 1. 
\end{equation}
\end{lemma}

\begin{proof}
After a change of variable we may write if $0 \leqs x \leqs 1$
\begin{equation}\label{eq:ghexpressions}
\begin{aligned}
g_{k,m}(x) & = \frac{1}{2k} \int_0^{x^{2k}} t^{\frac{1}{2k}-1} \left( 1 - t \right)^{\frac{1}{2m}-1} \, \dd t, \\
g_{m,k}(x) & = \frac{1}{2m} \int_0^{x^{2m}} t^{\frac{1}{2m}-1} \left( 1 - t \right)^{\frac{1}{2k}-1} \, \dd t \\
& = \frac{1}{2m} \int_{1-x^{2m}}^1 t^{\frac{1}{2k}-1} \left( 1 - t \right)^{\frac{1}{2m}-1} \, \dd t. 
\end{aligned}
\end{equation}
From this we may draw three conclusions. 
First we have the formula \eqref{eq:incompletebeta}
where $B(x,z,w)$ is the incomplete beta function \cite{Nikiforov1,Olver1} 
\begin{equation*}
B(x,z,w) = \int_0^x t^{z-1} (1-t)^{w-1} \dd t, \quad \re z > 0, \quad \re w > 0. 
\end{equation*}

Secondly we have
\begin{align*}
2 k \tau_{k,m} = 2 k g_{k,m}(1) = B \left( \frac{1}{2k}, \frac{1}{2m}\right), \\
2 m \tau_{m,k} = 2 m g_{m,k} (1) = B \left( \frac{1}{2m}, \frac{1}{2k}\right),
\end{align*}
where the beta function is defined by $B(z,w) = B (1,z,w)$. 
The well known identity \cite[Appendix~A]{Nikiforov1}
\begin{equation*}
B(z,w) = \frac{\Gamma(z) \Gamma(w)}{\Gamma(z+w)} 
\end{equation*}
proves \eqref{eq:tau} as well as $\tau_{m,k} = \frac{k}{m} \tau_{k,m}$. 

Thirdly we obtain from \eqref{eq:ghexpressions}
\begin{equation*}
2 k g_{k,m} \left( (1-x^{2m} )^{\frac1{2k}} \right) + 2 m g_{m,k}(x) = B \left( \frac{1}{2k}, \frac{1}{2m}\right)
\end{equation*}
which proves \eqref{eq:hexprg}. 
\end{proof}

\begin{remark}\label{rem:case11}
The functions $\{ g_{k,m} \}_{k,m \geqs 1}$ generalize $g_{1,1}(x) = \arcsin x$, for which
$\tau_{1,1} = \frac{\pi}{2}$ and $g_{1,1}^{-1}(x) = \sin x$. 
\end{remark}

The following lemma shows that $\wt g_{k,m} := g_{k,m}^{-1}$ with domain $[-\tau_{k,m}, \tau_{k,m}]$, 
can be extended to a smooth function with domain $[-\tau_{k,m}, 3 \tau_{k,m}]$
as $\wt g_{k,m}(x) = g_{k,m}^{-1}(2 \tau_{k,m} - x)$, $x \in [\tau_{k,m}, 3\tau_{k,m}]$, for $k,m \geqs 1$. 
In the case $k = m = 1$ this is the basic facts that 
$\sin x$ is smooth 
and $\sin(\pi-x) = \sin x$.
The oddness of $g_{k,m}^{-1}$ yields the same conclusion for the 
limits as $x \to -\tau_{k,m}^+$.

\begin{lemma}\label{lem:derivativelimit}
Let $k,m \in \no \setminus 0$, define $g_{k,m}$ by \eqref{eq:gdef}, 
and let $\tau_{k,m} > 0$ 
be defined by \eqref{eq:glimit}. 
Then for any $n \in \no$ we have
\begin{align}
& \lim_{x \to \tau_{k,m}^-} \left( g_{k,m}^{-1} \right)^{(2n)} (x) \in \ro, \label{eq:evenderivative} \\
& \lim_{x \to \tau_{k,m}^-} \left( g_{k,m}^{-1} \right)^{(2n+1)} (x) = 0. \label{eq:oddderivative}
\end{align}
\end{lemma}

\begin{proof}
Set $g = g_{k,m}$ and $\tau = \tau_{k,m}$. 
When $n=0$ the limit \eqref{eq:evenderivative} equals one as already noted. 
First we observe that \eqref{eq:gdef} gives for $|x| < 1$
\begin{equation}\label{eq:gprime}
g' (x) = \left( 1 - x^{2k} \right)^{\frac{1}{2m}-1} > 0,
\end{equation}
and from the inverse function theorem we obtain
\begin{equation}\label{eq:ginvprime}
\left( g^{-1} \right)' (g (x) ) = ( g'(x) )^{-1} = \left( 1 - x^{2k} \right)^{1 - \frac{1}{2m}}. 
\end{equation}
Since $x \to 1^-$ $\Longleftrightarrow$ $g(x) \to \tau^-$
this proves \eqref{eq:oddderivative} for $n=0$. 

From \eqref{eq:gprime} and \eqref{eq:ginvprime} and the Chain Rule we get
\begin{align*}
\left( g^{-1} \right)'' (g (x) ) 
& = \left( 1 - x^{2k} \right)^{1 - \frac{1}{2m}} \Big( \left( 1 - x^{2k} \right)^{1 - \frac{1}{2m}} \Big)' \\
& = - 2 k \frac{2m-1}{2m} x^{2k-1} \left( 1 - x^{2k} \right)^{1 - \frac{2}{2m}}. 
\end{align*}

By means of induction we can generalize this formula for $\left( g^{-1} \right)'' (g (x) ) $ into 
\begin{equation}\label{eq:ginvderivativeformula}
\left( g^{-1} \right)^{(n)} (g (x) ) 
= \sum_{j \geqs \frac{n}{2m}}^{n-1} 
p_j(x) \left( 1 - x^{2k} \right)^{j - \frac{n}{2m}}
\end{equation}
where $p_j$ are polynomials, for any $n \geqs 2$. 
Indeed suppose that \eqref{eq:ginvderivativeformula} holds for one $n \geqs 2$. 
By the Chain Rule and \eqref{eq:gprime} we deduce
\begin{align*}
\left( g^{-1} \right)^{(n + 1)} (g (x) ) 
& = \sum_{j \geqs \frac{n}{2m}}^{n-1} 
\left( 1 - x^{2k} \right)^{1 - \frac{1}{2m}} \Big( p_j(x) \left( 1 - x^{2k} \right)^{j - \frac{n}{2m}} \Big)' \\
& = \sum_{j \geqs \frac{n}{2m}}^{n-1} 
p_j' (x) \left( 1 - x^{2k} \right)^{j+1 - \frac{n+1}{2m}} 
+ \sum_{j > \frac{n}{2m}}^{n-1} 
\wt p_j (x) \left( 1 - x^{2k} \right)^{j - \frac{n+1}{2m}} \\
& = \sum_{j \geqs \frac{n}{2m} + 1}^{n} 
p_{j-1}' (x) \left( 1 - x^{2k} \right)^{j - \frac{n+1}{2m}} 
+ \sum_{j > \frac{n}{2m}}^{n-1} 
\wt p_j (x) \left( 1 - x^{2k} \right)^{j - \frac{n+1}{2m}}
\end{align*}
where $\wt p_j$ are polynomials. 

In the first sum we have, since $m \geqs 1$,
$2 m j \geqs n + 2m > n + 1$. 
For all possible $j \in \no$ such that $n + 1 \leqs 2 m j < n + 2m$
we define $p_j = 0$. 
Thus the first sum fits into the formula \eqref{eq:ginvderivativeformula} with $n$ replaced by $n+1$ for some polynomials $p_j$. 
In the second sum we have $2 m j > n$
which implies $2 m j \geqs n+1$, so also the second sum can be absorbed into \eqref{eq:ginvderivativeformula}
with $n$ replaced by $n+1$
and modified polynomials $p_j$. 
This argument shows that \eqref{eq:ginvderivativeformula} 
holds with $n$ replaced by $n+1$ for some polynomials $p_j$. 

We have proved the induction step and we may conclude that \eqref{eq:ginvderivativeformula} 
holds for all $n \geqs 2$. 
Finally we observe that \eqref{eq:evenderivative} and \eqref{eq:oddderivative} for $n \geqs 1$
are immediate consequences of \eqref{eq:ginvderivativeformula}. 
\end{proof}

We will also need a corresponding result for the function
\begin{equation}\label{eq:hdef}
h(x) := \left( 1 - \left( g_{k,m}^{-1} (x) \right)^{2k} \right)^{\frac1{2m}}
\end{equation}
which is well defined for $|x| \leqs \tau_{k,m}$. 
The limits are 
\begin{equation}\label{eq:limitzero}
\lim_{x \to -\tau_{k,m}^+} h(x) = \lim_{x \to \tau_{k,m}^-} h(x) = 0, 
\end{equation}
and the next result says that also the derivatives of $h$ have limits. 
Due to the fact that $h$ is even it suffices to study the limits as $x \to \tau_{k,m}^-$. 

As a consequence of the following lemma the function $\wt h = h$ with domain $[-\tau_{k,m}, \tau_{k,m}]$
can be extended to a smooth function with domain $[-\tau_{k,m}, 3 \tau_{k,m}]$
as $\wt h(x) = - h(2 \tau_{k,m} - x)$, $x \in [\tau_{k,m}, 3\tau_{k,m}]$, for $k,m \geqs 1$. 
In the case $k = m = 1$ this is the smoothness of
$\cos x$ and $\cos(\pi-x) = -\cos x$.

\begin{lemma}\label{lem:derivativelimit2}
Let $k,m \in \no \setminus 0$, define $g_{k,m}$ by \eqref{eq:gdef}, 
let $\tau_{k,m} > 0$ 
be defined by \eqref{eq:glimit}, 
and define $h$ by \eqref{eq:hdef}. 
Then for any $n \in \no$ we have
\begin{align}
& \lim_{x \to \tau_{k,m}^-} h^{(2n)} (x) = 0, \label{eq:evenderivative2} \\
& \lim_{x \to \tau_{k,m}^-} h^{(2n+1)} (x) \in \ro. \label{eq:oddderivative2}
\end{align}
\end{lemma}

\begin{proof}
The limit \eqref{eq:evenderivative2} for $n=0$ is already clear from \eqref{eq:limitzero}. 
By the inverse function theorem and \eqref{eq:gdef} we have 
\begin{equation}\label{eq:invfcnderivative}
\left( g_{k,m}^{-1} (x) \right)' 
= \left( g_{k,m}' ( g_{k,m}^{-1} (x) ) \right)^{-1}
= \left( h(x) \right)^{2 m - 1}
\end{equation}
which gives 
\begin{equation}\label{eq:firstderivative}
\begin{aligned}
h'(x) 
& = - \frac{k}{m} \left( 1 - \left( g_{k,m}^{-1} (x) \right)^{2k} \right)^{\frac1{2m}-1}
\left( g_{k,m}^{-1} (x) \right)^{2k-1}
\left( h(x) \right)^{2 m - 1} \\
& = - \frac{k}{m} \left( g_{k,m}^{-1} (x) \right)^{2k-1}. 
\end{aligned}
\end{equation}
This gives the limit \eqref{eq:oddderivative2} for $n = 0$, 
and combined with \eqref{eq:invfcnderivative} it also yields
\begin{equation*}
h''(x) = - \frac{k(2k-1)}{m} \left( g_{k,m}^{-1} (x) \right)^{2k-2}
\left( h(x) \right)^{2 m - 1}. 
\end{equation*}

By means of induction we can generalize this formula for $h''$ into 
\begin{equation}\label{eq:nderivativeformula}
h^{(n)}(x) 
= \sum_{j \in \no: \ \frac{n-1}{2m} \leqs j \leqs n \frac{2k-1}{2k}}
C_j \left( g_{k,m}^{-1} (x) \right)^{(2k-1)n - 2 k j} \left( h(x) \right)^{2 j m - (n-1)}
\end{equation}
where $C_j \in \ro$ are constants, 
for any $n \geqs 2$. 
Note that 
\begin{equation*}
\left\lceil \frac{n-1}{2m} \right\rceil \leqs \left\lceil \frac{n-1}{2} \right\rceil
= \left\lfloor \frac{n}{2} \right\rfloor \leqs \left\lfloor n \frac{2k-1}{2k} \right\rfloor. 
\end{equation*}

Thus we assume that \eqref{eq:nderivativeformula} holds for a fixed $n \geqs 2$. 
Using \eqref{eq:invfcnderivative} and \eqref{eq:firstderivative} we obtain
\begin{align*}
h^{(n+1)}(x) 
&= \sum_{j \in \no: \ \frac{n-1}{2m} \leqs j < n \frac{2k-1}{2k}}
\wt C_j \left( g_{k,m}^{-1} (x) \right)^{(2k-1)n - 2 k j - 1} \left( h(x) \right)^{2 (j+1) m - n} \\
& \quad + \sum_{j \in \no: \ \frac{n-1}{2m} < j \leqs n \frac{2k-1}{2k} }
C_j' \left( g_{k,m}^{-1} (x) \right)^{(2k-1)(n+1) - 2 k j } \left( h(x) \right)^{2 j m - n} \\
& = \sum_{ j \in \no: \ \frac{n-1}{2m} +1 \leqs j < n \frac{2k-1}{2k} +1}
\wt C_{j-1} \left( g_{k,m}^{-1} (x) \right)^{(2k-1)(n+1) - 2 k j } \left( h(x) \right)^{2 j m - n} \\
& \quad + \sum_{ j \in \no: \ \frac{n-1}{2m} < j \leqs n \frac{2k-1}{2k} } C_j' \left( g_{k,m}^{-1} (x) \right)^{(2k-1)(n+1) - 2 k j} \left( h(x) \right)^{2 j m - n}
\end{align*}
with new constants $\wt C_j, C_j'  \in \ro$. 

Due to $\frac{n-1}{2m} +1 \leqs j$
in the first sum we have 
$n-1 \leqs 2 m j - 2 m$ which implies
$n < n + 2 m - 1 \leqs 2 m j$.
For possible $j \in \no$ such that $n \leqs 2 m j < n + 2 m-1$ we define $\wt C_{j-1} = 0$. 
Due to $j < n \frac{2k-1}{2k} +1$ in the first sum we have 
$(j-1) 2 k < n (2k-1)$ which means $2 k j < (n+1)(2k-1) + 1$, that is $2 k j \leqs (n+1)(2k-1)$.
Thus the first sum fits into \eqref{eq:nderivativeformula} with $n$ replaced by $n+1$ 
for some constants $C_j$. 

Due to $\frac{n-1}{2m} < j$
in the second sum we have 
$n < 2 m j +1$, that is
$n \leqs 2 m j$.
Due to $j \leqs n \frac{2k-1}{2k}$ in the second sum we have 
$2 k j < (n+1) (2k-1)$.
For possible $j \in \no$ such that $n (2k-1) < 2 k j \leqs (n+1) (2k-1)$ we define $C_j' = 0$. 
Hence also the second sum fits into \eqref{eq:nderivativeformula} with $n$ replaced by $n+1$
for some constants $C_j$.

We have now proved the induction step and we may conclude that \eqref{eq:nderivativeformula} 
holds for all $n \geqs 2$. 
Finally \eqref{eq:evenderivative2} and \eqref{eq:oddderivative2} 
are consequences of \eqref{eq:limitzero}, \eqref{eq:nderivativeformula} and 
\begin{equation*}
\lim_{x \to \tau_{k,m}^-} g_{k,m}^{-1} (x) = 1. 
\end{equation*}
\end{proof}

In the next result we may due to \eqref{eq:initialdatum} assume either $\eta \neq 0$ or $y \neq 0$. 
If $\eta \neq 0$ we use the parameters
\begin{equation}\label{eq:solutionparameters1}
\begin{aligned}
T_1 & = (2 m p)^{-1} c^{\frac1{2m} + \frac1{2k} -p}  \left( - \tau_{k,m} - g_{k,m} \left( \sgn (\eta) \, y c^{- \frac1{2k}} \right) \right) < 0 \\
T_2 & = (2 m p)^{-1} c^{\frac1{2m} + \frac1{2k} -p}  \left( \tau_{k,m} - g_{k,m} \left( \sgn (\eta) \, y c^{- \frac1{2k}} \right) \right) > 0 \\
T & = 2(T_2-T_1)  = \frac{2 \tau_{k,m} }{m p} c^{\frac1{2m} + \frac1{2k} -p} > 0,
\end{aligned}
\end{equation}
and the functions
\begin{equation}\label{eq:solutionfunctions1}
\left\{
\begin{array}{l}
f_1(t) = \sgn (\eta) \, c^{\frac1{2k}} g_{k,m}^{-1} \left( 2 m p c^{p - \frac1{2m} - \frac1{2k} } t + g_{k,m} \left( \sgn (\eta) \, y c^{- \frac1{2k}} \right) \right) \\
f_2(t) = \sgn(\eta) c^{\frac1{2m}} \left( 1 - \left( g_{k,m}^{-1} \left( 2 m p c ^{p - \frac1{2m} - \frac1{2k} } t + g_{k,m} \left( \sgn (\eta) \, y c^{- \frac1{2k}} \right) \right) \right)^{2k} \right)^\frac1{2m}
\end{array}
\right. 
\end{equation}
defined for $T_1 \leqs t \leqs T_2$. 

If $y \neq 0$ we use instead the parameters
\begin{equation}\label{eq:solutionparameters2}
\begin{aligned}
T_3 & = (2 k p)^{-1} c^{\frac1{2m} + \frac1{2k} -p}  \left( - \tau_{m,k}+ g_{m,k} \left( \sgn (y) \, \eta c^{- \frac1{2m}} \right) \right) < 0 \\
T_4 & = (2 k p)^{-1} c^{\frac1{2m} + \frac1{2k} -p}  \left( \tau_{m,k} + g_{m,k} \left( \sgn (y) \, \eta c^{- \frac1{2m}} \right) \right) > 0 \\
T & = 2(T_4-T_3)  = \frac{2 \tau_{m,k} }{k p} c^{\frac1{2m} + \frac1{2k} -p} > 0,
\end{aligned}
\end{equation}
and the functions
\begin{equation}\label{eq:solutionfunctions2}
\left\{
\begin{array}{l}
f_3(t) = \sgn(y) \, c^{\frac1{2k}} \left( 1 - \left( g_{m,k}^{-1} \left( 2 k p c^{p - \frac1{2m} - \frac1{2k} } t - g_{m,k} \left( \sgn (y) \, \eta c^{- \frac1{2m}} \right) \right) \right)^{2m} \right)^\frac1{2k} \\
f_4(t) = \sgn (y) \, c^{\frac1{2m}} g_{m,k}^{-1} \left( 2 k p c^{p - \frac1{2m} - \frac1{2k} } t - g_{m,k} \left( \sgn (y) \, \eta c^{- \frac1{2m}} \right) \right)
\end{array}
\right. 
\end{equation}
defined for $T_3 \leqs t \leqs T_4$.

We note that Lemma \ref{lem:ellipticfunction} implies $2(T_4-T_3) = 2(T_2-T_1)$, 
that is the length of the domain for $f_1$ and $f_2$ equals the length 
of the domain for $f_3$ and $f_4$, 
and the notation in \eqref{eq:solutionparameters1} and \eqref{eq:solutionparameters2}
is consistent if $y \neq 0$ and $\eta \neq 0$. 
Moreover we have 
$( f_1(t) )^{2k} + ( f_2(t) )^{2m} = c$ for all $T_1 \leqs t \leqs T_2$
and 
$( f_3(t) )^{2k} + ( f_4(t) )^{2m} = c$ for all $T_3 \leqs t \leqs T_4$
as expected from the next result combined with general facts \cite{Arnold1}.

\begin{remark}\label{rem:identicalsolutions}
Using Lemma \ref{lem:ellipticfunction} it can be confirmed that $f_1(t) = f_3(t)$
and $f_2(t) = - f_4(t)$ when $\max(T_1,T_3) \leqs t \leqs \min (T_2, T_4)$
provided $y \neq 0$ and $\eta \neq 0$. 
\end{remark}

\begin{remark}\label{rem:harmonicoscparameters}
If $k = m = 1$ then by Remark \ref{rem:case11} we have if $\eta \neq 0$ and $T_1 \leqs t \leqs T_2$
\begin{align*}
f_1(t) & = \sgn (\eta) \, c^{\frac12} \sin \left( 2 p c^{p - 1 } t + \arcsin \left( \sgn (\eta) \, y c^{- \frac12} \right) \right) \\
& = \sgn (\eta) \, c^{\frac12} \sin \left( 2 p c^{p - 1} t \right) 
\cos \left( \arcsin \left( \sgn (\eta) \, y c^{- \frac12} \right)  \right)
+ y \cos \left( 2 p c^{p - 1 } t \right) \\
& = \eta \sin \left( 2 p c^{p - 1} t \right) + y \cos \left( 2 p c^{p - 1 } t \right)
\end{align*}
and
\begin{align*}
f_2(t) & = \sgn(\eta) c^{\frac12}
\cos \left( 2 p c ^{p - 1} t + \arcsin \left( \sgn (\eta) \, y c^{- \frac12} \right) \right) \\
& = \sgn(\eta) c^{\frac12} 
\left( \cos \left( 2 p c ^{p - 1} t \right) \cos \left( \arcsin \left( \sgn (\eta) \, y c^{- \frac12} \right) \right) 
- \sin \left( 2 p c ^{p - 1} t \right) \sgn (\eta) \, y c^{- \frac12}
\right) \\
& = 
\eta \cos \left( 2 p c ^{p - 1} t \right) 
- y \sin \left( 2 p c ^{p - 1} t \right). 
\end{align*}

If $y \neq 0$ and $T_3 \leqs t \leqs T_4$ then 
\begin{align*}
f_4(t) & = \sgn (y) \, c^{\frac12} \sin \left( 2 p c^{p - 1} t - \arcsin \left( \sgn (y) \, \eta c^{- \frac12} \right) \right) \\
& = \sgn (y) \, c^{\frac12} 
\left( 
\sin \left( 2 p c^{p - 1} t \right) 
c^{-\frac12} |y|
- \sgn (y) \eta c^{- \frac12} \cos \left( 2 p c^{p - 1 } t \right) \right) \\
& = y \sin \left( 2 p c^{p - 1} t \right) - \eta \cos \left( 2 p c^{p - 1 } t \right) = - f_2 (t)
\end{align*}
and 
\begin{align*}
f_3(t) & = \sgn(y) \, c^{\frac12} 
\cos \left( 2 p c ^{p - 1} t - \arcsin \left( \sgn (y) \, \eta c^{- \frac12} \right) \right)  \\
& = \sgn(y) c^{\frac12} 
\left( \cos \left( 2 p c ^{p - 1} t \right) \cos \left( \arcsin \left( \sgn (y) \, \eta c^{- \frac12} \right) \right) 
+ \sin \left( 2 p c ^{p - 1} t \right) \sgn (y) \, \eta c^{- \frac12}
\right) \\
& = 
y \cos \left( 2 p c ^{p - 1} t \right) 
+ \eta \sin \left( 2 p c ^{p - 1} t \right) = f_1(t). 
\end{align*}

Thus $(f_1, f_2) = (f_3, - f_4)$ 
agrees with the well known harmonic solution to \eqref{eq:hamilton2}
which is linear when $k = m = 1$, and extends to domain $t \in \ro$. 
\end{remark}

\begin{theorem}\label{thm:hamiltonflowperiodic}
Suppose $k,m \in \no \setminus 0$, $p \in \ro \setminus 0$, and let $a$ be defined by \eqref{eq:hamiltonian}. 
If $c = y^{2k} + \eta^{2m} > 0$
then the equation \eqref{eq:hamilton2} has a unique solution $\chi_t (y,\eta) = ( x(t), \xi(t) )$ with domain $\ro$. 
It is smooth and periodic with respect to $t \in \ro$, and has the following form. 

If $\eta \neq 0$ then
\begin{equation}\label{eq:solutioncase1a}
\left\{
\begin{array}{l}
x(t) = f_1 (t) \\
\xi(t) = f_2 (t)
\end{array}
\right. 
, \  T_1 \leqs t \leqs T_2, 
\quad 
\left\{
\begin{array}{l}
x(t) = f_1(2 T_2-t) \\
\xi(t) = - f_2(2 T_2-t)
\end{array}
\right. 
, \  T_2 \leqs t \leqs 2 T_2 - T_1, 
\end{equation}
and 
\begin{equation}\label{eq:solutioncase1b}
\left\{
\begin{array}{l}
x(t) = x(t-nT) \\
\xi(t) = \xi(t-nT)
\end{array}
\right. 
, \quad T_1 + nT \leqs t \leqs T_1 + (n+1) T, \quad n \in \zo. 
\end{equation}

If $y \neq 0$ then 
\begin{equation}\label{eq:solutioncase2a}
\left\{
\begin{array}{l}
x(t) = f_3 (t) \\
\xi(t) = -f_4 (t)
\end{array}
\right. 
, \ T_3 \leqs t \leqs T_4, 
\qquad 
\left\{
\begin{array}{l}
x(t) = - f_3(2 T_4 - t) \\
\xi(t) = - f_4(2 T_4 - t)
\end{array}
\right. 
, \ T_3 \leqs t \leqs 2 T_4 - T_3, 
\end{equation}
and 
\begin{equation}\label{eq:solutioncase2b}
\left\{
\begin{array}{l}
x(t) = x(t- n T) \\
\xi(t) = \xi(t-n T)
\end{array}
\right. 
, \quad T_3 + nT \leqs t \leqs T_3 + (n+1) T, \quad n \in \zo. 
\end{equation}
\end{theorem}

\begin{proof}
Set $b = 2 p c^{p-1} > 0$. 
First we assume $\eta \neq 0$. 
The first equation in \eqref{eq:hamilton2} can be written 
\begin{equation}\label{eq:frequencysolution}
\xi = (m b)^{-\frac1{2m-1}} (x')^{\frac1{2m-1}}. 
\end{equation}
Insertion into the second equation yields
\begin{align*}
- (x^{2k} )'
& = - 2 k x' x^{2k-1}
= \frac{2}{b} x' \xi'
= 2 m^{-\frac1{2m-1}} b^{-\frac{2m}{2m-1}} \frac1{2m-1} (x')^{\frac1{2m-1}} x'' \\
&= (m b)^{-1-\frac1{2m-1}} \frac{2m}{2m-1} (x')^{\frac{2m}{2m-1}-1} x'' \\
&= (m b)^{-\frac{2m}{2m-1}} \left( (x')^{\frac{2m}{2m-1}} \right)'. 
\end{align*}

Integration gives, using $x(0) = y$ and $x'(0) = m b \eta^{2m-1}$, 
\begin{equation*}
y^{2k} - ( x(t) )^{2k} )
= (m b)^{-\frac{2m}{2m-1}} \left( (x'(t))^{\frac{2m}{2m-1}}  - \left( m b \eta^{2m-1} \right)^{\frac{2m}{2m-1}} \right)
\end{equation*}
which can be written as the differential equations
\begin{equation}\label{eq:diffeq1}
x'(t) = \pm m b \left( c - ( x(t) )^{2k} ) \right)^{\frac{2m-1}{2m}}. 
\end{equation}
Since $\left( c - ( x(0) )^{2k} ) \right)^{\frac{2m-1}{2m}} = |\eta|^{2m-1}$
whereas $x'(0) = m b\eta^{2m-1}$ according to \eqref{eq:hamilton2}, 
the correct sign in \eqref{eq:diffeq1} is $\sgn \eta$. (Here we use the assumption $\eta \neq 0$.) 
Thus 
\begin{equation}\label{eq:diffeq2}
x'(t) = \sgn (\eta)  \, m b \left( c - ( x(t) )^{2k} ) \right)^{\frac{2m-1}{2m}}. 
\end{equation}

The solution to the Cauchy problem for \eqref{eq:diffeq2} with $x(0) = y$ is
\begin{equation}\label{eq:spacesolution1}
\begin{aligned}
x(t) & = \sgn (\eta) \, c^{\frac1{2k}} g_{k,m}^{-1} \left( m b c^{1 - \frac1{2m} - \frac1{2k} } t + g_{k,m} \left( \sgn (\eta) \, y \, c^{- \frac1{2k}} \right) \right) \\
& =  c^{\frac1{2k}} g_{k,m}^{-1} \left( \sgn (\eta) \, 2 m p c^{p - \frac1{2m} - \frac1{2k} } t + g_{k,m} \left( y \, c^{- \frac1{2k}} \right) \right)
\end{aligned}
\end{equation}
where we use the fact that $g_{k,m}^{-1}$ is an odd function. 

The solution $\xi(t)$ is obtained from a combination of \eqref{eq:frequencysolution}, 
\eqref{eq:diffeq2} and \eqref{eq:spacesolution1} as
\begin{align*}
\xi (t) & 
= (m b)^{-\frac1{2m-1}} \left( x' (t) \right)^{\frac1{2m-1}}
= \sgn(\eta) \left( c - ( x(t) )^{2k} ) \right)^{\frac{1}{2m}} \\
& = \sgn(\eta) \, c^{\frac1{2m}} \left( 1 - \left( g_{k,m}^{-1} \left( \sgn (\eta) \,  2 m p c^{p - \frac1{2m} - \frac1{2k} } t + g_{k,m} \left( y c^{- \frac1{2k}} \right) \right) \right)^{2k} \right)^\frac1{2m}. 
\end{align*}

In summary the solution to \eqref{eq:hamilton2} is
\begin{equation}\label{eq:solution1}
\left\{
\begin{array}{l}
x(t) = c^{\frac1{2k}} g_{k,m}^{-1} \left( \sgn (\eta) \, 2 m p c^{p - \frac1{2m} - \frac1{2k} } t + g_{k,m} \left( y c^{- \frac1{2k}} \right) \right) \\
\xi(t) = \sgn(\eta) \, c^{\frac1{2m}} \left( 1 - \left( g_{k,m}^{-1} \left(  \sgn (\eta) \, 2 m p c^{p - \frac1{2m} - \frac1{2k} } t + g_{k,m} \left( y c^{- \frac1{2k}} \right) \right) \right)^{2k} \right)^\frac1{2m}
\end{array}
\right. .
\end{equation}

Due to $| y | c^{- \frac1{2k}} < 1$ we have $| g_{k,m} \left( \sgn (\eta) \, y c^{- \frac1{2k}} \right) | < \tau_{k,m}$. 
The solution \eqref{eq:solution1} is hence well defined for $T_1 \leqs t \leqs T_2$. 
Moreover we have the limits
\begin{equation}\label{eq:periodic1}
\begin{aligned}
\lim_{t \to T_2^-} x(t) & = \sgn(\eta) c^{\frac1{2k}}, \\
\lim_{t \to T_1^+} x(t) & = - \sgn(\eta) c^{\frac1{2k}}, \\
\lim_{t \to T_2^-} \xi(t) & = \lim_{t \to T_1^+} \xi(t) = 0. 
\end{aligned}
\end{equation}

Next we consider the differential equation 
\begin{equation}\label{eq:diffeq3}
x'(t) = - \sgn (\eta)  \, m b \left( c - ( x(t) )^{2k} ) \right)^{\frac{2m-1}{2m}}
\end{equation}
with initial conditions $x( 2 T_2 ) = y$ and $\xi( 2 T_2 ) = - \eta$.  
This is consistent with equation \eqref{eq:hamilton2}, since 
according to \eqref{eq:diffeq3} we have $x'( 2 T_2  ) = -  \sgn (\eta)  \, m b \, |\eta|^{2m-1} = m b ( - \eta)^{2m-1}$. 
As above we find the solution 
\begin{equation}\label{eq:solution2}
\left\{
\begin{array}{l}
x(t) =  \sgn (\eta) \, c^{\frac1{2k}} g_{k,m}^{-1} \left( 2 m p c^{p - \frac1{2m} - \frac1{2k} } ( 2T_2  - t) + g_{k,m} \left( \sgn (\eta) \, y \, c^{- \frac1{2k}} \right) \right) \\
\xi(t) = - \sgn(\eta) c^{\frac1{2m}} \left( 1 - \left( g_{k,m}^{-1} \left( 2 m p c^{p - \frac1{2m} - \frac1{2k} } ( 2 T_2  - t) + g_{k,m} \left( \sgn (\eta) \, y \, c^{- \frac1{2k}} \right) \right) \right)^{2k} \right)^\frac1{2m}
\end{array}
\right. 
\end{equation}
which is well defined when $T_2 \leqs t \leqs 2 T_2 - T_1$. 

The limits are
\begin{align*}
\lim_{t \to T_2^+ } x(t) & = \sgn(\eta) c^{\frac1{2k}}, \\
\lim_{t \to ( 2T_2 - T_1)^- } x(t) & = - \sgn(\eta) c^{\frac1{2k}}, \\
\lim_{t \to ( 2 T_2 - T_1)^- } \xi(t) & = \lim_{t \to T_2^+ } \xi(t) = 0. 
\end{align*}
Comparing with \eqref{eq:periodic1} we have now a continuous solution $( x(t), \xi(t) )$ to 
\eqref{eq:hamilton2}
on $[T_1, 2 T_2 - T_1]$.
It can be periodized into the solution 
\begin{equation}\label{eq:sol1}
\left\{
\begin{array}{l}
x(t) = x(t-nT) \\
\xi(t) = \xi(t-nT)
\end{array}
\right. 
, \quad T_1 + nT \leqs t \leqs T_1 + (n+1) T, 
\end{equation}
which is continuous on $t \in \ro$. 
This concludes the proof of the solution \eqref{eq:solutioncase1a}, \eqref{eq:solutioncase1b}
when $\eta \neq 0$.

The proof when $y \neq 0$ is analogous to the proof for $\eta \neq 0$,
except for that we start to solve the second equation in \eqref{eq:hamilton2}
instead of the first. 
The solution \eqref{eq:solutioncase2a}, \eqref{eq:solutioncase2b}
when $y \neq 0$ follows. 

If $\eta \neq 0 \neq y$ then the solutions \eqref{eq:solutioncase1a}, \eqref{eq:solutioncase1b}
and \eqref{eq:solutioncase2a}, \eqref{eq:solutioncase2b} must be identical due to the uniqueness of the solution. 
This agrees with Remark \ref{rem:identicalsolutions}. 

Finally we address the claimed smoothness of the solution $( x(t), \xi(t) )$. 
The function $g_{k,m}^{-1}$ is smooth on $(-\tau_{k,m},\tau_{k,m})$ and $g_{m,k}^{-1}$ is smooth on $(-\tau_{m,k},\tau_{m,k})$. 
Moreover the function \eqref{eq:hdef} is smooth on $(-\tau_{k,m},\tau_{k,m})$, 
and on $(-\tau_{m,k},\tau_{m,k})$ after a swap of indices $k,m$. 
The solution is thus smooth except possibly at the boundary points of the intervals above, 
where the solutions are extended by reflection and periodization. 
There we have only shown that the solution is continuous. 
Nevertheless Lemmas \ref{lem:derivativelimit} and \ref{lem:derivativelimit2}  
show that the solutions are smooth also at the boundary points. 
Alternatively we may appeal to the general result \cite[Theorem~V.4.1]{Hartman1}
to conclude that the solution is smooth also at the boundary points of the intervals. 
The solution is thus smooth with respect to $t \in \ro$ everywhere. 
\end{proof}

\begin{remark}\label{rem:hamiltonflow1}
Using \eqref{eq:tau} we obtain the expression for the period of the solution 
\begin{equation*}
T = \frac{2 \tau_{k,m} }{m p} c^{\frac1{2m} + \frac1{2k} -p}
= \frac{\Gamma \left( \frac1{2k}\right) \Gamma \left( \frac1{2m}\right)}{k m p \Gamma \left( \frac1{2k} + \frac1{2m} \right)} c^{\frac1{2m} + \frac1{2k} - p}.
\end{equation*}
Thus $T = C_{k,m} ( y^{2k} + \eta^{2m} )^{p_c- p}$ where the critical exponent is defined by \eqref{eq:pcrit}
and $C_{k,m} > 0$. 
If $p = p_c$ then $T$ does not depend on the initial data $(y,\eta) \in \rr 2 \setminus 0$. All solution trajectories have the same period. 

If $p < p_c$ then the solution period grows as a function of $y^{2k} + \eta^{2m}$, whereas if $p > p_c$
then the solution period decreases as a function of $y^{2k} + \eta^{2m}$. 
\end{remark}

\begin{remark}\label{rem:hamiltonflow2}
If $p = p_c$ then as expected Theorem \ref{thm:hamiltonflowperiodic} 
is consistent with \cite[Propositions~6.2 and 6.4]{CRW}. 
In fact defining the anisotropic scaling map $\Lambda_\sigma (\lambda): \rr 2  \to \rr 2$ as
\begin{equation*}
\Lambda_\sigma (\lambda) (y,\eta) = ( \lambda y, \lambda^\sigma \eta ), \quad (y,\eta) \in \rr 2, \quad \lambda > 0, 
\end{equation*}
then the Hamilton flow in Theorem \ref{thm:hamiltonflowperiodic} commutes with $\Lambda_\sigma (\lambda)$
if $p = p_c$ and $\sigma = \frac{k}{m}$: 
\begin{equation*}
\chi_t(\Lambda_\sigma (\lambda) (y,\eta) ) 
= \Lambda_\sigma (\lambda) \chi_t( y,\eta ),  
\quad (y,\eta) \in \rr 2 \setminus 0, \quad \lambda > 0, \quad t \in \ro. 
\end{equation*}
By \cite[Proposition 6.2]{CRW} this commutativity holds if the Hamiltonian Weyl symbol satisfies the anisotropic homogeneity
\begin{equation*}
a( \lambda y, \lambda^\sigma \eta ) = \lambda^{1 + \sigma} a( y,\eta ), \quad ( y,\eta ) \in \rr 2 \setminus 0, \quad \lambda > 0, 
\end{equation*}
which is the case for \eqref{eq:hamiltonian} when $p = p_c$. 
\end{remark}

From Theorem \ref{thm:hamiltonflowperiodic} we will extract the behavior of the Hamilton flow 
$(x(t), \xi(t) ) = (x(t, y,\eta), \xi(t, y, \eta) ) = \chi_t ( y,\eta)$ with respect to $(y,\eta) \in \rr 2 \setminus 0$. 
If $\eta \neq 0$ the flow has the form
\begin{equation}\label{eq:flowform1}
\left\{
\begin{array}{l}
x(t, y,\eta) = \sgn(\eta) ( y^{2k} + \eta^{2m} )^{\frac1{2k}} g_1 \left( q(t) ( y^{2k} + \eta^{2m} )^{p - \frac1{2m} - \frac1{2k}}  + h_1 \left( y ( y^{2k} + \eta^{2m} )^{ - \frac1{2k} } \right) \right) \\
\xi(t, y,\eta) = \sgn(\eta) ( y^{2k} + \eta^{2m} )^{\frac1{2m}} g_2 \left( q(t) ( y^{2k} + \eta^{2m} )^{p - \frac1{2m} - \frac1{2k}}  + h_1 \left( y ( y^{2k} + \eta^{2m} )^{ - \frac1{2k} } \right) \right)
\end{array}
\right. 
\end{equation}
and if $y \neq 0$ it has the form
\begin{equation}\label{eq:flowform2}
\left\{
\begin{array}{l}
x(t, y,\eta) = \sgn(y) ( y^{2k} + \eta^{2m} )^{\frac1{2k}} g_1 \left( q(t) ( y^{2k} + \eta^{2m} )^{p - \frac1{2m} - \frac1{2k}}  + h_2 \left( \eta ( y^{2k} + \eta^{2m} )^{ - \frac1{2m} } \right) \right) \\
\xi(t, y,\eta) = \sgn(y) ( y^{2k} + \eta^{2m} )^{\frac1{2m}} g_2 \left( q(t) ( y^{2k} + \eta^{2m} )^{p - \frac1{2m} - \frac1{2k}}  + h_2 \left( \eta ( y^{2k} + \eta^{2m} )^{ - \frac1{2m} } \right) \right)
\end{array}
\right. 
\end{equation}
where $g_j \in \cB^\infty(\ro)$, 
$h_j \in \cB^\infty(I)$ where $I = (-1+\ep, 1 - \ep)$ for any $\ep > 0$, 
$j=1,2$, 
and where the function $q$ has the form $q(t) = \pm a (t-b)$ with $a> 0$ and $b \in \ro$. 
Indeed we can think of $g_1, g_2$ as periodic, and 
$g_1, g_2 \in \cB^\infty(\ro)$ by Lemmas \ref{lem:derivativelimit} and \ref{lem:derivativelimit2}.  
We also have 
$h_1 = g_{k,m}$ and $h_2 = - g_{m,k}$, which both belong to 
$\cB^\infty( (-1+\ep, 1 - \ep) )$ for any $\ep > 0$.

Referring to Remark \ref{rem:cutoffsmooth} we assume in the sequel that 
if $p \in \ro \setminus \no$ in \eqref{eq:hamiltonian} then we multiply 
$a$ with a cutoff function $\psi_\delta \in C^\infty(\rr 2)$ in order to get 
$a(x,\xi) = \psi_\delta (x,\xi) \left( x^{2k} + \xi^{2m} \right)^p \in C^\infty(\rr 2)$. 

First we need a result concerning the composition of a function $f \in \cB^\infty(I)$ and an anisotropic Shubin symbol
with values in $I$.  

\begin{proposition}\label{prop:symbolcomposition}
Let $\sigma > 0$, $r \in \ro$, $\delta > 0$, let $I \subseteq \ro$ be an open interval, 
and let $\Omega \subseteq \rr 2 \setminus 0$ be open. 
Suppose $a: \rr 2 \to I$ satisfies 
\begin{equation*}
\left| \pdd y \alpha \pdd \eta \beta a(y,\eta) \right|
\lesssim \theta_\sigma (y,\eta)^{ r - \alpha - \sigma \beta}, \quad \alpha, \beta \in \no, \quad  (y,\eta) \in \Omega,
\end{equation*}
and suppose $f \in \cB^\infty(I)$.
If $b (y,\eta) = \psi_\delta (y,\eta) f( a(y,\eta) )$ then
\begin{equation*}
\left| \pdd y \alpha \pdd \eta \beta b(y,\eta) \right|
\lesssim \theta_\sigma (y,\eta)^{ r (\alpha+\beta) - \alpha - \sigma \beta}, \quad \alpha, \beta \in \no, \quad  (y,\eta) \in \Omega. 
\end{equation*}
\end{proposition}

\begin{proof}
The claim is obvious if $\alpha = \beta = 0$ 
so we may assume $| \alpha + \beta | \geqs 1$. 
We use Fa\`a di Bruno's formula, which in one variable may be written as
\begin{equation}\label{eq:faadibruno}
\frac{\dd^n}{\dd x^n} \big( f (g(x) ) \big)
= \sum_{m_1 + 2 m_2 + \cdots + n m_n = n} \frac{n!}{m_1! m_2! \cdots m_n!}
f^{(m_1 + \cdots + m_n)} (g(x) ) 
\prod_{j=1}^n \left( \frac{g^{(j)} (x)}{j!} \right)^{m_j}. 
\end{equation}

Derivatives of $\psi_\delta$ are compactly supported which implies that we may ignore this term. 
We obtain if $(y,\eta) \in \Omega$
\begin{equation*}
\begin{aligned}
& \pdd y \alpha \pdd \eta \beta f(a(y,\eta)) 
= \sum_{\gamma_1 + 2 \gamma_2 + \cdots + \alpha \gamma_\alpha = \alpha} \frac{\alpha!}{\gamma_1! \cdots \gamma_\alpha!}
\pdd \eta \beta \left[ f^{(\gamma_1 + \cdots + \gamma_\alpha)} ( a(y,\eta) )  
\prod_{j=1}^\alpha \left( \frac{\pdd y j a (y,\eta) }{j!} \right)^{\gamma_j}
\right] \\
& = \sum_{\substack{\gamma_1 + 2 \gamma_2 + \cdots + \alpha \gamma_\alpha = \alpha \\ \kappa_0 + \kappa_1 \cdots + \kappa_\alpha = \beta}} 
\frac{\alpha! \beta!}{\gamma_1! \cdots \gamma_\alpha! \kappa_0! \cdots \kappa_\alpha!}
\pdd \eta {\kappa_0} \left( f^{(\gamma_1 + \cdots + \gamma_\alpha)} ( a(y,\eta) ) \right)
\prod_{j=1}^\alpha \pdd \eta {\kappa_j} \left( \frac{\pdd y j a (y,\eta) }{j!} \right)^{\gamma_j} \\
& = \sum_{\substack{\gamma_1 + 2 \gamma_2 + \cdots + \alpha \gamma_\alpha = \alpha \\ \kappa_0 + \kappa_1 \cdots + \kappa_\alpha = \beta \\ \tau_1 + 2 \tau_2 + \cdots + \kappa_0 \tau_{\kappa_0} = \kappa_0}} 
\frac{\alpha! \beta! \kappa_0!}
{\gamma_1! \cdots \gamma_\alpha! \kappa_0! \cdots \kappa_\alpha! \tau_1! \cdots \tau_{\kappa_0}!}
f^{(\gamma_1 + \cdots + \gamma_\alpha + \tau_1 + \cdots + \tau_{\kappa_0})} ( a(y,\eta) ) \\
& \qquad \qquad \qquad \qquad \times  
\prod_{k=1}^{\kappa_0} \left( \frac{\pdd \eta k a (y,\eta) }{k!} \right)^{\tau_k}  
\prod_{j=1}^\alpha \pdd \eta {\kappa_j} \left( \frac{\pdd y j a (y,\eta) }{j!} \right)^{\gamma_j}
\end{aligned}
\end{equation*}
which gives the estimates
\begin{equation*}
\begin{aligned}
\left| \pdd y \alpha \pdd \eta \beta b(y,\eta) \right|
& \lesssim  
\sum_{\substack{\gamma_1 + 2 \gamma_2 + \cdots + \alpha \gamma_\alpha = \alpha \\ \kappa_0 + \kappa_1 \cdots + \kappa_\alpha = \beta \\ \tau_1 + 2 \tau_2 + \cdots + \kappa_0 \tau_{\kappa_0} = \kappa_0}} 
\theta_\sigma(y,\eta)^{\sum_{k=1}^{\kappa_0} \tau_k (r-\sigma k) + \sum_{j=1}^\alpha  \gamma_j (r-j) - \sigma \kappa_j} \\
& \lesssim  
\theta_\sigma(y,\eta)^{ r (\alpha+\beta) - \alpha - \sigma \beta }. 
\end{aligned}
\end{equation*}
\end{proof}

\begin{corollary}\label{cor:symbolcomposition}
Let $\sigma > 0$, $\delta > 0$, let $I \subseteq \ro$ be an open interval,  
and let $\Omega \subseteq \rr 2 \setminus 0$ be open. 
Suppose 
\begin{equation*}
\left| \pdd y \alpha \pdd \eta \beta a(y,\eta) \right|
\lesssim \theta_\sigma (y,\eta)^{ - \alpha - \sigma \beta}, \quad \alpha, \beta \in \no, \quad  (y,\eta) \in \Omega,
\end{equation*}
and suppose $f \in \cB^\infty(I)$. 
If $b (y,\eta) = \psi_\delta (y,\eta) f( a(y,\eta) )$ then
\begin{equation*}
\left| \pdd y \alpha \pdd \eta \beta b (y,\eta) \right|
\lesssim \theta_\sigma (y,\eta)^{ - \alpha - \sigma \beta}, \quad \alpha, \beta \in \no, \quad  (y,\eta) \in \Omega.
\end{equation*}
\end{corollary}

From Corollary \ref{cor:symbolcomposition} we obtain in the next result
the behavior of the Hamilton flow with respect to derivatives 
for the Hamiltonian \eqref{eq:hamiltonian}. 
Here the critical power \eqref{eq:pcrit} appears.

\begin{theorem}\label{thm:flowestimates}
Let $k,m \in \no \setminus 0$, $\sigma = \frac{k}{m}$, $\delta > 0$, $p \in \ro \setminus 0$, 
and consider the Hamilton flow $\chi_t: \rr 2 \setminus 0 \to \rr 2 \setminus 0$ for the Hamiltonian \eqref{eq:hamiltonian}, with coordinates $\chi_t(y,\eta) = ( x(t,y,\eta), \xi(t,y,\eta) )$. 
If $\alpha, \beta \in \no$ and $(y,\eta) \in \rr 2$ then 
\begin{equation}\label{eq:hamiltonflowderivatives}
\begin{aligned}
\left| \pdd y \alpha \pdd \eta \beta \left( \psi_\delta(y,\eta) x(t, y,\eta) \right) \right|
& \lesssim \theta_\sigma(y,\eta)^{ 2k \max( p - p_c,0) (\alpha+\beta) + 1 - \alpha - \sigma \beta}, \\
\left| \pdd y \alpha \pdd \eta \beta \left( \psi_\delta(y,\eta) \xi(t, y,\eta) \right) \right|
& \lesssim \theta_\sigma(y,\eta)^{ 2 k \max( p - p_c,0) (\alpha+\beta) + \sigma - \alpha - \sigma \beta}. 
\end{aligned}
\end{equation}
\end{theorem}

\begin{proof}
Let $0 < \ep < 1$ be sufficiently small to guarantee $1 < (1-\ep)^{2k} + (1-\ep)^{2m}$. 
Denoting
\begin{equation*}
C_{\ep,k} : = ( 1 - \ep ) \left( 1 - (1-\ep)^{2k} \right)^{- \frac1{2k}} > 0
\end{equation*}
then $(1-\ep)^{2k} + (1-\ep)^{2m} > 1$ is equivalent to $C_{\ep,k}^{2k} C_{\ep,m}^{2m} > 1$. 
Define for $C > 0$
\begin{align*}
\Omega_{C} & = \{ (y,\eta) \in \rr 2 \setminus 0: |y| < C |\eta|^{\frac{1}{\sigma}} \} \subseteq \rr 2 \setminus 0, \\
\wt \Omega_{C} & = \{ (y,\eta) \in \rr 2 \setminus 0: |y| > C |\eta|^{\frac{1}{\sigma}} \} \subseteq \rr 2 \setminus 0. 
\end{align*}
If $(y,\eta) \in \rr 2 \setminus \left( \Omega_{C_1} \cup \{0\} \right)$
with $C_1 = C_{\ep,k}$ then 
\begin{equation*}
|y| \geqs C_{\ep,k} |\eta|^{\frac{1}{\sigma}}
> C_{\ep,m}^{-\frac{1}{\sigma}} |\eta|^{\frac{1}{\sigma}}
\end{equation*}
so $(y,\eta) \in \wt \Omega_{C_2}$ where $C_2 = C_{\ep,m}^{-\frac{1}{\sigma}}$. 
We have thus shown

\begin{equation}\label{eq:phasespaceunion}
\Omega_{C_1} \cup  \wt \Omega_{C_2} = \rr 2 \setminus 0. 
\end{equation}

If $(y,\eta) \in \Omega_{C_1}$ then $\eta \neq 0$ 
and $|y| ( y^{2k} + \eta^{2m} )^{ - \frac1{2k} } < 1-\ep$ by the design of $C_1 = C_{\ep,k}$. 
Thus we may use \eqref{eq:flowform1} for the Hamilton flow. 
By the proof of \cite[Lemma~3.6]{CRW} we have $\psi_\delta (y,\eta) y ( y^{2k} + \eta^{2m} )^{ - \frac1{2k} } \in G^{0,\sigma}$. 
Since $h_1 \in \cB^\infty( (-1+\ep, 1 - \ep) )$ 
it thus follows from Proposition \ref{prop:symbolcomposition} that 
\begin{equation*}
\left| \pdd y \alpha \pdd \eta \beta \left( h_1 \left( y ( y^{2k} + \eta^{2m} )^{ - \frac1{2k} } \right) \right) \right|
\lesssim \theta_\sigma (y,\eta)^{ - \alpha - \sigma \beta}, \quad \alpha, \beta \in \no, \quad  (y,\eta) \in \Omega_{C_1}.
\end{equation*}

From the proof of \cite[Lemma~3.6]{CRW} we also get $\psi_\delta (y,\eta) ( y^{2k} + \eta^{2m} )^{ p - \frac1{2m} - \frac1{2k} } \in G^{2 p k - \sigma - 1,\sigma}$, 
and \eqref{eq:pcrit} gives $2 p k - \sigma - 1 = 2 k (p-p_c)$.
It follows that 
\begin{align*}
& \left| \pdd y \alpha \pdd \eta \beta 
\left( q( t ) ( y^{2k} + \eta^{2m} )^{p - \frac1{2m} - \frac1{2k}}  + h_1 \left( y ( y^{2k} + \eta^{2m} )^{ - \frac1{2k} } \right) \right)
\right| \\
& \lesssim \theta_\sigma (y,\eta)^{ 2 k \max(p - p_c ,0) - \alpha - \sigma \beta}, \quad \alpha, \beta \in \no, \quad  (y,\eta) \in \Omega_{C_1}.
\end{align*}
Using $g_1, g_2 \in \cB^\infty( \ro)$
we apply Proposition \ref{prop:symbolcomposition} once again and conclude that 
\eqref{eq:hamiltonflowderivatives} holds when $(y,\eta) \in \Omega_{C_1} = \Omega_{C_{\ep,k}}$. 

By \eqref{eq:phasespaceunion} it remains to consider $(y,\eta) \in \wt \Omega_{C_2}$
in which case $y \neq 0$, 
$|\eta| ( y^{2k} + \eta^{2m} )^{ - \frac1{2m} } < 1-\ep$ by the design of $C_2 = C_{\ep,m}^{-\frac{1}{\sigma}}$, 
and we may use \eqref{eq:flowform2}. 
Again by the proof of \cite[Lemma~3.6]{CRW} we have $\psi_\delta (y,\eta) \eta ( y^{2k} + \eta^{2m} )^{ - \frac1{2m} } \in G^{0,\sigma}$. 
Since $h_2 \in \cB^\infty( (-1+\ep, 1 - \ep) )$ it follows from Proposition \ref{prop:symbolcomposition} that 
\begin{equation*}
\left| \pdd y \alpha \pdd \eta \beta \left( h_2 \left( \eta ( y^{2k} + \eta^{2m} )^{ - \frac1{2 m} } \right) \right) \right|
\lesssim \theta_\sigma (y,\eta)^{ - \alpha - \sigma \beta}, \quad \alpha, \beta \in \no, \quad  (y,\eta) \in \wt  \Omega_{C_2}.
\end{equation*}
As above we obtain
\begin{align*}
& \left| \pdd y \alpha \pdd \eta \beta 
\left( q( t ) ( y^{2k} + \eta^{2m} )^{p - \frac1{2m} - \frac1{2k}}  + h_2 \left( \eta ( y^{2k} + \eta^{2m} )^{ - \frac1{2m} } \right) \right)
\right| \\
& \lesssim \theta_\sigma (y,\eta)^{ 2 k \max(p - p_c ,0) - \alpha - \sigma \beta}, \quad \alpha, \beta \in \no, \quad  (y,\eta) \in \wt \Omega_{C_2}.
\end{align*}
Using $g_1, g_2 \in \cB^\infty( \ro)$ we apply Proposition \ref{prop:symbolcomposition} once again and conclude that 
\eqref{eq:hamiltonflowderivatives} holds when $(y,\eta) \in \wt \Omega_{C_2}$. 
\end{proof}

From Theorem \ref{thm:flowestimates} we see that the Hamilton flow behaves like 
anisotropic Shubin symbols, provided the exponent $0 \neq p \leqs p_c$. 

\begin{corollary}\label{cor:flowanisosymbol}
Let $k,m \in \no \setminus 0$, $\sigma = \frac{k}{m}$, $\delta > 0$, $0 \neq p \leqs p_c$, 
and consider the Hamilton flow $\chi_t: \rr 2 \setminus 0 \to \rr 2 \setminus 0$ for the Hamiltonian \eqref{eq:hamiltonian}, with coordinates $\chi_t(y,\eta) = ( x(t,y,\eta), \xi(t,y,\eta) )$,  
$x(t) = x(t,\cdot, \cdot)$ and $\xi(t) = \xi(t,\cdot,\cdot)$. 
Then for any $t \in \ro$ we have $\psi_\delta x(t) \in G^{1,\sigma}$ and $\psi_\delta \xi(t) \in G^{\sigma,\sigma}$. 
\end{corollary}

Next we will use Theorem \ref{thm:flowestimates} in order to study the behavior of a symbol of order zero composed with the inverse Hamilton flow
for the Hamiltonian \eqref{eq:hamiltonian}. 

\begin{proposition}\label{prop:compositionsymbolflow}
Let $k,m \in \no \setminus 0$, $\sigma = \frac{k}{m}$, $\delta > 0$, $p \in \ro \setminus 0$, 
and consider the Hamilton flow $\chi_t: \rr 2 \setminus 0 \to \rr 2 \setminus 0$ for the Hamiltonian \eqref{eq:hamiltonian}.  
If $q \in G^{0,\sigma}(\rr 2)$ and $q(t,y,\eta) = \psi_\delta(y,\eta) q(\chi_t^{-1}(y,\eta))$ 
then we have for every $\alpha, \beta \in \no$ and $t \in \ro$
\begin{equation}\label{eq:compderivative1}
|\pdd y \alpha \pdd \eta \beta q(t,y,\eta)| 
\lesssim \theta_\sigma(y,\eta)^{2 k \max(p - p_c,0) (\alpha+\beta) - \alpha - \sigma \beta}. 
\end{equation}
\end{proposition}

\begin{proof} 
Again we may ignore the term $\psi_\delta$.
We note that the conservation of the Hamiltonian implies 
\begin{equation}\label{eq:weightflow}
\begin{aligned}
\theta_\sigma( \chi_t^{-1} (y,\eta)) 
& = \theta_\sigma \left( x(-t,y,\eta), \xi(-t,y,\eta) \right) \\
& \asymp \left( 1 + |x(-t,y,\eta)|^{2k} + |\xi(-t,y,\eta)|^{2m} \right)^{\frac1{2k}} \\
& = \left( 1 + y^{2k} + \eta^{2m} \right)^{\frac1{2k}}
\asymp \theta_\sigma(y,\eta)
\end{aligned}
\end{equation}
for all $t \in \ro$. 

Set $M = 2k \max( p - p_c,0)$.
We will prove the following generalization of \eqref{eq:compderivative1}.  
For any $\alpha, \beta, \gamma, \kappa \in \no$ we have
\begin{equation}\label{eq:compderivative2}
|\pdd y \alpha \pdd \eta \beta \left( \pdd x \gamma \pdd \xi \kappa q ( \chi_{-t}(y,\eta) ) \right)| 
\lesssim \theta_\sigma(y,\eta)^{ M (\alpha+\beta) - \alpha - \sigma \beta - \gamma - \sigma \kappa}. 
\end{equation}

The validity of \eqref{eq:compderivative2} for $\alpha = \beta = 0$ and all $\gamma, \kappa \in \no$ 
is a consequence of \eqref{eq:weightflow} and 
the assumption $q \in G^{0,\sigma}$. 
Using Theorem \ref{thm:flowestimates} we prove \eqref{eq:compderivative2} by induction with respect to $\alpha + \beta$
starting from $\alpha + \beta = 1$. 
We have
\begin{align*}
& |\partial_y \left( \pdd x \gamma \pdd \xi \kappa q ( \chi_{-t}(y,\eta) ) \right) | \\
& \leqs \left| \pdd x {\gamma+1} \pdd \xi \kappa q ( x(-t,y,\eta), \xi(-t,y,\eta) ) \partial_y x( -t,y,\eta)\right| \\ &+ \left|\pdd x \gamma \pdd \xi {\kappa + 1} q ( x(-t,y,\eta), \xi(t,y,\eta) ) \partial_y \xi( -t,y,\eta) \right| \\
& \lesssim \theta_\sigma(y,\eta)^{- (\gamma +1) - \sigma \kappa + M} + \theta_\sigma(y,\eta)^{- \gamma - \sigma (\kappa+1) + M - 1 + \sigma}
\asymp \theta_\sigma(y,\eta)^{M-1 - \gamma - \sigma \kappa} 
\end{align*}
and 
\begin{align*}
& |\partial_\eta \left( \pdd x \gamma \pdd \xi \kappa q ( \chi_{-t}(y,\eta) ) \right) | \\
& \leqs \left| \pdd x {\gamma+1} \pdd \xi \kappa q ( x(-t,y,\eta), \xi(-t,y,\eta) ) \partial_\eta x( -t,y,\eta) \right| \\ &+\left| \pdd x \gamma \pdd \xi {\kappa+1} q ( x(-t,y,\eta), \xi(-t,y,\eta) ) \partial_\eta \xi( -t,y,\eta) \right| \\
& \lesssim \theta_\sigma(y,\eta)^{- (\gamma +1) - \sigma \kappa + M + 1 -\sigma} + \theta_\sigma(y,\eta)^{- \gamma - \sigma (\kappa+1) + M + \sigma - \sigma }
\asymp \theta_\sigma(y,\eta)^{M-\sigma - \gamma - \sigma \kappa}. 
\end{align*}
Thus \eqref{eq:compderivative2} is true when $\alpha + \beta = 1$.

In the induction step we assume that \eqref{eq:compderivative2} is true for $\alpha + \beta = n \geqs 1$. 
This gives 
\begin{align*}
& \left| \pdd y {\alpha+1} \pdd \eta \beta \left( \pdd x \gamma \pdd \xi \kappa q ( \chi_{-t}(y,\eta) ) \right) \right| \\
& = \left| \pdd y \alpha \pdd \eta \beta \left( \pdd x {\gamma+1} \pdd \xi \kappa q ( x(-t,y,\eta), \xi(-t,y,\eta) ) \right) \partial_y x( -t,y,\eta) \right. \\
& \left. \qquad \qquad \qquad + \pdd y \alpha \pdd \eta \beta \left( \pdd x \gamma \pdd \xi {\kappa + 1} q ( x(-t,y,\eta), \xi(t,y,\eta) ) \right) \partial_y \xi( -t,y,\eta) \right| \\
& \lesssim \theta_\sigma(y,\eta)^{M(\alpha+\beta) - \alpha - \sigma \beta - (\gamma +1) - \sigma \kappa + M} 
+ \theta_\sigma(y,\eta)^{M(\alpha+\beta) - \alpha - \sigma \beta - \gamma - \sigma (\kappa+1) + M + \sigma - 1} \\
& \asymp \theta_\sigma(y,\eta)^{M(\alpha+1+\beta) - (\alpha+1) - \sigma \beta - \gamma - \sigma \kappa} 
\end{align*}
and 
\begin{align*}
& \left| \pdd y \alpha \pdd \eta {\beta+1} \left( \pdd x \gamma \pdd \xi \kappa q ( \chi_{-t}(y,\eta) ) \right) \right| \\
& = \left| \pdd y \alpha \pdd \eta \beta \left( \pdd x {\gamma+1} \pdd \xi \kappa q ( x(-t,y,\eta), \xi(-t,y,\eta) ) \right) \partial_\eta x( -t,y,\eta) \right. \\
& \left. \qquad \qquad \qquad + \pdd y \alpha \pdd \eta \beta \left( \pdd x \gamma \pdd \xi {\kappa + 1} q ( x(-t,y,\eta), \xi(t,y,\eta) ) \right) \partial_\eta \xi( -t,y,\eta) \right| \\
& \lesssim \theta_\sigma(y,\eta)^{M(\alpha+\beta) - \alpha - \sigma \beta - (\gamma +1) - \sigma \kappa + M + 1 - \sigma} 
+ \theta_\sigma(y,\eta)^{M(\alpha+\beta) - \alpha - \sigma \beta - \gamma - \sigma (\kappa+1) + M} \\
& \asymp \theta_\sigma(y,\eta)^{M (\alpha+\beta + 1) - \alpha - \sigma (\beta+1) - \gamma - \sigma \kappa}. 
\end{align*}
This proves the induction step. 
Thus \eqref{eq:compderivative2} holds for all $\alpha, \beta, \gamma, \kappa \in \no$
which implies \eqref{eq:compderivative1} for all $\alpha, \beta \in \no$ and $t \in \ro$.
\end{proof}

\begin{corollary}\label{cor:symbolflowsymbol}
Let $k,m \in \no \setminus 0$, $\sigma = \frac{k}{m}$, $\delta > 0$, $0 \neq p \leqs p_c$, 
and consider the Hamilton flow $\chi_t: \rr 2 \setminus 0 \to \rr 2 \setminus 0$ for the Hamiltonian \eqref{eq:hamiltonian}. 
If $q \in G^{0,\sigma}(\rr 2)$ and $q(t,y,\eta) = \psi_\delta(y,\eta) q(\chi_t^{-1}(y,\eta))$ 
then $q(t, \cdot, \cdot) \in G^{0,\sigma}$ for all $t \in \ro$. 
\end{corollary}

\begin{remark}\label{rem:compflow}
The idea to compose a symbol $q$ with the inverse Hamilton flow $\chi_{-t}$ into $q \circ \chi_t^{-1}$ 
is fundamental in many proofs of propagation of singularities, cf. \cite[Theorem~8.3]{CRW}, \cite[Theorem~23.1.4]{Hormander2} and \cite[Theorem~4.2]{Nicola2}. 
When $p \leqs p_c$ Corollary \ref{cor:symbolflowsymbol} therefore supports Theorem \ref{thm:propanisogabor} in the sense that the time-dependent symbol 
$q(t,\cdot)$ remains in the same symbol class as $q$. 

When on the other hand $p > p_c$ the estimates given by Proposition \ref{prop:compositionsymbolflow}
are not sufficient for the purpose of propagation of the wave front set. 
Indeed $q(t,\cdot)$ is not guaranteed to be a symbol when $t \neq 0$. 
\end{remark}

\section{Proof of Theorem \ref{thm:propanisogabor}}\label{sec:propanisogabor}
 
In the sequel we need a lemma concerning $\sigma$-conic cutoff symbols (cf. \cite[Lemma~3.5]{RW2} and \cite[Lemma~3.3 and Remark~3.4]{Wahlberg1}). 
We use $\sigma$-conic neighborhoods, defined and denoted for $\sigma, \ep > 0$ and $z_0 \in \sr {2d-1}$ as
\begin{equation*}
\Gamma_{z_0, \ep} = \Gamma_{\sigma, z_0, \ep}
= \{ (x,\xi) \in \rr {2d} \setminus 0, \ | z_0 - \pi_\sigma(x,\xi) | < \ep \}
\subseteq  T^* \rr {2d} \setminus 0.
\end{equation*}
Here $\pi_\sigma: \rr {2d} \setminus 0 \to \sr {2d-1}$ is the anisotropic projection
\begin{equation}\label{eq:projection}
\pi_\sigma(x,\xi) = \left( \lambda_\sigma(x,\xi)^{-1} x, \lambda_\sigma(x,\xi)^{-\sigma} \xi \right), \quad (x,\xi) \in \rr {2d} \setminus 0, 
\end{equation}
along the curve $\ro_+ \ni \lambda \mapsto (\lambda x, \lambda^\sigma \xi)$, onto $\sr {2d-1}$, 
and the smooth function $\lambda_\sigma: \rr {2d} \setminus 0 \to \ro_+$ is defined implicitly 
by
\begin{equation*}
\lambda_\sigma (x,\xi)^{-2} |x|^2 + \lambda_\sigma (x,\xi)^{-2 \sigma} |\xi|^2 = 1
\end{equation*}
(cf. \cite[Section~3]{RW2}).

\begin{lemma}\label{lem:cutoffsymbol}
Let $\sigma, r > 0$, $0 < \ep < \delta \leqs 1$ and $z_0 \in \sr {2d-1}$. 
There exists $q \in G^{0,\sigma}$ such that 
$0 \leqs q \leqs 1$,
$\supp q \subseteq \Gamma_{z_0, \delta} \setminus \rB_{r/2}$
and $q |_{\Gamma_{z_0, \ep} \setminus \overline \rB_{r} } \equiv 1$. 
\end{lemma} 
 
\begin{proof}
Let $\fy \in C_c^\infty (\rr {2d})$ satisfy 
$0 \leqs \fy \leqs 1$, 
$\supp \fy \subseteq z_0 + \rB_{\delta}$ 
and $\fy |_{z_0 + \rB_{\ep}} \equiv 1$ \cite[Theorem~1.4.1]{Hormander2}. 
Let $g \in C^\infty(\ro)$ satisfy $0 \leqs g \leqs 1$, $g(x) = 0$ if $x \leqs \frac12$ and $g(x) = 1$ if $x \geqs 1$.  
Set 
\begin{equation}\label{eq:homogen1}
\psi (\lambda x, \lambda^\sigma \xi) = \fy (x,\xi), \quad (x,\xi) \in \sr {2d-1}, \quad \lambda > 0, 
\end{equation}
and 
\begin{equation}\label{eq:cutoff1}
q (z) = g ( r^{-1} |z| ) \psi (z), \quad z \in \rr {2d}. 
\end{equation}

Then \eqref{eq:homogen1} can be written
\begin{equation*}
\psi (x,\xi) = \fy ( \pi_\sigma (x,\xi) ), \quad (x,\xi) \in \rr {2d} \setminus 0, 
\end{equation*}
and it follows that $\psi \in C^\infty( \rr {2d} \setminus 0 )$, 
and thus 
$q \in C^\infty (\rr {2d})$. 
The properties 
$q |_{\Gamma_{z_0, \ep} \setminus \overline \rB_{r} } \equiv 1$ and 
$\supp q \subseteq \Gamma_{z_0, \delta} \setminus \rB_{r/2}$ follow.

If $(x,\xi) \in \rr {2d} \setminus 0$ and $\lambda > 0$ then, since 
$\pi_\sigma ( \lambda x, \lambda^\sigma \xi) = \pi_\sigma (x, \xi)$, we have
\begin{equation*}
\psi (\lambda x, \lambda^\sigma \xi) = \psi( x, \xi )
\end{equation*}
which gives 
\begin{equation}\label{eq:homogeneous3}
(\pdd x \alpha \pdd \xi \beta \psi) (\lambda x, \lambda^\sigma \xi)
= \lambda^{-|\alpha| - \sigma |\beta|}  \pdd x \alpha \pdd \xi \beta \psi( x, \xi ), \quad (x,\xi) \in \rr {2d} \setminus 0, 
\quad \lambda > 0, \quad \alpha, \beta \in \nn d. 
\end{equation}

Let $(y,\eta) \in \rr {2d} \setminus \rB_1$.
Then $(y,\eta) = (\lambda x, \lambda^\sigma \xi)$ for a unique $(x,\xi) \in \sr {2d-1}$ and 
$\lambda = \lambda_\sigma (y,\eta) \geqs 1$.
Combining
\begin{equation*}
1 + |y| + |\eta|^{\frac1\sigma} = 1 + \lambda ( |x| + |\xi|^{\frac1\sigma} ) \asymp 1+ \lambda
\end{equation*}
with \eqref{eq:homogeneous3} we obtain for any $\alpha, \beta \in \nn d$
\begin{equation*}
\left| \pdd y \alpha \pdd \eta \beta \psi (y, \eta) \right|
\leqs C_{\alpha,\beta} (1+\lambda)^{ -|\alpha| - \sigma |\beta|} 
\lesssim ( 1 + |y| + |\eta|^{\frac1\sigma} )^{ -|\alpha| - \sigma |\beta|}. 
\end{equation*}
Referring to \eqref{eq:cutoff1} we may conclude that $q \in G^{0,\sigma}$. 
\end{proof}

\vskip0.2cm
 
\textit{Proof of Theorem \ref{thm:propanisogabor}.} 
For the symbol $a$ we have by the proof of \cite[Lemma~3.6]{CRW}
$a \in G^{2 k p, \sigma} \subseteq G^{2 k p_c, \sigma} = G^{1 + \sigma, \sigma}$. 
By $u_0 \in \cS'(\rr d)$, \eqref{eq:temperedmodsp}
and \cite[Corollary~7.12]{CRW} there exists $s_0 \in \ro$ such that $u_0 \in M_{\sigma, s_0}(\rr d)$ and a unique solution $u =u(t,\cdot)\in C(\ro, M_{\sigma,s_0}(\rr d))$
to \eqref{eq:SchrCP}. In particular $u(t,\cdot) \in \cS'(\rr d)$ for all $t \in \ro$. 
If $p = p_c$ then the propagation result \eqref{eq:propcritical} is a particular case of \cite[Theorem~8.3]{CRW}. 
 
It remains to show \eqref{eq:propsubcritical} when $0 \neq p <  p_c$. 
Then $\alpha = 1 + \sigma - 2 k p = 2 k (p_c - p) > 0$. 
By the proof of \cite[Lemma~3.6]{CRW} we have $a \in G^{2 k p,\sigma} = G^{1 + \sigma - \alpha,\sigma}$.

Suppose $z_0 \notin \WFgs u_0$. 
By the proof of \cite[Proposition~3.5]{Wahlberg1} there exists $\gamma, r > 0$ 
and $q \in G^{0,\sigma}$
such that 
$0 \leqs q \leqs 1$,
$\supp q \subseteq \Gamma_{z_0, 2 \gamma} \setminus \rB_{r/2}$, 
$q |_{\Gamma_{z_0, \gamma} \setminus \overline \rB_{r} } \equiv 1$, 
and $q^w u_0 \in \cS(\rr d)$. 
From $u \in C(\ro, M_{\sigma,s_0}(\rr d))$ and Theorem \ref{thm:PseudoShubinSobolev} we obtain
\begin{equation}\label{eq:modspace2}
q^w u \in C( \ro , M_{\sigma,s_0}(\rr d)). 
\end{equation} 

We use \eqref{eq:modspace2} as the starting point for an iterative procedure. 
Namely we first deduce from \eqref{eq:modspace2} the improved regularity
\begin{equation}\label{eq:regularization1}
q^w u \in C( \ro , M_{\sigma,s_0 + \alpha}(\rr d)). 
\end{equation} 

In fact, cf. \cite[Theorem 23.1.4]{Hormander2} and \cite{Nicola1}, we may 
regard $v = q^w u$ as the solution of the non-homogeneous Cauchy problem 
\begin{equation}\label{eq:SchrodEqIter1}
\begin{cases} 
P v = f \\ 
v(0,\cdot) = v_0
\end{cases}
\end{equation}
where $P = \partial_t + i a^w(x,D)$
and $v_0 = q^w u_0 \in \cS(\rr d) \subseteq M_{\sigma, s_0 + \alpha}(\rr d)$. 
To express $f$ we write, using $P u = 0$,
\begin{equation*}
f = P v = P q^w u = P q^w u - q^w Pu.
\end{equation*}
For the commutator we have 
$P q^w - q^w P = [ P, q^w] = [ i a^w, q^w ]$. 
By \cite[Theorem~18.5.4]{Hormander2} (cf. the proof of \cite[Proposition~8.2]{CRW}) 
the commutator has therefore symbol asymptotic expansion
\begin{equation}\label{eq:commutatorsymbol1}
b = i \left( a \wpr  q - q \wpr a \right) 
\sim \sum_{j=0}^{+\infty} \frac{(-1)^{j}}{(2j + 1)! 2^{ 2 j }} \, \{ a, q \}_{2j+1}
\end{equation}
where 
\begin{equation*}
\{ a, q \}_{j} (x,\xi) 
= (-1)^j \left( \la \partial_x, \partial_\eta \ra - \la \partial_y, \partial_\xi \ra  \right)^{j} 
a(x,\xi) q (y,\eta) \Big|_{(y,\eta) = (x,\xi)}
\end{equation*}
is a bilinear differential operator which extends the Poisson bracket $\{ a, q \}_{1} = \{ a, q \}$ to higher order differential operators, 
cf. \eqref{eq:Poissonbracket}. 

From $a \in G^{1 + \sigma - \alpha,\sigma}$ and $q \in G^{0,\sigma}$ we get 
\begin{equation*}
\{ a, q \} = 
\la \nabla_\xi a, \nabla_x q \ra - \la \nabla_x a, \nabla_\xi q \ra
\in G^{1 - \alpha -1,\sigma} +G^{\sigma - \alpha - \sigma ,\sigma} = G^{-\alpha,\sigma}
\end{equation*}
which by \eqref{eq:commutatorsymbol1} implies $b \in G^{-\alpha,\sigma}$. 
From $v \in C( \ro , M_{\sigma,s_0}(\rr d))$ and Theorem \ref{thm:PseudoShubinSobolev} we obtain
\begin{equation*}
f  = P v = b^w u \in C( \ro , M_{\sigma,s_0 + \alpha}(\rr d) ). 
\end{equation*}
Now \eqref{eq:SchrodEqIter1}, $v_0 \in M_{\sigma, s_0 + \alpha}(\rr d)$ and \cite[Corollary~7.10]{CRW}
show that \eqref{eq:regularization1} holds. 

In the next step we take $0 < \ep < \delta \leqs \gamma$
and define $q_{\ep, \delta} \in G^{0,\sigma}$ as in Lemma \ref{lem:cutoffsymbol}. 
Then $\supp q_{\ep, \delta} \cap \supp (1 - q ) \subseteq \rr {2d}$ is compact. 
By the calculus this yields $r = q_{\ep, \delta} \wpr ( 1 - q ) \in \cS(\rr {2d})$ which combined with
$q^w u_0 \in \cS$
and
Theorem \ref{thm:PseudoShubinSobolev} 
implies
\begin{equation}\label{eq:schwartzmodspace3}
q_{\ep, \delta}^w u_0 = q_{\ep, \delta}^w q^w u_0 + r^w u_0 \in M_{\sigma,s}(\rr d) \qquad \forall s \in \ro
\end{equation}
since $r^w: \cS' \to \cS$ is regularizing. 

Set $w = q_{\ep,\delta}^w u$ and consider the non-homogeneous Cauchy problem 
\begin{equation}\label{eq:SchrodEqIter2}
\begin{cases} 
P w = g \\ 
w(0,\cdot) = w_0
\end{cases} 
\end{equation}
where $w_0 = q_{\ep,\delta}^w u_0 \in M_{\sigma, s_0 + 2 \alpha}(\rr d)$ by 
\eqref{eq:schwartzmodspace3}. 
As above we obtain
\begin{equation*}
g = P q_{\ep,\delta}^w u - q_{\ep,\delta}^w Pu = 
b_{\ep,\delta}^w u
= b_{\ep,\delta}^w q^w u + b_{\ep,\delta}^w (1 - q)^w u
\end{equation*}
with $b_{\ep,\delta} \in G^{-\alpha,\sigma}$ defined by \eqref{eq:commutatorsymbol1} with $q$ replaced by $q_{\ep,\delta}$.

Since $\supp b_{\ep, \delta} \subseteq \Gamma_{z_0,\delta}$ we may conclude that $\supp b_{\ep, \delta} \cap \supp (1 - q ) \subseteq \rr {2d}$
is compact,
which implies that $b_{\ep,\delta}^w (1 - q)^w:  \cS' \to \cS$ is regularizing.  
Now \eqref{eq:regularization1} and again Theorem \ref{thm:PseudoShubinSobolev} 
give $g \in C( \ro , M_{\sigma,s_0 + 2 \alpha}(\rr d) )$. 
By 
\eqref{eq:SchrodEqIter2}, $w_0 \in M_{\sigma, s_0 + 2 \alpha}(\rr d)$ and \cite[Corollary~7.10]{CRW}
we obtain
$w = q_{\ep,\delta}^w u \in C( \ro , M_{\sigma,s_0 + 2 \alpha}(\rr d) )$. 

This bootstrap argument shows that we may conclude
\begin{equation*}
q_{\ep',\delta'}^w u \in C( \ro , M_{\sigma,s_0 + k \alpha}(\rr d) )
\end{equation*}
for any $k \in \no$, in each step decreasing $\ep'$ and $\delta'$ such that $0 < \ep' < \delta' \leqs \ep$
where $\ep$ has been chosen in the preceding step. 
We may do this keeping $\ep' > \delta_0 > 0$ for a fixed $\delta_0 > 0$. 
If $0 < \ep_0 < \delta_0$ then for any $t \in \ro$
\begin{equation*}
q_{\ep_0,\delta_0}^w u(t,\cdot) \in \bigcap_{s \in \ro} M_{\sigma,s}(\rr d) = \cS(\rr d). 
\end{equation*}
By \cite[Proposition~3.5]{Wahlberg1} this shows that
$z_0 \notin \WFgs (u(t, \cdot))$. 
Thus we have shown $\WFgs (u(t,\cdot)) \subseteq \WFgs (u_0)$ for any $t \in \ro$, 
and the reverse inclusion is immediate since the solution operator 
is invertible and $\cK_t^{-1} = \cK_{-t}$. 
We have proved \eqref{eq:propsubcritical}. 
\qed

\section{The filter of singularities}\label{sec:filtersing}

Let 
\begin{equation}\label{eq:Sigmacondition}
\Sigma \subseteq \ro_+ \setminus (0,1],
\end{equation}
let $k,m \in \no \setminus 0$ and let $\sigma = \frac{k}{m}$. 
We define the anisotropic annular region $\wt \Sigma$ with base $\Sigma$ as
\begin{equation}\label{eq:annulardef}
\wt \Sigma =  \{(x,\xi) \in \rr {2d}: |x|^{2k} + |\xi|^{2m}  \in \Sigma \} 
\subseteq \rr {2d}. 
\end{equation}
Thus \eqref{eq:Sigmacondition} implies $\wt \Sigma \subseteq \rr {2d} \setminus \rB_{k,m}$
where 
\begin{equation*}
\rB_{k,m} =  \{(x,\xi) \in \rr {2d}: |x|^{2k} + |\xi|^{2m}  \leqs 1 \} 
\subseteq \rr {2d}.
\end{equation*}

Given $\wt \Sigma \subseteq \rr {2d}$ we say that $a \in G^{r,\sigma}$ is elliptic in $\wt \Sigma$ if 
\begin{equation}\label{eq:ellipticLambda}
|a(x,\xi)| \geqs C \, \theta_\sigma (x,\xi)^r \asymp \left( 1 + |x|^{2k} + |\xi|^{2m} \right)^{\frac{r}{2 k}}, 
\qquad (x,\xi) \in \wt \Sigma \setminus \rB_R,
\end{equation}
for some $C, R >0$ independent of $(x,\xi) \in \wt \Sigma$. 

For a given subset $\Sigma \subseteq \ro$ we define for $\ep > 0$
\begin{equation}\label{eq:LambdaEp}
\Sigma_\ep = \bigcup_{y \in \Sigma} \rB_{\ep |y|} (y). 
\end{equation}
Thus if $x \in \ro \setminus \Sigma_\ep$ then $|x - y| \geqs \ep |y|$ for all $y \in \Sigma$. 

Next we show that if $a \in G^{r,\sigma}$ is elliptic in the annular set $\wt \Sigma \subseteq \rr {2d}$
defined by \eqref{eq:Sigmacondition} and \eqref{eq:annulardef}
then it is automatically elliptic in the larger set $\wt \Sigma_\rho = \wt{(\Sigma_\rho)} \subseteq \rr {2d}$ for some $\rho > 0$, 
where 
\begin{equation*}
\wt \Sigma_\rho =  \{(x,\xi) \in \rr {2d}: |x|^{2k} + |\xi|^{2m} \in \Sigma_\rho \} 
\subseteq \rr {2d}.
\end{equation*}

\begin{proposition}\label{prop:ellipticlarger}
Let $a \in G^{r,\sigma}$ and suppose that $a$ is elliptic in an annular set $\wt \Sigma \subseteq \rr {2d}$
defined by \eqref{eq:Sigmacondition} and \eqref{eq:annulardef}. 
Then there exists $\rho > 0$ such that $a$ is elliptic in $\wt \Sigma_\rho$. 
\end{proposition}

\begin{proof}
Let $0 < \rho < 1$. If $(x,\xi) \in \wt \Sigma_\rho$ then there exists $t \in \Sigma$ such that $| |x|^{2k} + |\xi|^{2m} - t | < \rho t$, 
and $|x|^{2k} + |\xi|^{2m} > 0$. 
Define $y = \lambda^{\frac1{2k}} x$ and $\eta = \lambda^{\frac1{2m}} \xi$ where 
\begin{equation*}
\lambda = \frac{t}{|x|^{2k} + |\xi|^{2m}} > 0. 
\end{equation*}
Then $|y|^{2k} + |\eta|^{2m} = \lambda \left(|x|^{2k} + |\xi|^{2m} \right) = t$ which gives 
\begin{equation*}
(1-\rho) \left( |y|^{2k} + |\eta|^{2m} \right) < |x|^{2k} + |\xi|^{2m}
< (1 + \rho) \left( |y|^{2k} + |\eta|^{2m} \right).
\end{equation*}
Thus $|y|^{2k} + |\eta|^{2m} \asymp |x|^{2k} + |\xi|^{2m}$ which implies $\theta_\sigma(x,\xi) \asymp \theta_\sigma(y,\eta)$. 
	
By the assumed ellipticity we have, since $(y,\eta) \in \wt \Sigma$, 
\begin{equation}\label{eq:elliptic1}
| a(y,\eta) | \geqs C_1 \, \theta_\sigma (y,\eta)^r \geqs C_2 \, \theta_\sigma (x,\xi)^r
\end{equation}
with $C_1, C_2 > 0$ if $|(x,\xi)| \geqs R$ for some $R > 0$. 
	
We have 
\begin{equation}\label{eq:lambdarho}
|\lambda - 1 | = \frac{ \left| t - \left( |x|^{2k} + |\xi|^{2m} \right)\right|}{|x|^{2k} + |\xi|^{2m}} 
< \frac{\rho t}{|x|^{2k} + |\xi|^{2m}} 
< \frac{\rho}{1-\rho}. 
\end{equation}
Since we will choose $\rho > 0$ small, $\lambda$ will be close to 1, so we may assume that 
\begin{equation*}
\frac12 \leqs \lambda^{\frac1{2k}} \leqs 2, \quad \frac12 \leqs \lambda^{\frac1{2m}} \leqs 2, 
\end{equation*}
which implies for any $0 \leqs \tau \leqs 1$
\begin{equation*}
\frac12 \leqs \tau + (1-\tau) \lambda^{\frac1{2k}} \leqs 2, \quad
\frac12 \leqs \tau + (1-\tau) \lambda^{\frac1{2m}} \leqs 2. 
\end{equation*}
In turn this gives 
\begin{equation*}
\theta_\sigma \left( \tau (x,\xi) + (1-\tau) (y,\eta) \right) 
\asymp \theta_\sigma (x,\xi). 
\end{equation*}

Combined with \eqref{eq:anisosymbolest} we get, using the mean value theorem, for some $0 \leqs \tau \leqs 1$, 
\begin{align*}
& \left| a(x,\xi) - a(y,\eta) \right| \\
& \leqs \left| \la \nabla_x a \big( \tau (x,\xi) + (1-\tau) (y,\eta) \big) , x - y \ra + \la \nabla_\xi a \big( \tau (x,\xi) + (1-\tau) (y,\eta) \big) , \xi - \eta \ra  \right| \\
& \leqs C \left| \lambda^{\frac1{2k}} - 1 \right| \theta_\sigma (x,\xi)^{r-1} |x|
+ C  \left| \lambda^{\frac1{2m}} - 1 \right|  \theta_\sigma (x,\xi)^{r-\sigma} | \xi | \\
& \leqs C \max \left( \left| \lambda^{\frac1{2k}} - 1 \right|, \left| \lambda^{\frac1{2m}} - 1 \right| \right) \theta_\sigma (x,\xi)^r
\end{align*}
where $C > 0$ depends on $a$ and $\rho$ only. 
Taking into account \eqref{eq:elliptic1} we obtain finally 
\begin{align*}
\left| a(x,\xi) \right| 
\geqs 
& \left| a(y,\eta) \right| -  \left| a(x,\xi) - a(y,\eta) \right| \\
& \geqs \left( C_2 - C \max \left( \left| \lambda^{\frac1{2k}} - 1 \right|, \left| \lambda^{\frac1{2m}} - 1 \right| \right) \right) \theta_\sigma (x,\xi)^r
\geqs \frac12 C_2 \theta_\sigma (x,\xi)^r
\end{align*}
provided $\rho > 0$ is sufficiently small, using \eqref{eq:lambdarho}, and provided $|(x,\xi)| \geqs R$. 
We have shown that $a$ is elliptic in $\wt \Sigma_\rho$. 
\end{proof}

In the sequel we develop a few results that allow us to use suitable cut-off symbols 
in the definition of ellipticity. 

\begin{lemma}\label{lem:separation}
Suppose $\Sigma \subseteq \ro_+$ satisfies \eqref{eq:Sigmacondition}, 
and suppose $0 < \ep < \delta < 1$. 
Then there exist $\mu > 0$ such that 
\begin{equation}\label{eq:separatednbh}
\left( \ro \setminus \Sigma_\delta \right)_\mu \cap \left( \Sigma_\ep \right)_\mu = \emptyset. 
\end{equation}
\end{lemma}

\begin{proof}
Let $\ep < \gamma < \delta$ and set 
\begin{equation}\label{eq:mudef}
\mu = \min \left( \frac{\delta-\gamma}{1 + \delta}, \frac{\gamma-\ep}{1 + \ep} \right) > 0. 
\end{equation}
Then we have on the one hand
\begin{equation}\label{eq:Sigmainclusion1}
\left( \Sigma_\ep \right)_\mu \subseteq \Sigma_{\gamma}.
\end{equation}
Indeed if $x \in \left( \Sigma_\ep \right)_\mu$ then there exist $y \in \Sigma_\ep$ and $z \in \Sigma$
such that $|x - y| < \mu |y|$ and $|y - z| < \ep z$. 
This gives
\begin{equation*}
\frac{|x-z|}{z}
\leqs \frac{|y-z| + |x-y|}{z}
< \ep + \mu \frac{|y|}{z}
\leqs \ep + \mu \frac{z + |y-z|}{z}
< \ep + \mu (1 + \ep) \leqs \gamma,
\end{equation*}
in the last inequality using $\mu \leqs \frac{\gamma-\ep}{1 + \ep}$
which is due to \eqref{eq:mudef}.
This shows that $x \in \Sigma_\gamma$
which proves \eqref{eq:Sigmainclusion1}.
	
On the other hand we have
\begin{equation}\label{eq:Sigmainclusion2}
\left( \ro \setminus \Sigma_\delta \right)_\mu \subseteq \ro \setminus \Sigma_\gamma.
\end{equation}
To wit if $x \in \left( \ro \setminus \Sigma_\delta \right)_\mu$ then there exists $y \in \ro \setminus \Sigma_\delta$ 
such that $|x - y| < \mu |y|$, 
and $|y - z| \geqs \delta z$ for all $z \in \Sigma$.
	
Let $z \in \Sigma$. 
If $\frac{|y|}{z} \leqs \frac{\delta-\gamma}{\mu}$ then 
\begin{equation*}
\frac{|x-z|}{z}
= \frac{|y-z - (y-x)|}{z}
> \delta - \mu \frac{|y|}{z}
\geqs \delta - \left( \delta-\gamma \right)
= \gamma, 
\end{equation*}
and if $\frac{|y|}{z} > \frac{\delta-\gamma}{\mu}$ then 
\begin{equation*}
\frac{|x-z|}{z}
\geqs \frac{|x|}{z} - 1
= \frac{|y + x- y|}{z} - 1
\geqs \frac{|y|}{z} (1 - \mu) - 1
> (\delta-\gamma) \left( \frac1{\mu} -1 \right) - 1
\geqs \gamma
\end{equation*}
since $\frac1{\mu} - 1 \geqs \frac{1 + \gamma}{\delta - \gamma}$
again due to \eqref{eq:mudef}. 
	
Thus we have $|x-z| \geqs \gamma z$ for all $z \in \Sigma$. 
It follows $x \in \ro \setminus \Sigma_\gamma$ so we have proved \eqref{eq:Sigmainclusion2}. 
Finally \eqref{eq:separatednbh} is a consequence of  
\eqref{eq:Sigmainclusion1} and \eqref{eq:Sigmainclusion2}. 
\end{proof}

\begin{lemma}\label{lem:symbolcutoff}
If $\Sigma \subseteq \ro_+$ satisfies \eqref{eq:Sigmacondition},
and $0 < \ep < \delta < 1$
then there exists $g_{\ep, \delta} \in C^\infty(\ro_+)$ such that 
$0 \leqs g_{\ep , \delta} \leqs 1$, 
$\supp g_{\ep , \delta} \subseteq \Sigma_\delta$, $g_{\ep, \delta}(x)=1$ when $x \in \Sigma_\ep$, and 
\begin{equation}\label{eq:symbolcutoff}
| g_{\ep, \delta}^{(n)} (x)| \leqs c_n (1+|x|)^{-n}, \quad x > 0, \quad n \in \no.
\end{equation}
The constants $c_n > 0$ depend on $\ep$ and $\delta$, but are independent of $x \in \ro_+$. 
\end{lemma}

\begin{proof}
By Lemma \ref{lem:separation} there exists $0 < \mu < 2$
such that \eqref{eq:separatednbh} is valid. 
Let $\fy \in C_c^\infty(\ro)$ satisfy $\fy \geqs 0$, $\supp \fy \subseteq [-\frac12,\frac12]$ and $\int \fy \, \dd x = 1$.  
Define $g_{\ep, \delta} = g$ with 
\begin{equation*}
g(x) = \frac1{\mu x} \int_{\left( \Sigma_\ep \right)_\mu} \fy \left( \frac{x-y}{\mu x} \right) \, \dd y \in C^\infty(\ro_+).
\end{equation*}
Then $g  (x) \geqs 0$ for all $x > 0$, and 
\begin{equation*}
g(x) 
= \frac1{\mu x} \int_{\left( \Sigma_\ep \right)_\mu} \fy \left( \frac{x-y}{\mu x} \right) \, \dd y 
\leqs \int_\ro \fy (y) \, \dd y = 1 
\end{equation*}
for all $x > 0$. 

Let $x \in \ro_+ \setminus \Sigma_\delta$. 
Then $\supp \fy \left( \frac{x-\cdot}{\mu x} \right) \subseteq \left( \ro \setminus \Sigma_\delta \right)_\mu$
so it follows from \eqref{eq:separatednbh} that $g(x) = 0$. Thus $\supp g \subseteq \Sigma_\delta$. 
Next take $x \in \Sigma_\ep$. Then $\supp \fy \left( \frac{x-\cdot}{\mu x} \right) \subseteq \left( \Sigma_\ep \right)_\mu$ 
which implies
\begin{equation*}
g(x) 
= \frac1{\mu x} \int_\ro \fy \left( \frac{x-y}{\mu x} \right) \, \dd y 
= \int_\ro \fy (y) \, \dd y = 1. 
\end{equation*}
It follows that $g\big|_{\Sigma_\ep} \equiv 1$. 
	
It remains to prove the estimates for the derivatives \eqref{eq:symbolcutoff}. 
From 
\begin{equation*}
\frac{\dd^j}{\dd x^j}\left( \frac{1}{\mu x} \right)
= \frac1\mu (-1)^j j! x^{-1-j}, 
\end{equation*}
Fa\`a di Bruno's formula \eqref{eq:faadibruno}, 
and the fact that $(1-\mu/2) x \leqs y \leqs (1+\mu/2) x$ when $y \in \supp \fy \left( \frac{x-\cdot}{\mu x} \right)$
we first deduce the estimate when $x > 0$
\begin{equation}\label{eq:phiderivative}
\begin{aligned}
\left| \frac{\dd^k}{\dd x^k}\left( \fy \left( \frac1\mu -\frac{y}{\mu x} \right) \right) \right|
& \leqs C_k x^{-k} 
\sum_{m_1 + 2 m_2 + \cdots + k m_k = k} 
\left| \fy^{(m_1 + \cdots + m_k)} \left( \frac1\mu -\frac{y}{\mu x} \right) \right|
\left( \frac{y}{x} \right)^{m_1 + \cdots + m_k} \\
& \leqs C_{k,\mu} x^{-k}.
\end{aligned}
\end{equation}

Combined with Leibniz' rule this yields for $x \geqs 1$
\begin{equation}\label{eq:gderivative}
\begin{aligned}
\left| D^n g(x) \right|
& = \left| \sum_{k \leqs n} \binom{n}{k} D^{n-k} \left( \frac1{\mu x} \right) \int_{\left( \Sigma_\ep \right)_\mu} 
D^{k} \left( \fy \left( \frac{x-y}{\mu x} \right) \right) \, \dd y \right| \\
& \leqs \sum_{k \leqs n} \binom{n}{k} \mu^{-1} (n-k)! x^{-1 - n + k} C_{k,\mu} x^{-k}  \left| \supp \fy \left( \frac{x-\cdot}{\mu x} \right) \right| \\
& \leqs C_{n,\mu} x^{-n} 
\end{aligned}
\end{equation}
which proves \eqref{eq:symbolcutoff}. 
\end{proof}

\begin{remark}\label{rem:extension}
The function $g_{\ep, \delta} \in C^\infty(\ro_+)$ may be extended by $g_{\ep, \delta}(x) = 0$ for $x \leqs 0$ into
$g_{\ep, \delta} \in C^\infty(\ro)$. 
Indeed obviously
\begin{equation}\label{eq:rapiddecay}
\sup_{x \in \ro} |x|^n \left| \fy^{(k)} (x) \right| \leqs C_{n,k} < \infty
\end{equation}
for any $n,k \in \no$. 
If $y \in \left( \Sigma_\ep \right)_\mu$ then $y \in \Sigma_\gamma$ by 
\eqref{eq:Sigmainclusion1} which 
together with \eqref{eq:Sigmacondition}
imply $y > 1 - \gamma > 0$. 
If $0 < x < (1-\gamma)/2$ then $|y - x | > (1-\gamma)/2$. 
Thus from \eqref{eq:phiderivative} and \eqref{eq:rapiddecay}
we obtain for any $n,k \in \no$
\begin{align*}
\left| \frac{\dd^k}{\dd x^k}\left( \fy \left( \frac{x-y}{\mu x} \right) \right) \right|
& \lesssim x^{-k} x^{n+k} |x-y|^{- n - k}
\lesssim x^n
\end{align*}
when $0 < x < (1-\gamma)/2$. 
By \eqref{eq:gderivative} this gives $\lim_{x \to 0^+} g_{\ep, \delta}^{(k)}(x) = 0$ for all $k \in \no$, and
finally \cite[Corollary~1.1.2]{Hormander2}
shows that $g_{\ep, \delta} \in C^\infty(\ro)$ and $g_{\ep, \delta}^{(k)} (0) = 0$ for all $k \in \no$. 
\end{remark}

Let $0 < \ep < \delta < 1$ and $k, m \in \no \setminus 0$. 
We define 
\begin{equation}\label{eq:qdef}
q_{\ep, \delta} (x,\xi)= g_{\ep, \delta}(|x|^{2k} + |\xi|^{2m}), \quad x, \xi \in \rr d
\end{equation}
where $g_{\ep, \delta} \in C^\infty(\ro)$ is defined as in Lemma \ref{lem:symbolcutoff}
and Remark \ref{rem:extension}. 
Then $q_{\ep, \delta} \in C^\infty(\rr {2d})$.  

\begin{proposition}\label{prop:qsymbol}
Let $0 < \ep <  \delta < 1$ and $k, m \in \no \setminus 0$, 
and let $\Sigma \subseteq \ro_+$ satisfy \eqref{eq:Sigmacondition}. 
Suppose $g_{\ep, \delta} \in C^\infty(\ro)$ is defined as in Lemma \ref{lem:symbolcutoff} and Remark \ref{rem:extension}, 
and define $q_{\ep, \delta} \in C^\infty(\rr {2d})$ by \eqref{eq:qdef}. 
Then 
$0 \leqs q_{\ep , \delta} \leqs 1$, 
$\supp q_{\ep, \delta} \subseteq \wt \Sigma_\delta$, $q_{\ep, \delta}\big|_{\wt \Sigma_\ep} \equiv 1$, 
and $q_{\ep ,\delta} \in G^{0,\sigma}$ with $\sigma = \frac{k}{m}$. 
\end{proposition}

\begin{proof}
The first three conclusions $0 \leqs q_{\ep , \delta} \leqs 1$, 
$\supp q_{\ep, \delta} \subseteq \wt \Sigma_\delta$ and $q_{\ep, \delta}\big|_{\wt \Sigma_\ep} \equiv 1$
are consequences of \eqref{eq:annulardef} and Lemma \ref{lem:symbolcutoff}. 
	
It remains to show $q = q_{\ep ,\delta} \in G^{0,\sigma}$. An induction argument with respect to $|\alpha + \beta|$ shows that 
for all $\alpha, \beta \in \nn d$ we have
\begin{equation}\label{eq:qderivative}
\pdd x \alpha \pdd \xi \beta q (x,\xi) = \sum_{\frac{|\gamma+\alpha|}{2k} + \frac{|\kappa+\beta|}{2m} \leqs n \leqs |\alpha+\beta|} g_{\ep, \delta}^{(n)} (|x|^{2k} + |\xi|^{2m}) p_{\gamma, \kappa}(x,\xi)
\end{equation}
where $p_{\gamma, \kappa}$ are polynomials of orders $|\gamma|$ and $|\kappa|$ with respect to the variables $x$ and $\xi$, 
respectively, for $\gamma, \kappa \in \nn d$. 
	
For the indices in the sum we thus have 
\begin{equation*}
\max \left( \frac{|\gamma|}{2k},  \frac{|\kappa|}{2m} \right)
\leqs \frac{|\gamma|}{2k} + \frac{|\kappa|}{2m}
\leqs n - \frac{|\alpha|}{2k} - \frac{|\beta|}{2m}. 
\end{equation*}
If $|x| \leqs 1 \leqs |\xi|$ we obtain
\begin{equation*}
|p_{\gamma, \kappa}(x,\xi)|
\lesssim |\xi|^{|\kappa|}
\leqs \left( 1 + |x|^{2k} + |\xi|^{2m} \right)^{n - \frac{|\alpha|}{2k} - \frac{|\beta|}{2m}}, 
\end{equation*}
if $|\xi| \leqs 1 \leqs |x|$ we have 
\begin{equation*}
|p_{\gamma, \kappa}(x,\xi)|
\lesssim |x|^{|\gamma|}
\leqs \left( 1 + |x|^{2k} + |\xi|^{2m} \right)^{n - \frac{|\alpha|}{2k} - \frac{|\beta|}{2m}}, 
\end{equation*}
and if $\min (|\xi|, |x|) \geqs 1$ then  
\begin{equation*}
|p_{\gamma, \kappa}(x,\xi)|
\lesssim |x|^{|\gamma|} |\xi|^{|\kappa|}
\leqs \max \left( |x|^{2k}, |\xi|^{2m} \right)^{\frac{|\gamma|}{2k} + \frac{|\kappa|}{2m}}
\leqs \left( 1 + |x|^{2k} + |\xi|^{2m} \right)^{n - \frac{|\alpha|}{2k} - \frac{|\beta|}{2m}}. 
\end{equation*}

Combining this with \eqref{eq:qderivative} and Lemma \ref{lem:symbolcutoff} yields finally
\begin{equation*}
\left| \pdd x \alpha \pdd \xi \beta q (x,\xi) \right|
\lesssim \left( 1 + |x|^{2k} + |\xi|^{2m} \right)^{-n + n - \frac{|\alpha|}{2k} - \frac{|\beta|}{2m}}
\asymp \theta_\sigma(x,\xi)^{- |\alpha| - \sigma |\beta|}
\end{equation*}
which proves $q \in G^{0,\sigma}$.
\end{proof}

Next we define the concept of smoothness of a tempered distribution in an anisotropic annular subset of phase space $T^* \rr d$. 

\begin{definition}\label{def:smoothannular}
Let $k,m \in \no \setminus 0$, $\sigma = \frac{k}{m}$, and let $\wt \Sigma \subseteq \rr {2d}$ be the anisotropic annular region defined by \eqref{eq:annulardef}
where $\Sigma \subseteq [0,+ \infty)$.
The distribution $u \in \cS'(\rr d)$ is smooth in $\wt \Sigma$ if there exists $r \in \ro$ and $q \in G^{r,\sigma}$, elliptic in $\wt \Sigma$, such that $q^w(x,D) u \in \cS(\rr d)$.
\end{definition}

\begin{remark}\label{rem:smooth}
Two observations:
\begin{enumerate}
\item
If 
$\Sigma \subseteq [0,+ \infty)$ 
and $\sup \Sigma < \infty$ 
then there exists $r > 0$ such that 
$\wt \Sigma \subseteq \rB_r \subseteq \rr {2d}$. 
Let $a \in C_c^\infty(\rr {2d})$ be a cutoff function such that $\supp a \subseteq \rB_{2 r}$
and $a \big|_{\rB_r} \equiv 1$. 
Then $a \in G^{0,\sigma}$, $a^w(x,D) u \in \cS(\rr d)$ for any $u \in \cS'(\rr d)$ and $a$ is elliptic on $\wt \Sigma$. 
It follows that any $u \in \cS'(\rr d)$ is smooth in $\wt \Sigma$ defined by \eqref{eq:annulardef} for any 
$\Sigma \subseteq [0,+ \infty)$
such that $\sup \Sigma < \infty$. 
		
\item
If $u \in \cS(\rr d)$ and $a = 1 \in G^{0,\sigma}$ then $a^w(x,D) u = u \in \cS(\rr d)$ and $a$ is elliptic on $\wt \Sigma$ defined by 
\eqref{eq:annulardef} for any $\Sigma \subseteq [0,+ \infty)$. 
Thus $u \in \cS(\rr d)$ is smooth on all such $\wt \Sigma \subseteq \rr {2d}$. 
\end{enumerate}
\end{remark}

The following result says that smoothness in an annular set can always be defined using a symbol 
as in Proposition \ref{prop:qsymbol}.

\begin{proposition}\label{prop:smoothannular}
Let $\Sigma \subseteq \ro$ satisfy \eqref{eq:Sigmacondition}, and suppose $k, m \in \no \setminus 0$ and $\sigma = \frac{k}{m}$.  
A distribution $u \in \cS'(\rr d)$ is smooth in the annular region $\wt \Sigma$ defined by \eqref{eq:annulardef} if and only if 
there exists $0 < \delta < 1$ such that for any $0 < \ep < \delta$
we have $q_{\ep ,\delta}^w(x,D)u \in \cS(\rr d)$
where $q_{\ep, \delta} \in G^{0,\sigma}$ is defined by \eqref{eq:qdef}, Lemma \ref{lem:symbolcutoff} and Remark \ref{rem:extension}, 
and $q_{\ep ,\delta}$ is elliptic on $\wt \Sigma_\ep$. 
\end{proposition}

\begin{proof}
It suffices to show that smoothness in the annular region $\wt \Sigma$ of $u \in \cS'(\rr d)$ implies the stated condition. 
Thus there exists by assumption $r \in \ro$, $a \in G^{r, \sigma}$ and $C,R > 0$ such that $a^w u \in \cS$ and $|a (x,\xi)| \geqs C \theta_\sigma(x,\xi)^r$
when $(x,\xi) \in \wt \Sigma_\rho \setminus \rB_R$ for some $\rho > 0$, using Proposition \ref{prop:ellipticlarger}. 
	
We use the pseudodifferential calculus for the anisotropic Shubin symbols $G^{r, \sigma}$ \cite{RW2}. 
Defining $a_1 = \overline a \wpr a = |a|^2 + a_2 \in G^{2 r, \sigma}$ we thus have $a_2 \in G^{2 r-(1+\sigma), \sigma}$. 
	
Define $\chi = q_{\gamma, \rho}$ where $\gamma < \rho < 1$ by \eqref{eq:qdef}, Lemma \ref{lem:symbolcutoff} and Remark \ref{rem:extension}. 
Then $\chi \in G^{0, \sigma}$ by Proposition \ref{prop:qsymbol}. 
We set for $C_1 > 0$
\begin{equation*}
b = a_1 \chi + (1 - \chi) C_1 w_{k,m}^{\frac{2r}{k}}. 
\end{equation*}
By \cite[Lemma~3.6]{CRW} we have $b \in G^{2 r, \sigma}$. 
If $z \notin \wt \Sigma_\rho$ then $\chi(z) = 0$ so $b (z) = C_1 w_{k,m}(z)^{\frac{2 r}{k}} \asymp \theta_\sigma(z)^{2 r}$. 
If $z \in \wt \Sigma_\rho \setminus \rB_R$ then for some $C_2 = C_2(C_1) > 0$ and $C_3 > 0$
\begin{align*}
|b (z)| & = \left| a_1(z) \chi (z) + (1 - \chi(z)) C_1 w_{k,m}^{\frac{2 r}{k}}(z)  \right| \\
& \geqs \chi(z) |a(z)|^2 + (1 - \chi(z)) C_1 w_{k,m}^{\frac{2 r}{k}}(z) - \chi(z) |a_2(z)| \\
& \geqs \chi(z) C^2 \theta_\sigma(z)^{2 r} + (1-\chi(z)) C_2 \theta_\sigma(z)^{2 r} - C_3 \theta_\sigma(z)^{2 r -(1+\sigma)}. 
\end{align*}
We can pick $C_1 > 0$ such that $C_2 \leqs C^2$ which gives
\begin{align*}
|b (z)| & \geqs C_2 \theta_\sigma(z)^{2 r} \left( 1 - C_3 C_2^{-1} \theta_\sigma(z)^{-(1+\sigma)} \right) 
\geqs \frac12 C_2 \theta_\sigma(z)^{2 r}
\end{align*}
for $z \in \wt \Sigma_\rho \setminus \rB_R$
after possibly increasing $R > 0$. 

This argument shows that $b  \in G^{2 r, \sigma}$ is an elliptic symbol, and then \cite[Lemma~6.3]{RW2} implies that there exists a parametrix $c \in G^{-2 r, \sigma}$ that satisfies
$c \wpr b = 1 + r$ with $r \in \cS(\rr {2d})$. 
	
Let $\delta \leqs \gamma$ and $\ep < \delta$, and 
define $q = q_{\ep, \delta}$ by \eqref{eq:qdef}, Lemma \ref{lem:symbolcutoff} and Remark \ref{rem:extension}. 
Then $q \in G^{0, \sigma}$ by Proposition \ref{prop:qsymbol}. 
This gives 
\begin{equation*}
q = q \wpr c \wpr a_1 + q \wpr c \wpr (b - a_1) - q \wpr r
\end{equation*}
where $q \wpr r \in \cS(\rr {2d})$. 
Since $b(z) - a_1(z) = 0$ when $z \in \wt \Sigma_{\gamma}$ we have 
$\supp q \cap \supp (b - a_1) = \emptyset$. 
The calculus then implies $q \wpr c \wpr (b - a_1) \in \cS(\rr {2d})$. 
If we set $r_1 = q \wpr c \wpr (b - a_1) - q \wpr r \in \cS(\rr {2d})$,
the continuity of $b_1^w: \cS \to \cS$ for all $b_1 \in G^{t,\sigma}$ and $t \in \ro$ gives
\begin{align*}
q^w u 
& = q^w c^w  a_1^w u + r_1^w u \\
& = q^w c^w  {\overline a}^w a^w u + r_1^w u \in \cS
\end{align*}
due to $r_1 \in \cS(\rr {2d})$ and $a^w u \in \cS$. 
Finally we observe that $q$ is elliptic on $\wt \Sigma_\ep$ by construction. 
\end{proof}

The $\sigma$-anisotropic filter of singularities is defined as follows. 

\begin{definition}\label{def:filtersingularities}
Let $k,m \in \no \setminus 0$, $\sigma = \frac{k}{m}$ and let $u \in \cS'(\rr d)$. 
Consider anisotropic annular subsets $\wt \Sigma \subseteq \rr {2d}$ of the form \eqref{eq:annulardef}
where 
$\Sigma \subseteq [0,+ \infty)$. 
The collection of subsets
\begin{equation*}
\mathcal{F}(u) = \left\{\wt \Sigma \subseteq \rr {2d} : u \ \textrm{is smooth in} \ \rr {2d} \setminus \left( \wt \Sigma \cup \rB_{k,m} \right) \right\} \subseteq \rP(\rr {2d})
\end{equation*}
is called the $\sigma$-anisotropic filter of singularities of $u$.
\end{definition}

\begin{remark}\label{rem:filter}
From Remark \ref{rem:smooth} we may conclude: 
\begin{enumerate}

\item If $\sup \Sigma < \infty$ then $\rr {2d} \setminus \wt \Sigma = \wt{ \ro \setminus \Sigma} \in \mathcal{F}(u)$
for any $u \in \cS'(\rr d)$. In particular $\mathcal{F}(u) \neq \emptyset$ for all $u \in \cS'(\rr d)$. 
		
\item If $u \in \cS(\rr d)$ then $\wt \Sigma \in \mathcal{F}(u)$ for any 
$\wt \Sigma \subseteq \rr {2d}$ defined by \eqref{eq:annulardef}
where 
$\Sigma \subseteq [0,+ \infty)$. 
In particular $\emptyset \in \mathcal{F}(u)$. 
\end{enumerate}
\end{remark}

Remark \ref{rem:filter} (2) means that $\mathcal{F}(u)$ is extremely large when $u \in \cS(\rr d)$. 
The last observation extends to a characterization of $\cS(\rr d)$: 

\begin{lemma}\label{lem:charschwartz}
If $u \in \cS'(\rr d)$ then $\emptyset \in \mathcal{F}(u)$ if and only if $u \in \cS(\rr d)$. 
\end{lemma}

\begin{proof}
In Remark \ref{rem:filter} (2) we have shown that $u \in \cS(\rr d)$ implies $\emptyset \in \mathcal{F}(u)$. 
	
Suppose on the other hand $\emptyset \in \mathcal{F}(u)$. 
Then there exists $a \in G^{r,\sigma}$ such that $a^w(x,D) u \in \cS(\rr d)$ and $a$ is elliptic on $\rr {2d} \setminus \rB_{k,m}$. 
There exists a parametrix $b \in G^{-r,\sigma}$ which satisfies 
$b \wpr a = 1 + r$ where $r \in \cS(\rr {2d})$. 
We conclude
\begin{equation*}
u = b^w a^w u - r^w u \in \cS
\end{equation*}
since $b^w: \cS \to \cS$ is continuous and $r^w: \cS' \to \cS$ is regularizing.
\end{proof}

In Definition \ref{def:filtersingularities} we claim implicitly that $\mathcal F (u) \subseteq \rP(\rr {2d})$ is a 
\textit{filter}, which needs to be proved. 
A filter of subsets of $\rr {2d}$ is denoted by $\mathcal F \subseteq \rP(\rr {2d})$ and defined by the following properties
\cite{Schaefer1}. 

\begin{enumerate}
	
\item $\mathcal F \neq \emptyset$;
		
\item If $F \in \mathcal F$ and $F \subseteq G \subseteq \rr {2d}$ then $G \in \mathcal F$; 
	
\item If $F, G \in \mathcal F$ then $F \cap G \in \mathcal F$.
	
\end{enumerate}

The following result shows that $\mathcal F (u) \subseteq \rP(\rr {2d})$ indeed is a filter. 

\begin{lemma}\label{lem:filter}
If $k,m \in \no \setminus 0$, $u \in \cS'(\rr d)$ and $\mathcal F (u) \subseteq \rP(\rr {2d})$ is defined as in Definition \ref{def:filtersingularities}
then $\mathcal F (u)$ is a filter. 
\end{lemma}

\begin{proof}
In Remark \ref{rem:filter} (1) we have proved that $\mathcal F(u) \neq \emptyset$
for any $u \in \cS'(\rr d)$. Thus property $(1)$ of a filter holds. 
	
Suppose $\wt \Sigma \in \mathcal F (u)$ and $\wt \Sigma \subseteq \wt \Lambda$. 
Then there exists $a \in G^{r,\sigma}$ such that $a^w u \in \cS(\rr d)$ and $a$ is elliptic on 
$\rr {2d} \setminus \left( \wt \Sigma \cup \rB_{k,m} \right)$.
Due to 
$\rr {2d} \setminus \wt \Lambda \subseteq \rr {2d} \setminus \wt \Sigma$
the symbol $a$ is also elliptic on $\rr {2d} \setminus \left( \wt \Lambda \cup \rB_{k,m} \right)$. 
Thus $u$ is smooth in $\rr {2d} \setminus \left( \wt \Lambda \cup \rB_{k,m} \right)$ and hence $\wt \Lambda \in \mathcal F (u)$. 
We have shown property $(2)$ of a filter. 
	
Finally assume $\wt \Sigma, \wt \Lambda \in \mathcal F (u)$. 
By Proposition \ref{prop:smoothannular} there exist 
$a,b \in G^{0,\sigma}$ such that $a^w u \in \cS(\rr d)$, $b^w u \in \cS(\rr d)$, $a, b \geqs 0$, $a$ is elliptic on 
$\rr {2d} \setminus \left( \wt \Sigma \cup \rB_{k,m} \right)$ and $b$ is elliptic on $\rr {2d} \setminus \left( \wt \Lambda \cup \rB_{k,m} \right)$.
Then $a+b \in G^{0,\sigma}$ is elliptic on $\rr {2d} \setminus \left( (\wt \Sigma \cap \wt \Lambda ) \cup \rB_{k,m} \right)$
and $(a+b)^w(x,D) u \in \cS(\rr d)$. It follows $\wt \Sigma \cap \wt \Lambda \in \mathcal F (u)$
which proves property $(3)$ of a filter. 
\end{proof}

\begin{remark}\label{rem:property4}
In \cite{Schaefer1} a filter is assumed to satisfy properties (1), (2), (3), and furthermore 
$\emptyset \notin \mathcal F$. 
We exclude the last property from our definition of filter
in order to allow $u \in \cS(\rr d)$, cf. 
Lemma \ref{lem:charschwartz}.
\end{remark}

\section{Proof of Theorem \ref{thm:propagationannular}}\label{sec:propannular}

The proof is conceptually similar to the proof of Theorem \ref{thm:propanisogabor}. 
Before we start with the proof of Theorem \ref{thm:propagationannular} we state a simple lemma 
concerning the Poisson bracket of $a,b \in C^1(\rr {2d})$ that will be useful. 

\begin{lemma}\label{lem:Poisson}
If $a  \in C^1(\rr {2d})$ and $g \in C^1 (\co)$ then $\{ a, g \circ a \} = 0$. 
\end{lemma}

\begin{proof}
For $1 \leqs j \leqs d$ we have 
\begin{align*}
\partial_{\xi_j} a(x,\xi) \partial_{x_j} \big( g( a (x,\xi) ) \big)
& = \partial_{\xi_j} a(x,\xi) \, g' \big( a(x,\xi) \big) \partial_{x_j} a(x,\xi) \\
& = \partial_{x_j} a(x,\xi) \partial_{\xi_j} \big( g( a (x,\xi) ) \big). 
\end{align*}
\end{proof}

\vskip0.2cm

\textit{Proof of Theorem \ref{thm:propagationannular}.} 
Suppose $\wt \Lambda \subseteq \rr {2d}$ is defined by 
\eqref{eq:annulardef} with $\Lambda \subseteq [0, +\infty)$, 
and suppose $\wt \Lambda \in \mathcal F(u_0)$. 
Then $u_0 \in \cS'(\rr d)$ is smooth in $\wt \Sigma = \rr {2d} \setminus \left( \wt \Lambda \cup \rB_{k,m} \right)$. 
Proposition \ref{prop:smoothannular} implies that 
there exists $q_{\ep,\delta} \in G^{0,\sigma}$ defined by \eqref{eq:qdef}, Lemma \ref{lem:symbolcutoff} and Remark \ref{rem:extension}, 
$q_{\ep,\delta}$ is elliptic on $\wt \Sigma_{\ep}$, 
and $q_{\ep,\delta}^w u_0 \in \cS(\rr d)$,
for some $0 < \delta < 1$ and any $0 < \ep < \delta$. 

From \eqref{eq:schwartzmodsp} we obtain 
\begin{equation}\label{eq:schwartzmodspace1}
q_{\ep,\delta}^w u_0 \in M_{\sigma,s}(\rr d) \qquad \forall s \in \ro,
\end{equation}
and \eqref{eq:temperedmodsp} implies
$u_0 \in M_{\sigma,s_0}(\rr d)$ for some $s_0 \in \ro$. 
In view of Theorem \ref{thm:wptheorem1} this implies that there exists a unique solution 
$u \in C( \ro , M_{\sigma,s_0}(\rr d))$ to \eqref{eq:anharmonicCP}. 
Hence  
\begin{equation}\label{eq:modspace1}
q_{\ep,\delta}^w u \in C( \ro , M_{\sigma,s_0}(\rr d))
\end{equation} 
since $q_{\ep,\delta} \in G^{0,\sigma}$ and we can apply Theorem \ref{thm:PseudoShubinSobolev}. 

Again we use \eqref{eq:modspace1} as the starting point of an iteration. 
We first deduce from \eqref{eq:modspace1} the improved regularity
\begin{equation}\label{eq:regularization}
q_{\ep,\delta}^w u \in C( \ro , M_{\sigma,s_0 + \alpha}(\rr d))
\end{equation} 
where 
\begin{equation*}
\alpha = 3 (1 + \sigma) - 2 k > 0
\end{equation*} 
due to the assumption \eqref{eq:kmcondition}.

Thus we regard $v = q_{\ep,\delta}^w u$ as solution of the non-homogeneous Cauchy problem 
\begin{equation}\label{eq:SchrodEqIter3}
\begin{cases} 
P v = f \\ 
v(0,\cdot) = v_0
\end{cases}
\end{equation}
where 
\begin{equation*}
P = \partial_t + i \left(  |x|^{2k} + (-\Delta)^m \right)
\end{equation*}
and $v_0 = q_{\ep,\delta}^w u_0 \in M_{\sigma, s_0 + \alpha}(\rr d)$
in view of \eqref{eq:schwartzmodspace1}. 
Concerning $f = f(t,\cdot)$ we will show that $f \in C( \ro , M_{\sigma,s_0 + \alpha}(\rr d))$. 

Since $P u = 0$ we have
\begin{equation*}
f = P v = P q_{\ep,\delta}^w u = P q_{\ep,\delta}^w u - q_{\ep,\delta}^w Pu =  [ i a^w, q_{\ep,\delta}^w ] u
\end{equation*}
where $a(x,\xi) = |x|^{2k} + |\xi|^{2m}$.
The symbol $a$ is a polynomial so it has only a finite number of nonzero derivatives. 
The expansion \eqref{eq:commutatorsymbol1} yields
\begin{equation}\label{eq:commutatorsymbol2}
b_{\ep,\delta} = i \left( a \wpr  q_{\ep,\delta} - q_{\ep,\delta} \wpr a \right) 
\sim \sum_{j=0}^{\max(k,m)-1} \frac{(-1)^{j}}{(2j + 1)! 2^{ 2 j }} \, \{ a, q_{\ep,\delta} \}_{2j+1}. 
\end{equation}
In fact if $j \geqs \max(k,m)$ then $2 j + 1 \geqs \max(2 k,2 m) + 1$
which implies that $\{ a, q_{\ep,\delta} \}_{2j+1} = 0$. 

By the construction of $q_{\ep,\delta}$, 
cf. \eqref{eq:qdef}, Lemma \ref{lem:symbolcutoff} and Remark \ref{rem:extension},
we have $q_{\ep,\delta} = g \circ a$ with $g \in C^\infty(\ro)$. 
Lemma \ref{lem:Poisson} thus implies that $\{ a, q_{\ep,\delta} \} = 0$
which in turn means that the expansion \eqref{eq:commutatorsymbol2}
starts at $j = 1$. 
Since $a \in G^{2k,\sigma}$ and $q_{\ep,\delta} \in G^{0,\sigma}$
it follows from the calculus that $b_{\ep,\delta} \in G^{2k - 3 (1 + \sigma), \sigma} = G^{-\alpha,\sigma}$. 
(This argument is similar to \cite[Lemme~7.1]{Helffer1}.)

From $u\in C( \ro , M_{\sigma,s_0}(\rr d))$ and Theorem \ref{thm:PseudoShubinSobolev} we thus obtain
\begin{equation*}
f = P v = b_{\ep,\delta}^w u \in C( \ro , M_{\sigma,s_0 + \alpha}(\rr d) ). 
\end{equation*}
Now \eqref{eq:SchrodEqIter3}, $v_0 \in M_{\sigma, s_0 + \alpha}(\rr d)$ and Theorem \ref{thm:wptheorem3}
show that \eqref{eq:regularization} holds true. 

In the second step we take $0 < \ep' < \delta' \leqs \ep$
which gives $\supp q_{\ep', \delta'} \cap \supp (1 - q_{\ep,\delta} ) = \emptyset$. 
By the calculus this yields $r = q_{\ep', \delta'} \wpr ( 1 - q_{\ep, \delta} ) \in \cS(\rr {2d})$ which combined with
\eqref{eq:schwartzmodspace1} and
Theorem \ref{thm:PseudoShubinSobolev} 
implies
\begin{equation}\label{eq:schwartzmodspace2}
q_{\ep', \delta'}^w u_0 = q_{\ep', \delta'}^w q_{\ep , \delta }^w u_0 + r^w u_0 \in M_{\sigma,s}(\rr d) \qquad \forall s \in \ro
\end{equation}
since $r^w: \cS' \to \cS$ is regularizing. 

Set $w = q_{\ep',\delta'}^w u$ and consider the non-homogeneous Cauchy problem 
\begin{equation}\label{eq:SchrodEqIter4}
\begin{cases} 
P w = g \\ 
w(0,\cdot) = w_0
\end{cases}
\end{equation}
where $w_0 = q_{\ep',\delta'}^w u_0 \in M_{\sigma, s_0 + 2 \alpha}(\rr d)$ by 
\eqref{eq:schwartzmodspace2}. 

As above we obtain using \eqref{eq:commutatorsymbol2}
\begin{equation*}
g = P q_{\ep',\delta'}^w u - q_{\ep',\delta'}^w Pu= 
b_{\ep',\delta'}^w u
= b_{\ep',\delta'}^w q_{\ep,\delta}^w u+ b_{\ep',\delta'}^w (1 - q_{\ep,\delta})^w u
\end{equation*}
with $b_{\ep',\delta'} \in G^{-\alpha,\sigma}$. 
Since $\supp b_{\ep', \delta'} \subseteq \wt \Sigma_{\delta'}$ we have $\supp b_{\ep', \delta'} \cap \supp (1 - q_{\ep,\delta} ) = \emptyset$
which implies that $b_{\ep',\delta'}^w (1 - q_{\ep,\delta})^w:  \cS' \to \cS$ is regularizing.  
Thus $g \in C( \ro , M_{\sigma,s_0 + 2 \alpha}(\rr d) )$
by \eqref{eq:regularization} and again Theorem \ref{thm:PseudoShubinSobolev}. 
Again \eqref{eq:SchrodEqIter4} and Theorem \ref{thm:wptheorem3} give 
$w = q_{\ep',\delta'}^w u \in C( \ro , M_{\sigma,s_0 + 2 \alpha}(\rr d) )$. 

The bootstrap argument shows that we may obtain
\begin{equation*}
q_{\ep',\delta'}^w u \in C( \ro , M_{\sigma,s_0 + k \alpha}(\rr d) )
\end{equation*}
for any $k \in \no$, in each step decreasing $0 < \ep' < \delta' \leqs \ep$ with $\ep > 0$ coming from the preceding step. 
We may do this keeping $\ep' > \delta_0 > 0$ for a fixed $\delta_0 > 0$. 
If $0 < \ep_0 < \delta_0$ then for any $t \in \ro$
\begin{equation*}
q_{\ep_0,\delta_0}^w u(t,\cdot) \in M_{\sigma,s}(\rr d)
\end{equation*}
for any $s \in \ro$. 
Thus $q_{\ep_0,\delta_0}^w u(t,\cdot) \in \cS(\rr d)$ for any $t \in \ro$. 
Since $q_{\ep_0,\delta_0}$ is elliptic on $\wt \Sigma_{\ep_0}$
this proves that $u(t,\cdot)$ is smooth in $\wt \Sigma$ for any $t \in \ro$. 
It follows that $\wt \Lambda \in \mathcal F(u (t,\cdot))$ so we have shown
$\mathcal F(u_0) \subseteq \mathcal F(u (t,\cdot))$ for any $t \in \ro$. 
The opposite inclusion follows from 
$\cK_t^{-1} = \cK_{-t}$. 
\qed

\section*{Acknowledgments}
The first author is partially supported by the INDAM-GNAMPA project CUP E53C23001670001.
This work is partially supported by the MIUR project ``Dipartimenti di Eccellenza 2018-2022'' (CUP E11G18000350001).

\bibliographystyle{amsplain}

\end{document}